\documentclass{amsart}
\usepackage[margin=1in]{geometry} 
\usepackage{lipsum}
\usepackage{tikz-cd}
\usepackage[unicode]{hyperref}
\usepackage{graphicx,color}
\usepackage{amssymb,amsmath}
\usepackage{mathtools}
\usepackage{mathrsfs}
\usepackage{enumerate}
\usepackage{ragged2e}
\usepackage{enumitem,pifont}
\usepackage{tikz}
\usepackage{amsmath}
\usetikzlibrary{arrows,decorations.pathmorphing,automata,backgrounds}
\usetikzlibrary{backgrounds,positioning}
\usetikzlibrary{braids}

\usepackage{caption}
\usepackage{subcaption}
\usepackage{float}

\usepackage{egothic}
\usepackage[T1]{fontenc}
\usepackage{yfonts}

\graphicspath{{Figures/}}

\usepackage{calc}

\usepackage{dbnsymb}

\definecolor{ao}{rgb}{0.0, 0.5, 0.0}

\newcommand{\R}{\mathbb{R}}
\newcommand{\Z}{\mathbb{Z}}

\newcommand{\Q}{\mathbb{Q}}

\newcommand{\V}{\mathsf{V}}

\newcommand{\ba}{\mathbf{a}}

\newcommand{\bD}{\mathbf{D}}

\newcommand{\bT}{\mathbf{T}}

\newcommand{\sE}{\mathsf{E}}

\newcommand{\WD}{\mathfrak{D}}

\newcommand{\tin}{\textnormal{in}}
\newcommand{\tout}{\textnormal{out}}
\newcommand{\trace}{\textrm{tr}}

\newcommand{\calA}{\mathcal{A}}
\newcommand{\calB}{\mathcal{B}}
\newcommand{\calI}{\mathcal{I}}

\newcommand{\calG}{\mathscr{G}}
\newcommand{\calS}{\mathcal{S}}
\newcommand{\calP}{\mathscr{P}}
\newcommand{\id}{\operatorname{id}}

\newcommand{\Hom}{\operatorname{Hom}}

\newcommand{\Aut}{\operatorname{Aut}}
\newcommand{\TAut}{\operatorname{TAut}}

\newcommand{\divcc}{\operatorname{div}}
\newcommand{\gr}{\operatorname{gr}}

\newcommand{\rp}{\mathring{\rho}}

\newcounter{dummy} \numberwithin{dummy}{section}
 \newtheorem{theorem}[dummy]{Theorem}
\newtheorem{thm}[dummy]{Theorem}
\newtheorem{prop}[dummy]{Proposition}
\newtheorem{lemma}[dummy]{Lemma}
\newtheorem{cor}[dummy]{Corollary}

\newtheorem*{thm*}{Theorem}
\newtheorem*{prop*}{Proposition}

\theoremstyle{definition}
\newtheorem{definition}[dummy]{Definition}
\newtheorem{cons}[dummy]{Construction}
\newtheorem{example}[dummy]{Example}
\newtheorem{remark}[dummy]{Remark}
\newtheorem{notation}[dummy]{Notation}
\numberwithin{equation}{section}

\usepackage{multicol}

\newcommand{\gt}{\mathsf{GT}}
\newcommand{\grt}{\mathsf{GRT}}
\newcommand{\kv}{\mathsf{KV}}
\newcommand{\krv}{\mathsf{KRV}}

\newcommand{\PaCD}{\mathsf{PaCD}}
\newcommand{\PaB}{\mathsf{PaB}}

\newcommand{\lie}{\widehat{\mathfrak{lie}}}
\newcommand{\tder}{\mathfrak{tder}}

\newcommand{\cyc}{\mathrm{cyc}}
\newcommand{\ass}{\mathfrak{ass}}

\newcommand{\ucrossing}{\raisebox{-.5mm}{
\begin{tikzpicture}[scale=.15]
\draw[thick, ->](1,-1)--(-1.1,1.1);
\draw[thick](-1,-1)--(-.2,-.2);
\draw[thick] (.2,.2)--(.95,.95);
\draw[ thick,->] (.2,.2)--(1.1,1.1);
\draw[line width= 2 mm, white](.4,-.4)--(-.4,.4);
\draw[thick](1,-1)--(-.95,.95);
\end{tikzpicture}}
}

\newcommand{\ocrossing}{\raisebox{-.5mm}{
\begin{tikzpicture}[scale=.15]
\draw[ thick,->](-1,-1)--(1.1,1.1);
\draw[thick](1,-1)--(.2,-.2);
\draw[thick] (-.2,.2)--(-.95,.95);
\draw[ thick,->] (-.2,.2)--(-1.1,1.1);
\draw[line width= 2 mm, white](-.4,-.4)--(.4,.4);
\draw[thick](-1,-1)--(.95,.95);
\end{tikzpicture}}
}

\newcommand{\pv}{\raisebox{-.5mm}{
\begin{tikzpicture}[scale=.15]
\draw[red](-.09,-.09)--(-1,-1);
\draw[ thick,->](0,-.08)--(0,1.1);
\draw[thick](1, -1)--(0,0)--(0,.93);
\end{tikzpicture}}
}

\newcommand{\negvertex}{\raisebox{-.5mm}{
\begin{tikzpicture}[scale=.15]
\draw[thick](0,-1)--(0,0)--(-.95, .95);
\draw[thick,->](0,0)--(-1.1, 1.1);
\draw[thick](.5,.5)--(.95,.95);
\draw[ thick, ->](.5,.5)--(1.1,1.1);
\end{tikzpicture}}
}

\newcommand{\vertex}{\raisebox{-.5mm}{
\begin{tikzpicture}[scale=.15]
\draw[thick, ->](0,-.09)--(0,1.1);
\draw[thick](1, -1)--(.5,-.5);
\draw[thick](0,.9)--(0,0)--(-1,-1);
\end{tikzpicture}}
}

\newsavebox\cvtikzbox
\sbox\cvtikzbox{{\raisebox{-.5mm}{
\begin{tikzpicture}[scale=.15]
\draw[thick](0,-.09)--(0,1.1);
\draw[thick](1, -1)--(.25,-.25);
\draw[thick](0,.9)--(0,0)--(-1,-1);
\filldraw[fill=black, draw=black] (1,-1) circle (3mm) ;
\end{tikzpicture}}
}}

\newcommand{\cvtikz}{\usebox\cvtikzbox}

\newcommand{\ppv}{\raisebox{-.5mm}{
\begin{tikzpicture}[scale=.15]
\draw[red,->](0,0)--(0,1);
\draw[red](1, -1)--(0,0);
\draw[red](0,0)--(-1,-1);
\end{tikzpicture}}
}

\newcommand{\rarrow}{\raisebox{-1mm}{
\begin{tikzpicture}[scale=.2]
\draw[ thick](-1,-1)--(-1,1);
\draw[thick, ->](-1,-.09)--(-1,1.1);
\draw[ thick](1.5,-1)--(1.5,1);
\draw[thick, ->](1.5,-.09)--(1.5,1.1);
\draw[thick, dotted] (-1,0)--(1.5,0);
\draw[->] (1.49,0)--(1.5,0);
\end{tikzpicture}}
}

\newcommand{\chord}{\raisebox{-.5mm}{
\begin{tikzpicture}[scale=.15]
\draw[thick](-1,-1)--(-1,1);
\draw[-](-1,-.09)--(-1,1.1);
\draw[thick](1.5,-1)--(1.5,1);
\draw[thick, -](1.5,-.09)--(1.5,1.1);
\draw[thick,blue,dotted] (-1,0)--(1.5,0);
\draw[-] (1.49,0)--(1.5,0);
\end{tikzpicture}}
}

\newcommand{\rightarrowp}{\raisebox{-1mm}{
\begin{tikzpicture}[scale=.2]
\draw[thick,](-1,-1)--(-1,1);
\draw[thick, ->](-1,-.09)--(-1,1.1);
\draw[red,->](1.5,-1)--(1.5,1);
\draw[thick, dotted,] (-1,0)--(1.5,0);
\draw[->] (1.49,0)--(1.5,0);
\end{tikzpicture}}
}

\newcommand{\tcap}{\raisebox{-.5mm}{
\begin{tikzpicture}[scale=.15]
\draw[thick](0,-1)--(0,1);
\filldraw[fill=black, draw=black] (0,1) circle (3mm) ;
\end{tikzpicture}}
}

\newcommand{\bcap}{\raisebox{-.5mm}{
\begin{tikzpicture}[scale=.15]
\draw[thick](0,-1)--(0,1);
\filldraw[fill=black, draw=black] (0,-1) circle (3mm) ;
\end{tikzpicture}}
}

\newcommand{\pcv}{\raisebox{-1mm}{
\begin{tikzpicture}[scale=.2]
\draw[red](-.09,-.09)--(-1,-1);
\draw[thick](0,-.08)--(0,1.1);
\draw[thick,->](0,0)--(0,1);
\draw[thick](1, -1)--(0,0)--(0,.93);
\filldraw[fill=black, draw=black] (1,-1) circle (3mm) ;
\end{tikzpicture}}
}

\newcommand{\wf}{\mathsf{wF}}
\newcommand{\hatwf}{\widehat{\wf}}
\newcommand{\arrows}{\mathscr{A}}

\newcommand{\A}{\arrows}

\newcommand{\apr}{\overrightarrow{pr}}

\newcommand{\im}{\operatorname{im}}

\title[Comparing $\grt$ and $\krv$]{A knot-theoretic approach to comparing the Grothendieck-Teichm\"{u}ller and Kashiwara-Vergne groups}

\author[Z. Dancso]{Zsuzsanna Dancso}
\address{School of Mathematics and Statistics\\ The University of Sydney\\ Sydney, NSW, Australia}
\email{zsuzsanna.dancso@sydney.edu.au}

\author[T. Hogan]{Tamara Hogan}
\address{School of Mathematics and Statistics \\ The University of Melbourne \\ Melbourne, Victoria, Australia}
\email{hogant@student.unimelb.edu.au}

\author[M. Robertson]{Marcy Robertson}
\address{School of Mathematics and Statistics \\ The University of Melbourne \\ Melbourne, Victoria, Australia}
\email{marcy.robertson@unimelb.edu.au}

\begin{document}
\maketitle

\begin{abstract} 
Homomorphic expansions are combinatorial invariants of knotted objects, which are universal in the sense that all finite-type (Vassiliev) invariants factor through them. Homomorphic expansions are also important as bridging objects between low-dimensional topology and quantum algebra. For example, homomorphic expansions of parenthesised braids are in one-to-one correspondence with Drinfel'd associators (Bar-Natan 1998), and homomorphic expansions of $w$-foams are in one-to-one correspondence with solutions to the Kashiwara-Vergne (KV) equations (Bar-Natan and the first author, 2017). The sets of Drinfel'd associators and KV solutions are both bi-torsors, with actions by the pro-unipotent Grothendieck-Teichm\"{u}ller and Kashiwara-Vergne groups, respectively. The above correspondences are in fact maps of bi-torsors (Bar-Natan 1998, and the first and third authors with Halacheva 2022).

There is a deep relationship between Drinfel'd associators and KV equations—discovered by Alekseev, Enriquez and Torossian in the 2010s—including an explicit formula constructing KV solutions in terms of associators, and an injective map $\rho:\grt_1 \to \krv$.
This paper is a topological/diagrammatic study of the image of the Grothendieck-Teichm\"{u}ller groups in the Kashiwara-Vergne symmetry groups, using the fact that both parenthesised braids and $w$-foams admit respective finite presentations as an operad and as a tensor category (circuit algebra or prop). 
\end{abstract}

\tableofcontents

\section{Introduction}

Given a class of knotted objects, $\mathcal{K}$, an \emph{expansion}, or {\em universal finite type invariant}, is an assignment $Z:\widehat{\mathcal{K}}\rightarrow \mathcal{A}$ taking 
values in a vector space, $\mathcal{A}$, which is graded with 
respect to a \emph{Vassiliev filtration} of $\mathcal{K}$, and satisfies a universality condition (Definition~\ref{def: expansion}). Expansions are called \emph{homomorphic} if they respect any operations defined on $\mathcal{K}$, such as braid composition, knot connected sum or strand doubling. One of the best known examples of a homomorphic expansion is the Kontsevich integral (\cite{MR1237836}, \cite{Invariants_Book}, \cite{Dancso_Kont}) which is universal in the sense that any finite type invariant of oriented links factors through it.

In practice, constructing homomorphic expansions is difficult and, in many cases (e.g. for the planar algebra of tangles), they do not exist. The construction can be simplified, however, when $\mathcal{K}$ is finitely presented as an algebraic structure. Examples of this include parenthesised braids viewed as an operad or prop (Section~\ref{sec: braids};\cite{BNGT}), welded tangles as a tensor category (Section~\ref{sec:foams};\cite[Section 2]{BND:WKO1}) and so forth. When this is the case, the Vassiliev filtration of $\mathcal{K}$ coincides with the $I$-adic filtration by powers of the augmentation ideal, and finding a homomorphic expansion is equivalent to solving a set of equations in the associated graded space $\mathcal{A}$, which admits a diagrammatic description as vector spaces spanned by Jacobi or Feynman diagrams (Section~\ref{section:PACD} and Definition~\ref{def:wJacobiDiag}).

\medskip 
Homomorphic expansions also have deep connections to quantum algebra. For example, using the description of parenthesised braids as a tensor category, Bar-Natan \cite{BNGT} showed that there is a bijection between Drinfel'd associators and homomorphic expansions of parenthesised braids. Similarly, work of Bar-Natan and the first author showed that homomorphic expansions of a four-dimensional class of knotted objects called {\em welded foams} are in bijection with solutions to the {\em Kashiwara-Vergne equations} (see Section~\ref{sec:foams} or \cite[Theorem 4.11]{BND:WKO2}).

This paper is part of an ongoing study of the relationship between the homomorphic expansions of parenthesised braids and welded foams, informed by equivalent phenomena in quantum algebra. Informally, the Kashiwara-Vergne (KV) problem (\cite{KV78}) asks when one can simplify the expression:
\begin{align}
\label{eq:bchdef}
    e^xe^y = e^{x+y + \frac{1}{2}[x,y] + \frac{1}{12}[x,[x,y]] + \dots},
\end{align} 
where $[x,y]=xy-yx$. The Lie series in the exponent is called the Baker-Campbell-Hausdorff series \cite{HallLie} of $x$ and $y$, and is denoted by $\mathfrak{bch}(x,y)$. A solution to the Kashiwara-Vergne equations allows one to replace $\mathfrak{bch}(x,y)$ with a convergent power series and adjoint operations~\cite{KV78}. In practice, this problem transforms into a question about when \emph{tangential}, or basis-conjugating, automorphisms of free Lie algebras (Section~\ref{sec:tder}) satisfy two equations (\cite[Section 5.3]{AT12}). 

Drinfel'd observed in \cite[Section 4]{Drin90} that elements of the Lie algebras of infinitesimal braids (also known as Drinfel'd-Kohno Lie algebras, Section~\ref{section:PACD}) induce tangential derivations of the free Lie algebras (e.g. Example~\ref{example: t}; Proposition 3.11 of \cite{AT12}). In a series of papers 
(\cite{AT12}, \cite{ATE10}), Alekseev, Enriquez and Torossian exploited this observation to show that   
solutions of the Kashiwara-Vergne equations can be explicitly constructed from Drinfel'd associators.

In \cite[Section 4]{BND:WKO2}, Bar-Natan and the first author showed that KV solutions are in bijection with homomorphic expansions of welded foams.
The explicit construction of KV solutions from Drinfel'd associators can be carried out on the topological side by explicitly constructing homomorphic expansions of welded foams from expansions of parenthesised braids (\cite[Theorem 1.1]{BND:WKO3}, Section~\ref{sec:PaCDtoA}).


\medskip

The relationship between Drinfel'd associators and KV solutions described by Alekseev, Enriquez and Torossian extends to an intricate 
relationship between the symmetry groups of these objects: the prounipotent Grothendieck--Teichm\"{u}ller groups 
($\gt$ and $\grt_1$) and Kashiwara-Vergne groups ($\kv$ and $\krv$), respectively. The sets of 
Drinfel'd associators and KV solutions are bi-torsors under the actions of these symmetry groups (Section~\ref{section:PACD} and Section~\ref{sec:KV groups}, respectively). In particular, Alekseev and Torossian \cite[Theorem 4.6]{AT12} describe a map  $\rho: \grt_1 \to \krv$. It is a still open conjecture to show that this map is nearly an isomorphism, i.e. that $\krv \cong \grt_1\times \mathbb{K}$ (\cite[Section 4]{AT12}).

Several authors have demonstrated that the Grothendieck--Teichm\"{u}ller groups, $\gt$ and $\grt_1$ are isomorphic to automorphism groups of (prounipotently completed) parenthesised braids and their associated graded, 
parenthesised \emph{chord diagrams} (\cite{BNGT,tamarkin1998proof, FresseVol1, Will_GT}): 
\[\Aut(\widehat{\PaB}) \cong \gt \quad \text{and} \quad \Aut(\PaCD)\stackrel{\tau}{\cong} \grt_1.\] Subsequently, 
the first and last authors, with Halacheva \cite{DHR21}, identified the Kashiwara-Vergne groups, $\kv$ and $\krv$, as automorphisms of (prounipotently completed) welded foams and their associated graded prop \cite{DHR1} of {\em 
arrow diagrams}, \[\Aut_{v}(\widehat{\wf}) \cong \kv \quad \text{and} \quad \Aut_{v}(\arrows)\stackrel{\Theta}{\cong} \krv.\]  This paper builds on these results and those of \cite{BND:WKO2, DHR1, DHR21} to construct an explicit topologically-based interpretation of the Alekseev-Torossian (AT) map (\cite[Section 4]{AT12}; Construction~\ref{cons:buckle})
\[\rho:\grt_1 \to \krv.\] Explicitly, we construct a map $\rp:\Aut(\PaCD)\rightarrow\Aut_v(\arrows)$ which agrees with the Alekseev-Torossian map: 

\begin{thm*}[Theorem~\ref{thm:main} and Corollary~\ref{rp_grouphomo}]
The following is a commutative diagram of prounipotent groups: 
\[\begin{tikzcd}
	\Aut({\PaCD}) && \Aut_v(\arrows) \\
	\\
	\grt_1 && \krv.
	\arrow["\rp", from=1-1, to=1-3]
	\arrow["\rho"', from=3-1, to=3-3]
	\arrow["\tau"', "\cong", swap, from=1-1, to=3-1]
	\arrow["\Theta", swap, "\cong"', from=1-3, to=3-3]
\end{tikzcd}\] 
\end{thm*}

In particular, we identify a subgroup of arrow diagram automorphisms isomorphic to $\grt_1$. This is a further step towards understanding the relationship between automorphisms of chord diagrams and arrow diagrams. Moreover, these automorphisms preserve the homomorphic expansions of $w$-foams which come from parenthesised braids. 

\subsection{Structure and Conventions}
Throughout this paper we assume that the reader is familiar with basic notations of category theory, in the form of linear tensor categories, as well as some basics around the theory of (coloured) operads. Some review of these details are provided in Appendix~\ref{sec: circuit algebras} as well as some basic definitions of (wheeled) props.  

Sections \ref{sec: braids} and \ref{section:PACD} provide an expository review of homomorphic expansions of \emph{parenthesised braids}, \[\begin{tikzcd}\widehat{\mathbb{Q}[\PaB]}\arrow[r, "Z"] & \PaCD, \end{tikzcd}\] and their relationship to the graded Grothendeick-Teichm\"{u}ller group, $\grt_1$. Throughout, we use the fact that the class of parenthesised braids, $\PaB$ and their associated graded {parenthesised chord diagrams} $\PaCD$, can be finitely presented as \emph{tensor categories} or \emph{props} as in \cite{BNGT}, \cite{Joyal_Street}, \cite{ATE10} and \cite{tamarkin1998proof}.  Readers conversant with this material may be more familiar with the presentation of parenthesised braids and chord diagrams as operads, as in \cite{FresseVol1}, \cite{merkGT}, etc., but the language of props leads to a more natural discussion of relationships between parenthesised braids and welded foams. 

Section~\ref{sec:foams} introduces a class of four-dimensional tangles called \emph{welded foams}. Welded foams, hereafter called $w$-\emph{foams}, are a simultaneous generalisation of $w$-braids (aka loop braids, \cite{FRRwBraids}, \cite{Damiani_Loop_braid,Baez_loop_braid_groups}) and classical tangles. Topologically, they are knotted ribbon tubes and ribbon tori in $\mathbb{R}^{4}$. In \cite{BND:WKO2}, Bar-Natan and the first author introduced the notion of a \emph{circuit algebra} to encode the algebraic structure of $w$-foams and construct their homomorphic expansions. Later, in a paper with Halacheva, the first and third author showed that circuit algebras are themselves a type of tensor category called a \emph{wheeled prop} \cite{DHR1}. This algebraic structure is described in full detail in \cite[Sections 2 and 4]{DHR1}, and also in abbreviated form in \cite{BND:WKO2, DHR21}. While it provides foundations for this paper, it does not play a major role in the proof of our main theorem, and thus we push the details to Appendix~\ref{sec: circuit algebras}. 

The construction of the map $\rp:\Aut(\PaCD)\rightarrow\Aut_v(\arrows)$ appears in Section~\ref{sec:map}. We then go through the technical proof that this map is an injective group homomorphism and agrees with the Alekseev-Torossian map $\rho:\grt_{1}\rightarrow \krv$. In particular, we conclude that $\Aut(\PaCD)$ can be seen as a subgroup of $\Aut_{v}(\arrows)$. Appendix~\ref{sec:DirectProof} gives an alternative, more constructive, proof of the same result.

\subsection{Acknowledgements}
This paper was prepared in honour of Ezra Getzler on the occasion of his 60th birthday.  Much of this work follows his influential work using operads and props to study geometric problems and quantum algebra.

\medskip

The third author acknowledges the support of Australian Research Council Future Fellowship FT210100256. 

\section{Preliminaries: Topology, Lie algebras, and Expansions}\label{sec: expansions}
In this section we give a short review of how invariants called {\em homomorphic expansions} provide a bridge between knot theory, quantum algebra and Lie theory.  Expansions in knot theory are inspired by an idea from group theory, which aims to study a group ring through its associated graded ring. More explicitly, given an abstract group $G$, the complete \emph{associated graded algebra} of the group ring $\mathbb Q[G]$ is the $\mathbb{Q}$-algebra $$A(G) :=\text{gr}(\mathbb{Q}[G]) = \prod_{m\geq 0} \calI^m/\calI^{m+1},$$ where $\calI=\text{ker}\left(\mathbb{Q}[G]\xrightarrow{\epsilon} \mathbb{Q}\right)$ is the augmentation ideal. 
\begin{definition}
An \emph{expansion} of the group $G$ is a filtered algebra isomorphism \[Z : \widehat{\mathbb{Q}[G]} \rightarrow A(G).\] 
\end{definition}
Here the hat notation denotes the pro-unipotent completion of the group-ring $\mathbb{Q}[G]$. An equivalent statement is to say that the induced map $$\text{gr}(Z): \left(\text{gr}(\widehat{\mathbb{Q}[G]} )= A(G)\right) \rightarrow \left(\gr (A(G)) = A(G)\right)$$ is the identity of $A(G)$.  The latter condition is saying that the degree $m$ piece of $Z$, restricted to $\calI^m$, is the projection onto $\calI^m/\calI^{m+1}$. 

Replacing groups by classes of knotted objects, this idea extends to provide useful knot invariants. This is particularly true when the class of knotted objects $\mathcal K$ can be finitely presented as a tensor category. Examples include braids, viewed as a prop \cite{Joyal_Street, BNGT}, (virtual) tangles, viewed as a ribbon category \cite{Brochier, Joyal_Street, Turaev, FY}, and welded tangles, viewed as a wheeled prop \cite{BND:WKO2,DHR21}. In the last example, the topological analogue of the group ring, $\Q \mathcal K$, is the linear tensor category whose morphism spaces are vector spaces spanned by elements of $\mathcal K$. The category $\Q\mathcal{K}$ is filtered by powers of the augmentation ideal $\calI=\{\sum c_iK_i | c_i\in \Q, K_i \in \mathcal{K}, \sum c_i=0 \}$. As with the examples from group theory, the complete associated graded structure is the direct product $\mathcal A:= \prod_{m\geq 0} \calI^m /\calI^{m+1}$. Further details on this type of filtration can be found in \cite[Section 2.1.1]{BNGT} or \cite[Section 2.3]{DHR21}.

\begin{definition}\cite[Definition 2.5]{BND:WKO2} \label{def: expansion}
Let $\mathcal{K}$ denote a class of knotted objects finitely presented as a tensor category (possibly with extra structure). A {\em homomorphic expansion} of $\mathcal{K}$ is a filtered linear isomorphism\footnote{\cite[Definition 2.5]{BND:WKO2} uses the equivalent statement ``$\gr(Z)=\id_{\mathcal{A}}$'', which avoids the pro-unipotent completion on the left.} $Z: \widehat{\Q\mathcal{K}} \to \mathcal{A}$ which intertwines all operations on $\mathcal{K}$ with their associated graded counterparts on $\mathcal{A}$.
\end{definition}

When a class of knotted objects $\mathcal K$ is finitely presented as a tensor category, defining a homomorphic expansion reduces to defining a functor at the values of the generators of $\mathcal K$ in $\mathcal A$. These values are subject to the relations of $\mathcal K$, which define a set of equations in $\mathcal A$. This seemingly simple exercise can build powerful bridges between knot invariants and algebra. We review two examples of interest:
\begin{enumerate}
\item When $\mathcal{K}$ is {\em parenthesised braids}, the equations which determine a homomorphic expansion turn out to be the {\em pentagon} and {\em hexagon} equations which define {\em Drinfel'd associators}. Hence, homomorphic expansions of parenthesised braids are in bijection with Drinfel'd associators (Theorem~\ref{thm:BN_Ass}). This bijection turns out to be an isomorphism of torsors: the expansion-preserving automorphism groups of (completed) parenthesised braids and their associated graded structure are identified with the symmetry groups of Drinfel'd associators, the Grothendieck-Teichm\"{u}ller groups (\cite{BN:NAT,BNGT, tamarkin1998proof, FresseVol1}, etc.). 

\item When $\mathcal K$ is the class of {\em welded foams}, the equations which determine homomorphic expansions are the {\em Kashiwara-Vergne} equations of Lie theory (Section~\ref{sec:FoamExpansions}). Thus, homomorphic expansions for welded foams are in bijection with Kashiwara-Vergne solutions (\cite[Theorem 4.9]{BND:WKO2}). This bijection can also be promoted to an isomorphism of torsors: the expansion preserving automorphism groups of (completed) welded foams, and their associated graded structure, are isomorphic to the Kashiwara-Vergne symmetry groups \cite[Theorems 5.12 and 5.16]{DHR21}.
\end{enumerate}

This preliminary section is dedicated to a review of these two examples.

\subsection{Braids}\label{sec: braids}
The braid group on $n$-strands is the fundamental group of the space of unordered configurations of $n$ points in the complex plane. By Alexander's Theorem, every link can be represented as a braid closure, and thus braids are an important tool in the study of knots and links. The braid group has a finite presentation 
\[\mathsf{B}_n=\left<\beta_1,\ldots,\beta_{n-1}\mid 
\beta_i\beta_j=\beta_j\beta_i \ \text{when} \ |i-j|\geq 2; \  
\beta_i\beta_{i+1}\beta_i =\beta_{i+1}\beta_i\beta_{i+1}\right>.\] where $\beta_i$ is braid on $n$-strands which crosses the $i$-th strand over the $(i+1)$-th strand. 
For any $n\geq 1$, there is a homomorphism to the symmetric group \begin{equation*}\label{projection_to_Sigma}
    \begin{tikzcd}\mathsf{B}_n \arrow[r, "\pi"] & \mathcal{S}_n \end{tikzcd}
\end{equation*} where the map $\pi$ sends the braid generator $\beta_i$ to the transposition $(i, i+1)$ in the symmetric group $\mathcal{S}_n$. 

The category of \emph{parenthesised braids} with $n$-strands, $\PaB(n)$, is a category whose objects are the set of \emph{parenthesised} permutations in $n$-letters and whose morphisms are braids on $n$-strands.  Composition of morphisms in this category is given by ``stacking'' braids with compatible parenthesisations, as pictured in Figure~\ref{fig:pabcomp}.  Every morphism $\beta$ in $\PaB(n)$ has an underlying \emph{skeleton}: this is the data of the source and target parenthesised permutations (which also specifies the underlying permutation $\pi(\beta)$ of $\beta$). For each skeleton $s$ we write $\PaB(s)$ for the set of braids with skeleton $s$.

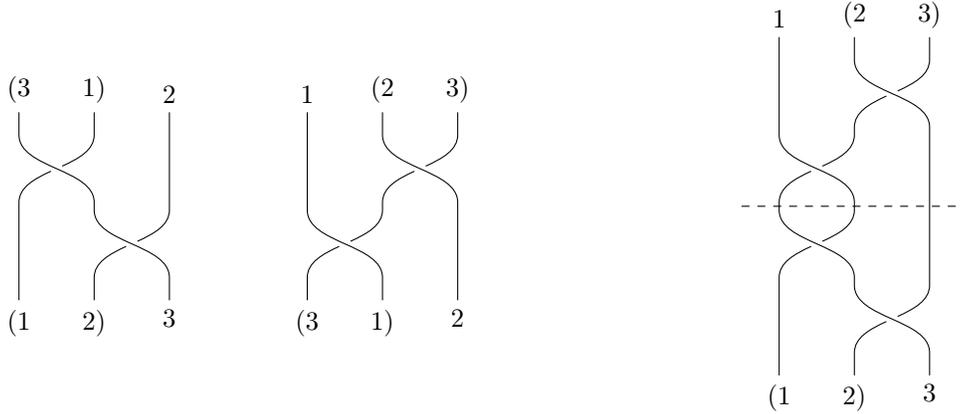
\begin{figure}[h]
    \centering
    \begin{subfigure}[c]{0.3\textwidth}
    \centering
    \begin{tikzpicture}
    \pic[name prefix = braid] at (0,0) {braid={s_1 s_2}};
    \node[at=(braid-1-s),above] {(3};
    \node[at=(braid-2-s),above] {1)};
    \node[at=(braid-3-s),above] {2};
    \node[at=(braid-1-e),below] {3};
    \node[at=(braid-2-e),below] {(1};
    \node[at=(braid-3-e),below] {2)};
    \end{tikzpicture}
    \end{subfigure}
        \centering
    \begin{subfigure}[c]{0.3\textwidth}
        \begin{tikzpicture}
    \pic[name prefix = braid] at (0,0) {braid={s_2 s_1}};
    \node[at=(braid-1-s),above] {1};
    \node[at=(braid-2-s),above] {(2};
    \node[at=(braid-3-s),above] {3)};
    \node[at=(braid-1-e),below] {1)};
    \node[at=(braid-2-e),below] {2};
    \node[at=(braid-3-e),below] {(3};
    \end{tikzpicture}
    \end{subfigure}
    \begin{subfigure}[c]{0.3\textwidth}
    \centering
    \begin{tikzpicture}
    \pic[braid/every floor/.style={draw=black,dashed},
    name prefix = braid] at (0,0) {braid={s_2 s_1 s_1 s_2}};
    \node[at=(braid-1-s),above] {1};
    \node[at=(braid-2-s),above] {(2};
    \node[at=(braid-3-s),above] {3)};
    \node[at=(braid-1-e),below] {(1};
    \node[at=(braid-2-e),below] {2)};
    \node[at=(braid-3-e),below] {3};
    \draw[dashed] (-0.5,-2.25)--(2.5,-2.25); 
    \end{tikzpicture}
    \end{subfigure}
    \caption{Two elements of $\PaB(3)$ and their stacking composition. We draw all diagrams in this paper as going from bottom to top, e.g. the braid on the far left goes from $(12)3$ to $(31)2$.}
    \label{fig:pabcomp}
\end{figure}

The collection of all parenthesised braids is a freely generated symmetric monoidal category, where the tensor product functor $\PaB(s)\otimes \PaB(t)\rightarrow \PaB(s \cup t)$ is given by the concatenation of parenthesised permutations $s$ and $t$ (e.g. \cite[Example 1, pg 55]{Joyal_Street}). On morphisms, this tensor product is the juxtaposition of braids — visually, placing the parenthesised braids next to each other. One often abuses notation and writes \[\PaB = \coprod_{s\in\mathcal{S}}\PaB(s)\] to describe this phenomenon.  

The category of parenthesised braids is \emph{finitely presented} in the sense that every morphism can be written as a combination of tensors, categorical compositions, strand deletions and strand doubling of two generators: \[\overcrossing\in \Hom_{\mathsf{PaB}(2)}((12), (21)) \quad \text{and} \quad \Associator\in \Hom_{\mathsf{PaB}(3)}((12)3,1(23)).\] These generating morphisms satisfy the \emph{pentagon} and  two \emph{hexagon} equations of braided monoidal categories. For further details on this categorical structure, we recommend \cite[Section 2.1, Claim 2.6]{BNGT} or \cite[Example 1, pg 56]{Joyal_Street}.

\begin{remark}
Many readers of this article will be familiar with the fact that the collection of parenthesised braids $\{\PaB(n)\}$, $n\geq 1$, is a groupoid model for the operad of little $2$-discs (see, for example, \cite[Chapter 6]{FresseVol1}, \cite{merkGT}). The tensor category described here is the prop associated to this operad. We have chosen this presentation because it allows for a more natural discussion of expansions of both braids and $w$-foams in the same language. 
\end{remark}

\subsection{The associated graded structure: the infinitesimal braid algebra}\label{section:PACD}
The analogue of the group ring for parenthesised braids is the linear extension of $\PaB$, $$\mathbb{Q}[\PaB]= \bigsqcup_{s\in\mathcal{S}}\mathbb{Q}[\PaB](s).$$ Here $\mathbb{Q}[\PaB](s)$ is a vector space spanned by parenthesised braids with skeleton $s$, i.e. elements are linear combinations of morphisms $B = a_1\beta_1+a_2\beta_2 + \ldots$ where $\beta_i$ is a braid with skeleton $s$ and $a_i\in\mathbb{Q}$. 

The complete associated graded space of $\PaB$ is the direct product of the successive quotients of  the powers of the augmentation ideal \[\calA:= \Pi _{n=0}^\infty \calI^n/\calI^{n+1}.\] The associated graded space of $\PaB$ has several (isomorphic) models, the most common of which is the tensor category $\PaCD$, whose morphism spaces, $$\PaCD:=\bigsqcup_{s\in\mathcal{S}}\PaCD(s),$$ are completed graded vector spaces spanned by parenthesised \emph{chord diagrams}.

Chord diagrams are most concisely defined as elements of the universal enveloping algebras of the {\em infinitesimal braid} Lie algebras. The Lie algebra of infinitesimal braids on $n$ strands, $\mathfrak{t}_n$, $n\geq 2$, is the quotient of the free Lie algebra $\lie_{\frac{n(n-1)}{2}}$ generated by the symbols $\{t^{i,j} = t^{j,i}\}_{1\leq i< j\leq n}$ subject to the relations \[ [t^{i,j} , t^{k,l}] = 0  \quad \text{and} \quad[t^{i,j} , t^{i,k} + t^{k,j}] = 0 \] whenever $i,j,k,l$ are distinct. The completed universal enveloping algebra, $\widehat{U}(\mathfrak{t}_n)$, can be viewed as a category of (non-parenthesised) \emph{chord diagrams} on $n$ strands, $\mathsf{CD}(n)$. This category has one object and is enriched in (completed, filtered, co-associative) co-algebras (e.g. \cite[Section 6.4]{merkGT}, \cite[Chapter 10.2]{FresseVol1}). Morphisms of $\mathsf{CD}(n)$ are depicted as chord diagrams on $n$ vertical strands, where $t^{i,j}$ is represented by a chord between strands $i$ and $j$. Composition of morphisms is depicted by ``stacking'' the chord diagrams. The category of \emph{parenthesised chord diagrams} on $n$-strands, $\mathsf{PaCD}(n)$, is obtained by formally replacing the single object of the category $\mathsf{CD}(n)$ with the set of parenthesised words of length $n$ with distinct letters. A parenthesised chord diagram is shown on the right-hand side of Figure~\ref{fig:pacdexample}.

An alternative combinatorial description of morphisms in $\PaCD$, which has many computational advantages, is as a vector space of {\em Jacobi diagrams}. Roughly speaking, a Jacobi diagram is a graph with trivalent and univalent vertices, the latter of which are arranged along the braid skeleton, as on the left side of Figure~\ref{fig:pacdexample}.These diagrams are subject to a set of relations (STU, IHX and AS), which encode the familiar relations of Lie algebras. For a more comprehensive introduction, we recommend the survey article \cite{BN_Survey_Knot_Invariants} or the book \cite{Invariants_Book} for further details. 

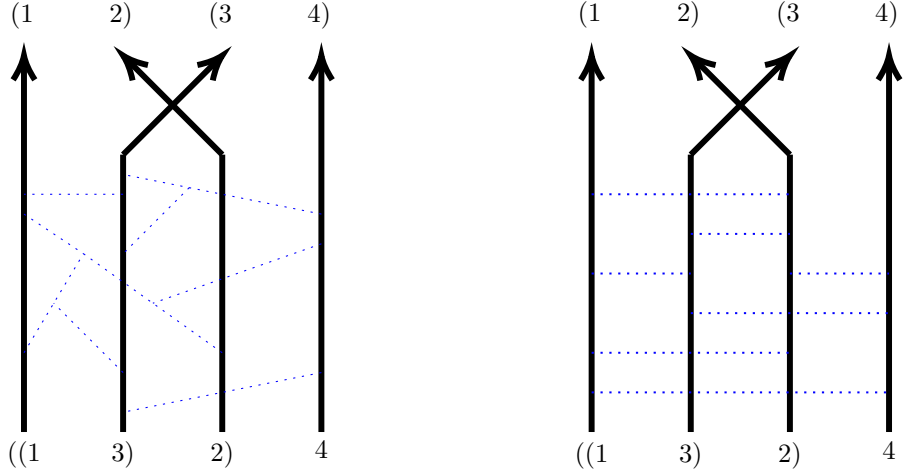
\begin{figure}[h]
    \centering
\begin{subfigure}[c]{0.45\textwidth}
    \centering
\[\begin{tikzpicture}[x=0.75pt,y=0.75pt,yscale=-1,xscale=1]

\draw [line width=2.25]    (100,240) -- (100,54) ;
\draw [shift={(100,50)}, rotate = 90] [color={rgb, 255:red, 0; green, 0; blue, 0 }  ][line width=2.25]    (12.24,-5.49) .. controls (7.79,-2.58) and (3.71,-0.75) .. (0,0) .. controls (3.71,0.75) and (7.79,2.58) .. (12.24,5.49)   ;
\draw [line width=2.25]    (150,240) -- (150,100) ;
\draw [line width=2.25]    (250,240) -- (250,54) ;
\draw [shift={(250,50)}, rotate = 90] [color={rgb, 255:red, 0; green, 0; blue, 0 }  ][line width=2.25]    (12.24,-5.49) .. controls (7.79,-2.58) and (3.71,-0.75) .. (0,0) .. controls (3.71,0.75) and (7.79,2.58) .. (12.24,5.49)   ;
\draw [line width=2.25]    (200,240) -- (200,100) ;
\draw [line width=2.25]    (150,100) -- (197.17,52.83) ;
\draw [shift={(200,50)}, rotate = 135] [color={rgb, 255:red, 0; green, 0; blue, 0 }  ][line width=2.25]    (12.24,-5.49) .. controls (7.79,-2.58) and (3.71,-0.75) .. (0,0) .. controls (3.71,0.75) and (7.79,2.58) .. (12.24,5.49)   ;
\draw [line width=2.25]    (200,100) -- (152.83,52.83) ;
\draw [shift={(150,50)}, rotate = 45] [color={rgb, 255:red, 0; green, 0; blue, 0 }  ][line width=2.25]    (12.24,-5.49) .. controls (7.79,-2.58) and (3.71,-0.75) .. (0,0) .. controls (3.71,0.75) and (7.79,2.58) .. (12.24,5.49)   ;
\draw [color={rgb, 255:red, 0; green, 0; blue, 255 }  ,draw opacity=1 ] [dash pattern={on 0.84pt off 2.51pt}]  (100,200) -- (130,150) ;
\draw [color={rgb, 255:red, 0; green, 0; blue, 255 }  ,draw opacity=1 ] [dash pattern={on 0.84pt off 2.51pt}]  (150,210) -- (115,175) ;
\draw [color={rgb, 255:red, 0; green, 0; blue, 255 }  ,draw opacity=1 ] [dash pattern={on 0.84pt off 2.51pt}]  (100,130) -- (130,150) ;
\draw [color={rgb, 255:red, 0; green, 0; blue, 255 }  ,draw opacity=1 ] [dash pattern={on 0.84pt off 2.51pt}]  (130,150) -- (200,200) ;
\draw [color={rgb, 255:red, 0; green, 0; blue, 255 }  ,draw opacity=1 ] [dash pattern={on 0.84pt off 2.51pt}]  (250,145) -- (165,175) ;
\draw [color={rgb, 255:red, 0; green, 0; blue, 255 }  ,draw opacity=1 ] [dash pattern={on 0.84pt off 2.51pt}]  (250,210) -- (150,230) ;
\draw [color={rgb, 255:red, 0; green, 0; blue, 255 }  ,draw opacity=1 ] [dash pattern={on 0.84pt off 2.51pt}]  (250,130) -- (233.33,126.67) -- (216.67,123.33) -- (200,120) -- (183.33,116.67) -- (166.67,113.33) -- (150,110) ;
\draw [color={rgb, 255:red, 0; green, 0; blue, 255 }  ,draw opacity=1 ] [dash pattern={on 0.84pt off 2.51pt}]  (183.33,116.67) -- (150,150) ;
\draw [color={rgb, 255:red, 0; green, 0; blue, 255 }  ,draw opacity=1 ] [dash pattern={on 0.84pt off 2.51pt}]  (150,120) -- (100,120) ;

\draw (99,37.6) node [anchor=south] [inner sep=0.75pt]    {$( 1$};
\draw (149,37.6) node [anchor=south] [inner sep=0.75pt]    {$2)$};
\draw (199,37.6) node [anchor=south] [inner sep=0.75pt]    {$( 3$};
\draw (249,37.6) node [anchor=south] [inner sep=0.75pt]    {$4)$};
\draw (100,243.4) node [anchor=north] [inner sep=0.75pt]    {$(( 1$};
\draw (150,243.4) node [anchor=north] [inner sep=0.75pt]    {$3)$};
\draw (200,243.4) node [anchor=north] [inner sep=0.75pt]    {$2)$};
\draw (250,243.4) node [anchor=north] [inner sep=0.75pt]    {$4$};

\end{tikzpicture}
\]
    \end{subfigure}
    \begin{subfigure}[c]{0.45\textwidth}
    \centering
\tikzset{every picture/.style={line width=0.75pt}} 
\[\begin{tikzpicture}[x=0.75pt,y=0.75pt,yscale=-1,xscale=1]

\draw [line width=2.25]    (100,240) -- (100,54) ;
\draw [shift={(100,50)}, rotate = 90] [color={rgb, 255:red, 0; green, 0; blue, 0 }  ][line width=2.25]    (12.24,-5.49) .. controls (7.79,-2.58) and (3.71,-0.75) .. (0,0) .. controls (3.71,0.75) and (7.79,2.58) .. (12.24,5.49)   ;
\draw [line width=2.25]    (150,240) -- (150,100) ;
\draw [line width=2.25]    (250,240) -- (250,54) ;
\draw [shift={(250,50)}, rotate = 90] [color={rgb, 255:red, 0; green, 0; blue, 0 }  ][line width=2.25]    (12.24,-5.49) .. controls (7.79,-2.58) and (3.71,-0.75) .. (0,0) .. controls (3.71,0.75) and (7.79,2.58) .. (12.24,5.49)   ;
\draw [line width=2.25]    (200,240) -- (200,100) ;
\draw [line width=2.25]    (150,100) -- (197.17,52.83) ;
\draw [shift={(200,50)}, rotate = 135] [color={rgb, 255:red, 0; green, 0; blue, 0 }  ][line width=2.25]    (12.24,-5.49) .. controls (7.79,-2.58) and (3.71,-0.75) .. (0,0) .. controls (3.71,0.75) and (7.79,2.58) .. (12.24,5.49)   ;
\draw [line width=2.25]    (200,100) -- (152.83,52.83) ;
\draw [shift={(150,50)}, rotate = 45] [color={rgb, 255:red, 0; green, 0; blue, 0 }  ][line width=2.25]    (12.24,-5.49) .. controls (7.79,-2.58) and (3.71,-0.75) .. (0,0) .. controls (3.71,0.75) and (7.79,2.58) .. (12.24,5.49)   ;
\draw [color={rgb, 255:red, 0; green, 0; blue, 255 }  ,draw opacity=1 ] [dash pattern={on 0.84pt off 2.51pt}]  (100,120) -- (200,120) ;
\draw [color={rgb, 255:red, 0; green, 0; blue, 255 }  ,draw opacity=1 ] [dash pattern={on 0.84pt off 2.51pt}]  (150,140) -- (200,140) ;
\draw [color={rgb, 255:red, 0; green, 0; blue, 255 }  ,draw opacity=1 ] [dash pattern={on 0.84pt off 2.51pt}]  (100,160) -- (150,160) ;
\draw [color={rgb, 255:red, 0; green, 0; blue, 255 }  ,draw opacity=1 ] [dash pattern={on 0.84pt off 2.51pt}]  (150,180) -- (250,180) ;
\draw [color={rgb, 255:red, 0; green, 0; blue, 255 }  ,draw opacity=1 ] [dash pattern={on 0.84pt off 2.51pt}]  (100,200) -- (200,200) ;
\draw [color={rgb, 255:red, 0; green, 0; blue, 255 }  ,draw opacity=1 ] [dash pattern={on 0.84pt off 2.51pt}]  (100,220) -- (250,220) ;
\draw [color={rgb, 255:red, 0; green, 0; blue, 255 }  ,draw opacity=1 ] [dash pattern={on 0.84pt off 2.51pt}]  (200,160) -- (250,160) ;

\draw (99,37.6) node [anchor=south] [inner sep=0.75pt]    {$( 1$};
\draw (149,37.6) node [anchor=south] [inner sep=0.75pt]    {$2)$};
\draw (199,37.6) node [anchor=south] [inner sep=0.75pt]    {$( 3$};
\draw (249,37.6) node [anchor=south] [inner sep=0.75pt]    {$4)$};
\draw (100,243.4) node [anchor=north] [inner sep=0.75pt]    {$(( 1$};
\draw (150,243.4) node [anchor=north] [inner sep=0.75pt]    {$3)$};
\draw (200,243.4) node [anchor=north] [inner sep=0.75pt]    {$2)$};
\draw (250,243.4) node [anchor=north] [inner sep=0.75pt]    {$4$};

\end{tikzpicture}\]
    \end{subfigure}
    \caption{On the right is an example of an element of $\PaCD(4)$. The skeleton is depicted by thick black strands and the elements of $U(\mathfrak{t}_4)$ are depicted by dotted blue chords. On the left is a Jacobi diagram on the same skeleton.}
    \label{fig:pacdexample}
\end{figure}

Homomorphic expansions of braids are functors from the tensor category $\PaB$ to its associated graded $\calA = \Pi _{n=0}^\infty \calI^n/\calI^{n+1} \cong \PaCD$. The filtration on $\PaB$ given by powers of the augmentation ideal turns out to agree with the Vassiliev filtration of knot theory, and thus homomorphic expansions agree with universal finite type invariants, such as the Konstevich integral (\cite{MR1237836}, \cite{MO:KontsevichIntegral}, \cite{Dancso_Kont}, etc.)

Explicitly, homomorphic expansions of parenthesised braids $Z:\widehat{\mathbb{Q}[\PaB]}\rightarrow \PaCD$ are isomorphisms of tensor categories determined by the values $Z$ takes on the generating morphisms:\[Z(\overcrossing)=e^{\frac{1}{2}t^{12}} \quad \text{and} \quad Z(\Associator)= \Phi(t^{12},t^{23}),\]  where $\Phi(x,y)\in U(\lie_2)$. These values must necessarily satisfy the following equations, arising from the relations of $\PaB$: 
\begin{equation}\label{inversion}
\Phi(x,y)=\Phi(y,x)^{-1},
\end{equation} 
\begin{equation}\label{hexagon} 
\Phi(t^{12},t^{23})e^{-\frac{1}{2}t^{23}}\Phi(t^{13},t^{23})^{-1}e^{-\frac{1}{2}t^{13}}\Phi(t^{13},t^{12})e^{-\frac{1}{2}(t^{13}+t^{23})} =1, \ \text{in $U(\mathfrak{t}_3)$}
\end{equation} 
\begin{equation}\label{pentagon}  
\Phi(t^{13}+t^{23}, t^{34})\Phi(t^{12},t^{23}+t^{24}) =\Phi(t^{12},t^{23})\Phi(t^{12}+t^{13}, t^{24}+t^{34})\Phi(t^{23},t^{34}), \ \text{in $U(\mathfrak{t}_4)$.}
\end{equation} \emph{Drinfel'd associators} are elements $\Phi(x,y)\in U(\lie_2)$ which satisfy these same equations, leading to the following theorem.

\begin{theorem}\label{thm:BN_Ass}\cite[Proposition 3.4]{BNGT}
There is a one-to-one correspondence between homomorphic expansions of parenthesised braids and the set of Drinfel'd associators. 
\end{theorem} 

\medskip
\medskip

The (prounipotent) Grothendieck--Teichm\"uller groups, $\gt$ and $\grt_1$, defined by Drinfel'd in \cite[Section 4, 5]{Drin90}, are known to act freely and transitively on the set of all Drinfel'd associators. It was shown by Bar-Natan \cite[Propositions 4.5, 4.8]{BNGT} that $\grt_1$ is isomorphic to the (expansion-preserving) automorphism group of $\PaCD$. It was later understood that the same result can be concisely presented using operads in \cite{tamarkin1998proof}, \cite{Will_GT} and \cite[Theorem 10.3.6]{FresseVol1}.

\begin{definition}\label{def:grt1}\cite{Drin90}
The graded Grothendieck--Teichm\"uller group $\grt_1$ is the collection of group-like elements, $\Psi(x, y) = e^{\psi(x,y)}$ of the completed universal enveloping (Hopf) algebra $U(\widehat{\mathfrak{lie}}_2)$ 

\begin{equation}\label{eq:grtinversion}
\Psi(x,y)=\Psi^{-1}(y,x),
\end{equation} 
\begin{equation}\label{eq:grthexagon} 
\Psi(x,y)\Psi(y,z)\Psi(z,x) =1 \ \text{for} \ x + y + z = 0,
\end{equation} 
\begin{equation}\label{eq:grtpentagon}  
\Psi(t^{13}+t^{23}, t^{34})\Psi(t^{12},t^{23}+t^{24}) =\Psi(t^{12},t^{23})\Psi(t^{12}+t^{13}, t^{24}+t^{34})\Psi(t^{23},t^{34})\ \text{in} \ U(\mathfrak{t}_4).
\end{equation} 
 
\medskip 
 \end{definition} 

Automorphisms of $\PaCD$ which preserve expansions by post-composition are characterised by the fact that they fix both the skeleta of chord diagrams and the morphism $\chord\in  \Hom_{\mathsf{PaCD}(2)}((12),(12))$. We slightly abuse notation and write $\Aut(\PaCD)$ for the group of {\em expansion preserving} automorphisms of $\PaCD$.  The following theorem can be found in \cite[Proposition 4.5, Proposition 4.8]{BNGT}, \cite{tamarkin1998proof} or \cite[Theorem 10.3.6]{FresseVol1}. See also \cite[Section 6.7]{merkGT} for a nice expository treatment.

\begin{theorem}\label{thm: GRT is Aut} 
There is an isomorphism of groups $\tau: \Aut(\mathsf{PaCD})\to \grt_1$. 
\end{theorem} 

\begin{proof}[Proof (sketch)]
The key idea in the proof is that an automorphism $F$ which preserves homomorphic expansions is determined by its values on the generators of $\mathsf{PaCD}$. Since expansion preserving automorphisms fix the value of the morphism $\chord$, an automorphism $F$ is uniquely determined by the value $F(\Associator)$, which can be written as a power series in chords $t^{12}$ and $t^{23}$, namely $F(\Associator)=\Psi(t^{12},t^{23})$. Understanding the equations imposed by the defining relations of $\mathsf{PaCD}$ leads precisely to the defining equations of $\grt_1$, and thus $\Psi \in \grt_1$. The isomorphism $\tau$ is defined by $\tau(F)=\Psi$.
\end{proof}

Since homomorphic expansions of parenthesised braids are isomorphisms of tensor categories, one can conclude that for each homomorphic expansion, $Z:\widehat{\mathbb{Q}[\PaB]}\rightarrow \PaCD$, there is an isomorphism $$\Aut(\widehat{\mathbb{Q}[\PaB]})\overset{Z}{\cong} \Aut(\PaCD).$$ From this one can deduce the following (e.g. \cite[Proposition 5.3]{BNGT}, \cite[Proposition 11.3.5]{FresseVol1}): 

\begin{cor}
There is an isomorphism of prounipotent groups: \[\gt_1\overset{Z}{\cong}\Aut(\widehat{\mathbb{Q}[\PaB]}).\]
\end{cor}

\subsection{Welded foams}\label{sec:foams}

Welded braid groups \cite{FRRwBraids}, also known as loop braid groups \cite{Damiani_Loop_braid,Baez_loop_braid_groups}, are generalisations of the classical braid groups. They arise in various contexts in the literature, such as in the study of the configuration spaces of (Euclidean) circles in 3 space, as automorphisms of free groups, string-like defects in 4D BF theory, or as braided ribbon tubes in $\mathbb{R}^4$. A good survey of these ideas can be found in \cite{Damiani_Loop_braid} or the introduction to \cite{BND:WKO1}. 

Welded foams, hereafter referred to as $w$-\emph{foams}, are a class of four-dimensional knotted objects extending $w$-braids in several ways. They can be thought of as tangles of ribbon tubes, ribbon tori and (one-dimensional) strings in $\mathbb{R}^4$. As this paper deals almost entirely with the graded side, the topology of $w$-foams plays a limited role, primarily as motivation. We present a brief overview here for context, and refer the reader to \cite[Sections 3 and 4]{BND:WKO2} for more detail. 

We follow the usual convention for presenting classes of tangles and describe $w$-foams using a Reidemeister, or diagrammatic, description (Definition~\ref{def:wfoams}). 
Specifically, as with parenthesised braids, $w$-foams have a presentation as a type of tensor category, called a two-coloured \emph{wheeled prop}: this is a rigid tensor category freely generated by two objects (cf. \cite[Definition 6.1]{DHR1}).  Morphisms in $w$-foams are all compositions of \emph{crossings},  \emph{vertices} ($\vertex$, $\negvertex$) and \emph{caps} ($\tcap$, $\bcap$) and subject to relations (of the Reidemeister flavour): 

\begin{definition}\label{def:wfoams}
The class of $w$-foams is the two-coloured \emph{wheeled prop}\footnote{In \cite{BND:WKO3} and \cite{BNDS_Duflo}, $\wf$ is denoted $\widetilde{wTF}$.}:
\begin{equation*}
\wf = \left\langle \ocrossing, \ucrossing, \vertex, \negvertex, \pv, \ppv,  \bcap \, \ \Bigg{|}\, R1^f, R2, R3, OC, CP \right\rangle.    
\end{equation*}
By convention and for brevity, generators are only shown with strands oriented upwards, but they are understood to come in all possible strand orientations.
\end{definition}

An example of a $w$-foam depicted as tangle diagram is shown in Figure~\ref{fig:wfoam}.  In this diagram, tubes are drawn as thick black lines and strings as thin red lines. Each generator, circled with dashed lines, stands for a specific topological feature embedded in $\R^4$: crossings are tubes threaded through each other as in Figure~\ref{fig:4DCrossing};  \emph{vertices} represent tubes merging, $\vertex$, or splitting, $\negvertex$, and the tube-string vertex $\pv$ represents a string attached to the side of a tube. For more detail on vertices, see around \cite[Figure 16]{BND:WKO2}. A cap, denoted by a capped black strand $\bcap$, stands for a tube capped off by a disk.  The crossings in Figure~\ref{fig:wfoam} which are \emph{not} circled are virtual or \emph{wiring diagram} crossings, as shown in Figure~\ref{fig:4DCrossing}. 

\begin{figure}[h]
\[\begin{tikzpicture}[x=0.75pt,y=0.75pt,yscale=-1,xscale=1]

\draw [color={rgb, 255:red, 255; green, 0; blue, 0 }  ,draw opacity=1 ]   (208.58,216.86) .. controls (213.44,210.56) and (217.78,201.56) .. (219.11,198.22) ;
\draw  [draw opacity=0] (300,141.5) .. controls (300,174.78) and (286.51,204.9) .. (264.71,226.71) -- (179.5,141.5) -- cycle ; \draw   (300,141.5) .. controls (300,174.78) and (286.51,204.9) .. (264.71,226.71) ;  
\draw  [draw opacity=0] (94.29,56.29) .. controls (116.1,34.49) and (146.22,21) .. (179.5,21) -- (179.5,141.5) -- cycle ; \draw   (94.29,56.29) .. controls (116.1,34.49) and (146.22,21) .. (179.5,21) ;  
\draw  [draw opacity=0] (179.5,21) .. controls (179.5,21) and (179.5,21) .. (179.5,21) .. controls (212.78,21) and (242.9,34.49) .. (264.71,56.29) -- (179.5,141.5) -- cycle ; \draw   (179.5,21) .. controls (179.5,21) and (179.5,21) .. (179.5,21) .. controls (212.78,21) and (242.9,34.49) .. (264.71,56.29) ;  
\draw  [draw opacity=0] (264.71,56.29) .. controls (286.51,78.1) and (300,108.22) .. (300,141.5) -- (179.5,141.5) -- cycle ; \draw   (264.71,56.29) .. controls (286.51,78.1) and (300,108.22) .. (300,141.5) ;  
\draw  [draw opacity=0][dash pattern={on 4.5pt off 4.5pt}] (162.87,92.6) .. controls (159.1,102.09) and (148.34,106.73) .. (138.84,102.96) -- (145.68,85.77) -- cycle ; \draw  [dash pattern={on 4.5pt off 4.5pt}] (162.87,92.6) .. controls (159.1,102.09) and (148.34,106.73) .. (138.84,102.96) ;  
\draw  [draw opacity=0][dash pattern={on 4.5pt off 4.5pt}] (128.49,78.93) .. controls (128.49,78.93) and (128.49,78.93) .. (128.49,78.93) .. controls (132.26,69.44) and (143.02,64.8) .. (152.51,68.57) -- (145.68,85.77) -- cycle ; \draw  [dash pattern={on 4.5pt off 4.5pt}] (128.49,78.93) .. controls (128.49,78.93) and (128.49,78.93) .. (128.49,78.93) .. controls (132.26,69.44) and (143.02,64.8) .. (152.51,68.57) ;  
\draw  [draw opacity=0][dash pattern={on 4.5pt off 4.5pt}] (152.51,68.57) .. controls (162.01,72.35) and (166.64,83.11) .. (162.87,92.6) -- (145.68,85.77) -- cycle ; \draw  [dash pattern={on 4.5pt off 4.5pt}] (152.51,68.57) .. controls (162.01,72.35) and (166.64,83.11) .. (162.87,92.6) ;  
\draw  [draw opacity=0][dash pattern={on 4.5pt off 4.5pt}] (237.67,71.5) .. controls (237.67,81.72) and (229.38,90) .. (219.17,90) -- (219.17,71.5) -- cycle ; \draw  [dash pattern={on 4.5pt off 4.5pt}] (237.67,71.5) .. controls (237.67,81.72) and (229.38,90) .. (219.17,90) ;  
\draw  [draw opacity=0][dash pattern={on 4.5pt off 4.5pt}] (219.17,90) .. controls (219.17,90) and (219.17,90) .. (219.17,90) .. controls (208.95,90) and (200.67,81.72) .. (200.67,71.5) -- (219.17,71.5) -- cycle ; \draw  [dash pattern={on 4.5pt off 4.5pt}] (219.17,90) .. controls (219.17,90) and (219.17,90) .. (219.17,90) .. controls (208.95,90) and (200.67,81.72) .. (200.67,71.5) ;  
\draw  [draw opacity=0][dash pattern={on 4.5pt off 4.5pt}] (200.67,71.5) .. controls (200.67,71.5) and (200.67,71.5) .. (200.67,71.5) .. controls (200.67,61.28) and (208.95,53) .. (219.17,53) -- (219.17,71.5) -- cycle ; \draw  [dash pattern={on 4.5pt off 4.5pt}] (200.67,71.5) .. controls (200.67,71.5) and (200.67,71.5) .. (200.67,71.5) .. controls (200.67,61.28) and (208.95,53) .. (219.17,53) ;  
\draw  [draw opacity=0][dash pattern={on 4.5pt off 4.5pt}] (219.17,53) .. controls (229.38,53) and (237.67,61.28) .. (237.67,71.5) -- (219.17,71.5) -- cycle ; \draw  [dash pattern={on 4.5pt off 4.5pt}] (219.17,53) .. controls (229.38,53) and (237.67,61.28) .. (237.67,71.5) ;  
\draw  [draw opacity=0][dash pattern={on 4.5pt off 4.5pt}] (194.33,134.17) .. controls (194.33,144.38) and (186.05,152.67) .. (175.83,152.67) -- (175.83,134.17) -- cycle ; \draw  [dash pattern={on 4.5pt off 4.5pt}] (194.33,134.17) .. controls (194.33,144.38) and (186.05,152.67) .. (175.83,152.67) ;  
\draw  [draw opacity=0][dash pattern={on 4.5pt off 4.5pt}] (175.83,152.67) .. controls (175.83,152.67) and (175.83,152.67) .. (175.83,152.67) .. controls (165.62,152.67) and (157.33,144.38) .. (157.33,134.17) -- (175.83,134.17) -- cycle ; \draw  [dash pattern={on 4.5pt off 4.5pt}] (175.83,152.67) .. controls (175.83,152.67) and (175.83,152.67) .. (175.83,152.67) .. controls (165.62,152.67) and (157.33,144.38) .. (157.33,134.17) ;  
\draw  [draw opacity=0][dash pattern={on 4.5pt off 4.5pt}] (157.33,134.17) .. controls (157.33,134.17) and (157.33,134.17) .. (157.33,134.17) .. controls (157.33,123.95) and (165.62,115.67) .. (175.83,115.67) -- (175.83,134.17) -- cycle ; \draw  [dash pattern={on 4.5pt off 4.5pt}] (157.33,134.17) .. controls (157.33,134.17) and (157.33,134.17) .. (157.33,134.17) .. controls (157.33,123.95) and (165.62,115.67) .. (175.83,115.67) ;  
\draw  [draw opacity=0][dash pattern={on 4.5pt off 4.5pt}] (175.83,115.67) .. controls (186.05,115.67) and (194.33,123.95) .. (194.33,134.17) -- (175.83,134.17) -- cycle ; \draw  [dash pattern={on 4.5pt off 4.5pt}] (175.83,115.67) .. controls (186.05,115.67) and (194.33,123.95) .. (194.33,134.17) ;  
\draw [line width=2.25]    (194.33,134.17) .. controls (183.22,132.33) and (175.89,133.67) .. (157.33,134.17) ;
\draw [line width=2.25]    (208.58,184.81) .. controls (197.22,170.33) and (174.56,172.33) .. (175.83,152.67) ;
\draw [shift={(195.36,174.72)}, rotate = 206.87] [fill={rgb, 255:red, 0; green, 0; blue, 0 }  ][line width=0.08]  [draw opacity=0] (8.57,-4.12) -- (0,0) -- (8.57,4.12) -- cycle    ;
\draw [line width=2.25]    (175.83,129.17) .. controls (175.89,124) and (175.89,125.67) .. (175.83,115.67) ;
\draw [line width=2.25]    (175.83,152.67) .. controls (175.89,147.5) and (175.89,147.17) .. (175.83,137.17) ;
\draw [line width=2.25]    (219.17,90) .. controls (210.33,126.78) and (172.33,51.44) .. (152.51,68.57) ;
\draw [line width=2.25]    (162.87,92.6) .. controls (176.33,106.11) and (135.67,128.11) .. (157.33,134.17) ;
\draw [line width=2.25]    (179.5,21) .. controls (177.89,39.67) and (174.56,91.67) .. (175.83,115.67) ;
\draw [shift={(176.48,68.54)}, rotate = 272.77] [fill={rgb, 255:red, 0; green, 0; blue, 0 }  ][line width=0.08]  [draw opacity=0] (8.57,-4.12) -- (0,0) -- (8.57,4.12) -- cycle    ;
\draw [line width=2.25]    (200.67,71.5) .. controls (189.56,69.67) and (211.22,131) .. (194.33,134.17) ;
\draw [line width=2.25]    (237.67,71.5) .. controls (254.56,73) and (225.22,37) .. (219.17,53) ;
\draw [line width=2.25]    (300,141.5) .. controls (281.22,146.33) and (252.56,69.67) .. (264.71,56.29) ;
\draw [shift={(268.05,98.37)}, rotate = 70.18] [fill={rgb, 255:red, 0; green, 0; blue, 0 }  ][line width=0.08]  [draw opacity=0] (8.57,-4.12) -- (0,0) -- (8.57,4.12) -- cycle    ;
\draw [line width=2.25]    (162.87,92.6) .. controls (153.22,86.79) and (145.92,85.32) .. (128.49,78.93) ;
\draw [line width=2.25]    (147.16,81.05) .. controls (149.12,76.27) and (148.87,77.89) .. (152.51,68.57) ;
\draw [line width=2.25]    (219.17,53) .. controls (219,76.11) and (219.67,79.44) .. (219.17,90) ;
\draw [line width=2.25]    (215.17,71.5) .. controls (210.44,71.33) and (205.78,70.67) .. (200.67,71.5) ;
\draw [line width=2.25]    (237.67,71.5) .. controls (232.94,71.33) and (228.28,70.67) .. (223.17,71.5) ;
\draw  [line width=2.25]  (231.44,128.67) .. controls (245.11,128) and (240.11,138.33) .. (243.44,144) .. controls (246.78,149.67) and (277.11,145.33) .. (271.78,158) .. controls (266.44,170.67) and (241.78,173.33) .. (231.11,170) .. controls (220.44,166.67) and (217.78,129.33) .. (231.44,128.67) -- cycle ;
\draw  [draw opacity=0] (264.71,226.71) .. controls (242.9,248.51) and (212.78,262) .. (179.5,262) -- (179.5,141.5) -- cycle ; \draw   (264.71,226.71) .. controls (242.9,248.51) and (212.78,262) .. (179.5,262) ;  
\draw  [draw opacity=0] (179.5,262) .. controls (146.22,262) and (116.1,248.51) .. (94.29,226.71) -- (179.5,141.5) -- cycle ; \draw   (179.5,262) .. controls (146.22,262) and (116.1,248.51) .. (94.29,226.71) ;  
\draw  [draw opacity=0] (94.29,226.71) .. controls (72.49,204.9) and (59,174.78) .. (59,141.5) -- (179.5,141.5) -- cycle ; \draw   (94.29,226.71) .. controls (72.49,204.9) and (59,174.78) .. (59,141.5) ;  
\draw  [draw opacity=0] (59,141.5) .. controls (59,141.5) and (59,141.5) .. (59,141.5) .. controls (59,108.22) and (72.49,78.1) .. (94.29,56.29) -- (179.5,141.5) -- cycle ; \draw   (59,141.5) .. controls (59,141.5) and (59,141.5) .. (59,141.5) .. controls (59,108.22) and (72.49,78.1) .. (94.29,56.29) ;  
\draw  [draw opacity=0][dash pattern={on 4.5pt off 4.5pt}] (129.29,173.71) .. controls (122.92,181.69) and (111.28,183) .. (103.29,176.63) .. controls (100.66,174.52) and (98.75,171.85) .. (97.61,168.92) -- (114.83,162.17) -- cycle ; \draw  [dash pattern={on 4.5pt off 4.5pt}] (129.29,173.71) .. controls (122.92,181.69) and (111.28,183) .. (103.29,176.63) .. controls (100.66,174.52) and (98.75,171.85) .. (97.61,168.92) ;  
\draw  [draw opacity=0][dash pattern={on 4.5pt off 4.5pt}] (97.61,168.92) .. controls (95.28,162.97) and (96.1,155.98) .. (100.37,150.63) .. controls (104.64,145.28) and (111.28,142.92) .. (117.6,143.87) -- (114.83,162.17) -- cycle ; \draw  [dash pattern={on 4.5pt off 4.5pt}] (97.61,168.92) .. controls (95.28,162.97) and (96.1,155.98) .. (100.37,150.63) .. controls (104.64,145.28) and (111.28,142.92) .. (117.6,143.87) ;  
\draw  [draw opacity=0][dash pattern={on 4.5pt off 4.5pt}] (117.6,143.87) .. controls (120.71,144.34) and (123.74,145.6) .. (126.37,147.71) .. controls (134.36,154.08) and (135.67,165.72) .. (129.29,173.71) -- (114.83,162.17) -- cycle ; \draw  [dash pattern={on 4.5pt off 4.5pt}] (117.6,143.87) .. controls (120.71,144.34) and (123.74,145.6) .. (126.37,147.71) .. controls (134.36,154.08) and (135.67,165.72) .. (129.29,173.71) ;  
\draw  [draw opacity=0][dash pattern={on 4.5pt off 4.5pt}] (236.33,200.83) .. controls (236.33,211.05) and (228.05,219.33) .. (217.83,219.33) .. controls (214.46,219.33) and (211.3,218.43) .. (208.58,216.86) -- (217.83,200.83) -- cycle ; \draw  [dash pattern={on 4.5pt off 4.5pt}] (236.33,200.83) .. controls (236.33,211.05) and (228.05,219.33) .. (217.83,219.33) .. controls (214.46,219.33) and (211.3,218.43) .. (208.58,216.86) ;  
\draw  [draw opacity=0][dash pattern={on 4.5pt off 4.5pt}] (208.58,216.86) .. controls (203.05,213.66) and (199.33,207.68) .. (199.33,200.83) .. controls (199.33,193.99) and (203.05,188.01) .. (208.58,184.81) -- (217.83,200.83) -- cycle ; \draw  [dash pattern={on 4.5pt off 4.5pt}] (208.58,216.86) .. controls (203.05,213.66) and (199.33,207.68) .. (199.33,200.83) .. controls (199.33,193.99) and (203.05,188.01) .. (208.58,184.81) ;  
\draw  [draw opacity=0][dash pattern={on 4.5pt off 4.5pt}] (208.58,184.81) .. controls (211.3,183.23) and (214.46,182.33) .. (217.83,182.33) .. controls (228.05,182.33) and (236.33,190.62) .. (236.33,200.83) -- (217.83,200.83) -- cycle ; \draw  [dash pattern={on 4.5pt off 4.5pt}] (208.58,184.81) .. controls (211.3,183.23) and (214.46,182.33) .. (217.83,182.33) .. controls (228.05,182.33) and (236.33,190.62) .. (236.33,200.83) ;  
\draw [line width=2.25]    (236.33,200.83) .. controls (230.44,200.22) and (221.11,199.56) .. (219.11,198.22) .. controls (217.11,196.89) and (215.44,193.56) .. (208.58,184.81) ;
\draw [line width=2.25]    (129.29,173.71) .. controls (123.78,170.22) and (116.83,163.5) .. (114.83,162.17) .. controls (112.83,160.83) and (115.78,154.89) .. (117.6,143.87) ;
\draw [line width=2.25]    (97.61,168.92) .. controls (102.44,165.89) and (100.11,167.89) .. (110.83,163.17) ;
\draw [line width=2.25]    (59,141.5) .. controls (73.89,135) and (81.22,181.67) .. (97.61,168.92) ;
\draw [shift={(81.23,162.9)}, rotate = 236.86] [fill={rgb, 255:red, 0; green, 0; blue, 0 }  ][line width=0.08]  [draw opacity=0] (8.57,-4.12) -- (0,0) -- (8.57,4.12) -- cycle    ;
\draw [line width=2.25]    (94.29,226.71) .. controls (109.89,212.33) and (140.56,184.33) .. (129.29,173.71) ;
\draw [shift={(122.59,197.89)}, rotate = 129.88] [fill={rgb, 255:red, 0; green, 0; blue, 0 }  ][line width=0.08]  [draw opacity=0] (8.57,-4.12) -- (0,0) -- (8.57,4.12) -- cycle    ;
\draw [color={rgb, 255:red, 255; green, 0; blue, 0 }  ,draw opacity=1 ]   (179.5,262) .. controls (187.22,253.67) and (190.56,241.67) .. (208.58,216.86) ;
\draw [shift={(195.81,235.97)}, rotate = 120.75] [fill={rgb, 255:red, 255; green, 0; blue, 0 }  ,fill opacity=1 ][line width=0.08]  [draw opacity=0] (7.14,-3.43) -- (0,0) -- (7.14,3.43) -- cycle    ;
\draw  [draw opacity=0][dash pattern={on 4.5pt off 4.5pt}] (138.84,102.96) .. controls (138.84,102.96) and (138.84,102.96) .. (138.84,102.96) .. controls (138.84,102.96) and (138.84,102.96) .. (138.84,102.96) .. controls (129.35,99.18) and (124.71,88.43) .. (128.49,78.93) -- (145.68,85.77) -- cycle ; \draw  [dash pattern={on 4.5pt off 4.5pt}] (138.84,102.96) .. controls (138.84,102.96) and (138.84,102.96) .. (138.84,102.96) .. controls (138.84,102.96) and (138.84,102.96) .. (138.84,102.96) .. controls (129.35,99.18) and (124.71,88.43) .. (128.49,78.93) ;  
\draw [line width=2.25]    (264.71,226.71) .. controls (261.89,215) and (259.22,202.33) .. (236.33,200.83) ;
\draw [line width=2.25]    (94.29,56.29) .. controls (97.5,71.17) and (109.5,87.17) .. (112.5,92.17) ;
\draw [shift={(98.36,68.13)}, rotate = 63.73] [fill={rgb, 255:red, 0; green, 0; blue, 0 }  ][line width=0.08]  [draw opacity=0] (8.57,-4.12) -- (0,0) -- (8.57,4.12) -- cycle    ;
\draw [line width=2.25]    (112.5,92.17) .. controls (105,107.67) and (116.33,130.11) .. (117.6,143.87) ;
\draw [shift={(110.34,110.6)}, rotate = 78.47] [fill={rgb, 255:red, 0; green, 0; blue, 0 }  ][line width=0.08]  [draw opacity=0] (8.57,-4.12) -- (0,0) -- (8.57,4.12) -- cycle    ;
\draw [line width=2.25]    (138.84,102.96) .. controls (140.81,98.18) and (140.93,97.87) .. (144.57,88.55) ;
\draw [line width=2.25]    (128.49,78.93) .. controls (119.5,80.67) and (116.5,84.67) .. (112.5,92.17) ;
\draw [line width=2.25]    (112.5,92.17) .. controls (118,105.17) and (136,113.67) .. (138.84,102.96) ;
\draw  [fill={rgb, 255:red, 0; green, 0; blue, 0 }  ,fill opacity=1 ] (222.13,145.33) -- (226.75,152.33) -- (217.5,152.33) -- cycle ;

\end{tikzpicture}\]
    \caption{An example of a circuit algebra composition of five generators in $\wf$.}
    \label{fig:wfoam}
\end{figure}
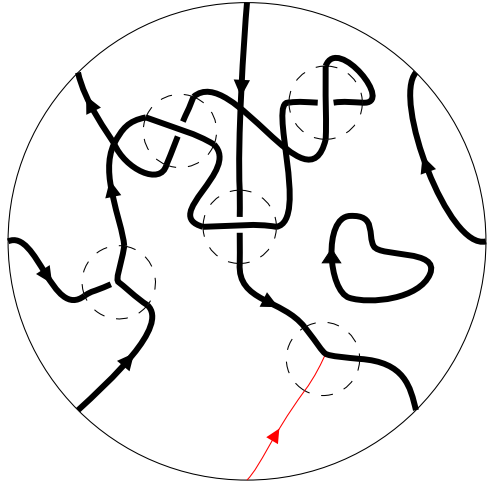

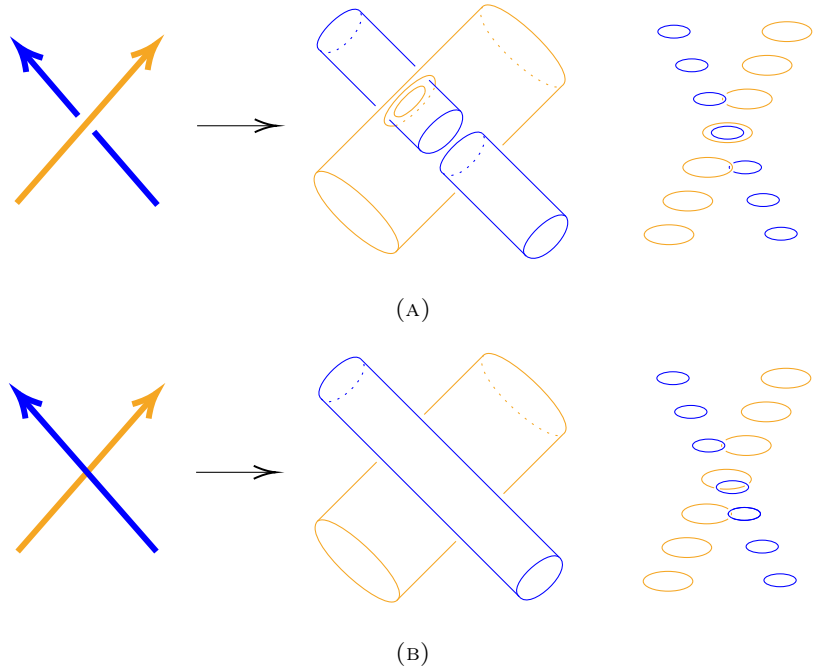
\begin{figure}
    \centering
    \begin{subfigure}[c]{\textwidth}
    \centering
\[\begin{tikzpicture}[x=0.75pt,y=0.75pt,yscale=-1,xscale=1]

\draw [color={rgb, 255:red, 0; green, 0; blue, 255 }  ,draw opacity=1 ][line width=2.25]    (128,143) -- (160,180) ;
\draw [color={rgb, 255:red, 0; green, 0; blue, 255 }  ,draw opacity=1 ][line width=2.25]    (120,135) -- (92.6,103.04) ;
\draw [shift={(90,100)}, rotate = 49.4] [color={rgb, 255:red, 0; green, 0; blue, 255 }  ,draw opacity=1 ][line width=2.25]    (12.24,-5.49) .. controls (7.79,-2.58) and (3.71,-0.75) .. (0,0) .. controls (3.71,0.75) and (7.79,2.58) .. (12.24,5.49)   ;
\draw [color={rgb, 255:red, 245; green, 166; blue, 35 }  ,draw opacity=1 ][line width=2.25]    (156.37,102.01) -- (89,179) ;
\draw [shift={(159,99)}, rotate = 131.19] [color={rgb, 255:red, 245; green, 166; blue, 35 }  ,draw opacity=1 ][line width=2.25]    (12.24,-5.49) .. controls (7.79,-2.58) and (3.71,-0.75) .. (0,0) .. controls (3.71,0.75) and (7.79,2.58) .. (12.24,5.49)   ;
\draw [color={rgb, 255:red, 0; green, 0; blue, 255 }  ,draw opacity=1 ]   (304,164) -- (346.05,205.82) ;
\draw [color={rgb, 255:red, 0; green, 0; blue, 255 }  ,draw opacity=1 ]   (324,144) -- (366.05,185.82) ;
\draw [color={rgb, 255:red, 245; green, 166; blue, 35 }  ,draw opacity=1 ]   (310.05,173.82) -- (280.05,203.82) ;
\draw [color={rgb, 255:red, 245; green, 166; blue, 35 }  ,draw opacity=1 ]   (294.05,109.82) -- (240.05,163.82) ;
\draw [color={rgb, 255:red, 245; green, 166; blue, 35 }  ,draw opacity=1 ]   (364.05,119.82) -- (334.05,149.82) ;
\draw [color={rgb, 255:red, 245; green, 166; blue, 35 }  ,draw opacity=1 ]   (324.05,79.82) -- (294.05,109.82) ;
\draw [color={rgb, 255:red, 245; green, 166; blue, 35 }  ,draw opacity=1 ]   (324.05,79.82) .. controls (332.05,72.82) and (374.05,112.82) .. (364.05,119.82) ;
\draw [color={rgb, 255:red, 245; green, 166; blue, 35 }  ,draw opacity=1 ]   (240.05,163.82) .. controls (248.05,156.82) and (290.05,196.82) .. (280.05,203.82) ;
\draw [color={rgb, 255:red, 245; green, 166; blue, 35 }  ,draw opacity=1 ] [dash pattern={on 0.84pt off 2.51pt}]  (324.05,79.82) .. controls (316.05,89.82) and (356.05,125.82) .. (364.05,119.82) ;
\draw [color={rgb, 255:red, 245; green, 166; blue, 35 }  ,draw opacity=1 ]   (240.05,163.82) .. controls (232.05,173.82) and (272.05,209.82) .. (280.05,203.82) ;
\draw [color={rgb, 255:red, 0; green, 0; blue, 255 }  ,draw opacity=1 ]   (241.99,101.99) -- (270,130) ;
\draw [color={rgb, 255:red, 0; green, 0; blue, 255 }  ,draw opacity=1 ]   (261.99,81.99) -- (290,110) ;
\draw [color={rgb, 255:red, 0; green, 0; blue, 255 }  ,draw opacity=1 ]   (293.05,152.82) .. controls (286,148.67) and (307,127.67) .. (313.05,132.82) ;
\draw [color={rgb, 255:red, 0; green, 0; blue, 255 }  ,draw opacity=1 ]   (293.05,152.82) .. controls (301,158.67) and (318,139.67) .. (313.05,132.82) ;
\draw [color={rgb, 255:red, 0; green, 0; blue, 255 }  ,draw opacity=1 ]   (304,164) .. controls (296.95,159.85) and (317.95,138.85) .. (324,144) ;
\draw [color={rgb, 255:red, 0; green, 0; blue, 255 }  ,draw opacity=1 ] [dash pattern={on 0.84pt off 2.51pt}]  (304,164) .. controls (311.95,169.85) and (328.95,150.85) .. (324,144) ;
\draw [color={rgb, 255:red, 0; green, 0; blue, 255 }  ,draw opacity=1 ]   (241.99,101.99) .. controls (234.94,97.84) and (255.94,76.84) .. (261.99,81.99) ;
\draw [color={rgb, 255:red, 0; green, 0; blue, 255 }  ,draw opacity=1 ] [dash pattern={on 0.84pt off 2.51pt}]  (241.99,101.99) .. controls (249.94,107.84) and (266.94,88.84) .. (261.99,81.99) ;
\draw [color={rgb, 255:red, 0; green, 0; blue, 255 }  ,draw opacity=1 ]   (346.05,205.82) .. controls (339,201.67) and (360,180.67) .. (366.05,185.82) ;
\draw [color={rgb, 255:red, 0; green, 0; blue, 255 }  ,draw opacity=1 ]   (346.05,205.82) .. controls (354,211.67) and (371,192.67) .. (366.05,185.82) ;
\draw [color={rgb, 255:red, 245; green, 166; blue, 35 }  ,draw opacity=1 ]   (280,140) .. controls (259.25,143.25) and (302.75,99.75) .. (300,120) ;
\draw [color={rgb, 255:red, 245; green, 166; blue, 35 }  ,draw opacity=1 ] [dash pattern={on 0.84pt off 2.51pt}]  (280,140) .. controls (287.5,137) and (297.25,127.5) .. (300,120) ;
\draw [color={rgb, 255:red, 0; green, 0; blue, 255 }  ,draw opacity=1 ]   (297,117) -- (313.05,132.82) ;
\draw [color={rgb, 255:red, 0; green, 0; blue, 255 }  ,draw opacity=1 ]   (277,137) -- (293.05,152.82) ;
\draw [color={rgb, 255:red, 245; green, 166; blue, 35 }  ,draw opacity=1 ]   (282.82,135.38) .. controls (270.25,137.34) and (296.59,111.01) .. (294.92,123.27) ;
\draw [color={rgb, 255:red, 245; green, 166; blue, 35 }  ,draw opacity=1 ]   (282.82,135.38) .. controls (287.36,133.56) and (293.26,127.81) .. (294.92,123.27) ;
\draw    (180,140) -- (218,140) ;
\draw [shift={(220,140)}, rotate = 180] [color={rgb, 255:red, 0; green, 0; blue, 0 }  ][line width=0.75]    (10.93,-3.29) .. controls (6.95,-1.4) and (3.31,-0.3) .. (0,0) .. controls (3.31,0.3) and (6.95,1.4) .. (10.93,3.29)   ;
\draw  [color={rgb, 255:red, 245; green, 166; blue, 35 }  ,draw opacity=1 ] (405.5,195) .. controls (405.5,192.24) and (411.1,190) .. (418,190) .. controls (424.9,190) and (430.5,192.24) .. (430.5,195) .. controls (430.5,197.76) and (424.9,200) .. (418,200) .. controls (411.1,200) and (405.5,197.76) .. (405.5,195) -- cycle ;
\draw  [color={rgb, 255:red, 245; green, 166; blue, 35 }  ,draw opacity=1 ] (415,178) .. controls (415,175.24) and (420.6,173) .. (427.5,173) .. controls (434.4,173) and (440,175.24) .. (440,178) .. controls (440,180.76) and (434.4,183) .. (427.5,183) .. controls (420.6,183) and (415,180.76) .. (415,178) -- cycle ;
\draw  [color={rgb, 255:red, 245; green, 166; blue, 35 }  ,draw opacity=1 ] (435,143.5) .. controls (435,140.74) and (440.6,138.5) .. (447.5,138.5) .. controls (454.4,138.5) and (460,140.74) .. (460,143.5) .. controls (460,146.26) and (454.4,148.5) .. (447.5,148.5) .. controls (440.6,148.5) and (435,146.26) .. (435,143.5) -- cycle ;
\draw  [draw opacity=0] (445.1,127.15) .. controls (445.04,126.94) and (445,126.72) .. (445,126.5) .. controls (445,126.28) and (445.04,126.06) .. (445.1,125.85) -- (457.5,126.5) -- cycle ; \draw  [color={rgb, 255:red, 245; green, 166; blue, 35 }  ,draw opacity=1 ] (445.1,127.15) .. controls (445.04,126.94) and (445,126.72) .. (445,126.5) .. controls (445,126.28) and (445.04,126.06) .. (445.1,125.85) ;  
\draw  [color={rgb, 255:red, 245; green, 166; blue, 35 }  ,draw opacity=1 ] (455,109.5) .. controls (455,106.74) and (460.6,104.5) .. (467.5,104.5) .. controls (474.4,104.5) and (480,106.74) .. (480,109.5) .. controls (480,112.26) and (474.4,114.5) .. (467.5,114.5) .. controls (460.6,114.5) and (455,112.26) .. (455,109.5) -- cycle ;
\draw  [color={rgb, 255:red, 245; green, 166; blue, 35 }  ,draw opacity=1 ] (465,92.5) .. controls (465,89.74) and (470.6,87.5) .. (477.5,87.5) .. controls (484.4,87.5) and (490,89.74) .. (490,92.5) .. controls (490,95.26) and (484.4,97.5) .. (477.5,97.5) .. controls (470.6,97.5) and (465,95.26) .. (465,92.5) -- cycle ;
\draw  [color={rgb, 255:red, 0; green, 0; blue, 255 }  ,draw opacity=1 ] (439.38,143.5) .. controls (439.38,141.71) and (443.01,140.25) .. (447.5,140.25) .. controls (451.99,140.25) and (455.63,141.71) .. (455.63,143.5) .. controls (455.63,145.29) and (451.99,146.75) .. (447.5,146.75) .. controls (443.01,146.75) and (439.38,145.29) .. (439.38,143.5) -- cycle ;
\draw  [draw opacity=0] (448.56,161.69) .. controls (448.44,161.47) and (448.38,161.24) .. (448.38,161) .. controls (448.38,160.76) and (448.44,160.53) .. (448.56,160.31) -- (456.5,161) -- cycle ; \draw  [color={rgb, 255:red, 0; green, 0; blue, 255 }  ,draw opacity=1 ] (448.56,161.69) .. controls (448.44,161.47) and (448.38,161.24) .. (448.38,161) .. controls (448.38,160.76) and (448.44,160.53) .. (448.56,160.31) ;  
\draw  [color={rgb, 255:red, 245; green, 166; blue, 35 }  ,draw opacity=1 ] (425,161) .. controls (425,158.24) and (430.6,156) .. (437.5,156) .. controls (444.4,156) and (450,158.24) .. (450,161) .. controls (450,163.76) and (444.4,166) .. (437.5,166) .. controls (430.6,166) and (425,163.76) .. (425,161) -- cycle ;
\draw  [color={rgb, 255:red, 0; green, 0; blue, 255 }  ,draw opacity=1 ] (457.38,177.5) .. controls (457.38,175.71) and (461.01,174.25) .. (465.5,174.25) .. controls (469.99,174.25) and (473.63,175.71) .. (473.63,177.5) .. controls (473.63,179.29) and (469.99,180.75) .. (465.5,180.75) .. controls (461.01,180.75) and (457.38,179.29) .. (457.38,177.5) -- cycle ;
\draw  [color={rgb, 255:red, 0; green, 0; blue, 255 }  ,draw opacity=1 ] (466.38,194.5) .. controls (466.38,192.71) and (470.01,191.25) .. (474.5,191.25) .. controls (478.99,191.25) and (482.63,192.71) .. (482.63,194.5) .. controls (482.63,196.29) and (478.99,197.75) .. (474.5,197.75) .. controls (470.01,197.75) and (466.38,196.29) .. (466.38,194.5) -- cycle ;
\draw  [color={rgb, 255:red, 0; green, 0; blue, 255 }  ,draw opacity=1 ] (430.38,126.5) .. controls (430.38,124.71) and (434.01,123.25) .. (438.5,123.25) .. controls (442.99,123.25) and (446.63,124.71) .. (446.63,126.5) .. controls (446.63,128.29) and (442.99,129.75) .. (438.5,129.75) .. controls (434.01,129.75) and (430.38,128.29) .. (430.38,126.5) -- cycle ;
\draw  [color={rgb, 255:red, 0; green, 0; blue, 255 }  ,draw opacity=1 ] (421.38,109.5) .. controls (421.38,107.71) and (425.01,106.25) .. (429.5,106.25) .. controls (433.99,106.25) and (437.63,107.71) .. (437.63,109.5) .. controls (437.63,111.29) and (433.99,112.75) .. (429.5,112.75) .. controls (425.01,112.75) and (421.38,111.29) .. (421.38,109.5) -- cycle ;
\draw  [color={rgb, 255:red, 0; green, 0; blue, 255 }  ,draw opacity=1 ] (412.38,92.5) .. controls (412.38,90.71) and (416.01,89.25) .. (420.5,89.25) .. controls (424.99,89.25) and (428.63,90.71) .. (428.63,92.5) .. controls (428.63,94.29) and (424.99,95.75) .. (420.5,95.75) .. controls (416.01,95.75) and (412.38,94.29) .. (412.38,92.5) -- cycle ;
\draw  [draw opacity=0] (450.49,158.81) .. controls (451.98,158.16) and (454.12,157.75) .. (456.5,157.75) .. controls (460.99,157.75) and (464.63,159.21) .. (464.63,161) .. controls (464.63,162.79) and (460.99,164.25) .. (456.5,164.25) .. controls (454.12,164.25) and (451.98,163.84) .. (450.49,163.19) -- (456.5,161) -- cycle ; \draw  [color={rgb, 255:red, 0; green, 0; blue, 255 }  ,draw opacity=1 ] (450.49,158.81) .. controls (451.98,158.16) and (454.12,157.75) .. (456.5,157.75) .. controls (460.99,157.75) and (464.63,159.21) .. (464.63,161) .. controls (464.63,162.79) and (460.99,164.25) .. (456.5,164.25) .. controls (454.12,164.25) and (451.98,163.84) .. (450.49,163.19) ;  
\draw  [draw opacity=0] (447.11,123.72) .. controls (449.36,122.38) and (453.17,121.5) .. (457.5,121.5) .. controls (464.4,121.5) and (470,123.74) .. (470,126.5) .. controls (470,129.26) and (464.4,131.5) .. (457.5,131.5) .. controls (453.17,131.5) and (449.36,130.62) .. (447.11,129.28) -- (457.5,126.5) -- cycle ; \draw  [color={rgb, 255:red, 245; green, 166; blue, 35 }  ,draw opacity=1 ] (447.11,123.72) .. controls (449.36,122.38) and (453.17,121.5) .. (457.5,121.5) .. controls (464.4,121.5) and (470,123.74) .. (470,126.5) .. controls (470,129.26) and (464.4,131.5) .. (457.5,131.5) .. controls (453.17,131.5) and (449.36,130.62) .. (447.11,129.28) ;

\end{tikzpicture}\]
\caption{}
    \end{subfigure}
        \begin{subfigure}[c]{\textwidth}
    \centering
\[\begin{tikzpicture}[x=0.75pt,y=0.75pt,yscale=-1,xscale=1]

\draw [color={rgb, 255:red, 245; green, 166; blue, 35 }  ,draw opacity=1 ][line width=2.25]    (137.37,83.01) -- (70,160) ;
\draw [shift={(140,80)}, rotate = 131.19] [color={rgb, 255:red, 245; green, 166; blue, 35 }  ,draw opacity=1 ][line width=2.25]    (12.24,-5.49) .. controls (7.79,-2.58) and (3.71,-0.75) .. (0,0) .. controls (3.71,0.75) and (7.79,2.58) .. (12.24,5.49)   ;
\draw [color={rgb, 255:red, 0; green, 0; blue, 255 }  ,draw opacity=1 ][line width=2.25]    (72.63,83.01) -- (140,160) ;
\draw [shift={(70,80)}, rotate = 48.81] [color={rgb, 255:red, 0; green, 0; blue, 255 }  ,draw opacity=1 ][line width=2.25]    (12.24,-5.49) .. controls (7.79,-2.58) and (3.71,-0.75) .. (0,0) .. controls (3.71,0.75) and (7.79,2.58) .. (12.24,5.49)   ;
\draw [color={rgb, 255:red, 0; green, 0; blue, 255 }  ,draw opacity=1 ]   (223.05,82.82) -- (323.05,182.82) ;
\draw [color={rgb, 255:red, 0; green, 0; blue, 255 }  ,draw opacity=1 ]   (243.05,62.82) -- (343.05,162.82) ;
\draw [color={rgb, 255:red, 0; green, 0; blue, 255 }  ,draw opacity=1 ]   (223.05,82.82) .. controls (217,77.33) and (237,57.33) .. (243.05,62.82) ;
\draw [color={rgb, 255:red, 0; green, 0; blue, 255 }  ,draw opacity=1 ]   (323.05,182.82) .. controls (316,177.33) and (337,158.33) .. (343.05,162.82) ;
\draw [color={rgb, 255:red, 0; green, 0; blue, 255 }  ,draw opacity=1 ]   (323.05,182.82) .. controls (329,190.33) and (351,169.33) .. (343.05,162.82) ;
\draw [color={rgb, 255:red, 245; green, 166; blue, 35 }  ,draw opacity=1 ]   (291.05,154.82) -- (261.05,184.82) ;
\draw [color={rgb, 255:red, 245; green, 166; blue, 35 }  ,draw opacity=1 ]   (251.05,114.82) -- (221.05,144.82) ;
\draw [color={rgb, 255:red, 245; green, 166; blue, 35 }  ,draw opacity=1 ]   (345.05,100.82) -- (315.05,130.82) ;
\draw [color={rgb, 255:red, 245; green, 166; blue, 35 }  ,draw opacity=1 ]   (305.05,60.82) -- (275.05,90.82) ;
\draw [color={rgb, 255:red, 245; green, 166; blue, 35 }  ,draw opacity=1 ]   (305.05,60.82) .. controls (313.05,53.82) and (355.05,93.82) .. (345.05,100.82) ;
\draw [color={rgb, 255:red, 245; green, 166; blue, 35 }  ,draw opacity=1 ]   (221.05,144.82) .. controls (229.05,137.82) and (271.05,177.82) .. (261.05,184.82) ;
\draw [color={rgb, 255:red, 245; green, 166; blue, 35 }  ,draw opacity=1 ] [dash pattern={on 0.84pt off 2.51pt}]  (305.05,60.82) .. controls (297.05,70.82) and (337.05,106.82) .. (345.05,100.82) ;
\draw [color={rgb, 255:red, 245; green, 166; blue, 35 }  ,draw opacity=1 ]   (221.05,144.82) .. controls (213.05,154.82) and (253.05,190.82) .. (261.05,184.82) ;
\draw    (160,120) -- (198,120) ;
\draw [shift={(200,120)}, rotate = 180] [color={rgb, 255:red, 0; green, 0; blue, 0 }  ][line width=0.75]    (10.93,-3.29) .. controls (6.95,-1.4) and (3.31,-0.3) .. (0,0) .. controls (3.31,0.3) and (6.95,1.4) .. (10.93,3.29)   ;
\draw [color={rgb, 255:red, 0; green, 0; blue, 255 }  ,draw opacity=1 ] [dash pattern={on 0.84pt off 2.51pt}]  (223.05,82.82) .. controls (229,90.33) and (251,69.33) .. (243.05,62.82) ;
\draw  [color={rgb, 255:red, 245; green, 166; blue, 35 }  ,draw opacity=1 ] (385.5,175) .. controls (385.5,172.24) and (391.1,170) .. (398,170) .. controls (404.9,170) and (410.5,172.24) .. (410.5,175) .. controls (410.5,177.76) and (404.9,180) .. (398,180) .. controls (391.1,180) and (385.5,177.76) .. (385.5,175) -- cycle ;
\draw  [color={rgb, 255:red, 245; green, 166; blue, 35 }  ,draw opacity=1 ] (395,158) .. controls (395,155.24) and (400.6,153) .. (407.5,153) .. controls (414.4,153) and (420,155.24) .. (420,158) .. controls (420,160.76) and (414.4,163) .. (407.5,163) .. controls (400.6,163) and (395,160.76) .. (395,158) -- cycle ;
\draw  [draw opacity=0] (421.3,127.84) .. controls (417.54,126.98) and (415,125.36) .. (415,123.5) .. controls (415,120.74) and (420.6,118.5) .. (427.5,118.5) .. controls (434.4,118.5) and (440,120.74) .. (440,123.5) .. controls (440,124.22) and (439.62,124.9) .. (438.94,125.52) -- (427.5,123.5) -- cycle ; \draw  [color={rgb, 255:red, 245; green, 166; blue, 35 }  ,draw opacity=1 ] (421.3,127.84) .. controls (417.54,126.98) and (415,125.36) .. (415,123.5) .. controls (415,120.74) and (420.6,118.5) .. (427.5,118.5) .. controls (434.4,118.5) and (440,120.74) .. (440,123.5) .. controls (440,124.22) and (439.62,124.9) .. (438.94,125.52) ;  
\draw  [draw opacity=0] (425.1,107.15) .. controls (425.04,106.94) and (425,106.72) .. (425,106.5) .. controls (425,106.28) and (425.04,106.06) .. (425.1,105.85) -- (437.5,106.5) -- cycle ; \draw  [color={rgb, 255:red, 245; green, 166; blue, 35 }  ,draw opacity=1 ] (425.1,107.15) .. controls (425.04,106.94) and (425,106.72) .. (425,106.5) .. controls (425,106.28) and (425.04,106.06) .. (425.1,105.85) ;  
\draw  [color={rgb, 255:red, 245; green, 166; blue, 35 }  ,draw opacity=1 ] (435,89.5) .. controls (435,86.74) and (440.6,84.5) .. (447.5,84.5) .. controls (454.4,84.5) and (460,86.74) .. (460,89.5) .. controls (460,92.26) and (454.4,94.5) .. (447.5,94.5) .. controls (440.6,94.5) and (435,92.26) .. (435,89.5) -- cycle ;
\draw  [color={rgb, 255:red, 245; green, 166; blue, 35 }  ,draw opacity=1 ] (445,72.5) .. controls (445,69.74) and (450.6,67.5) .. (457.5,67.5) .. controls (464.4,67.5) and (470,69.74) .. (470,72.5) .. controls (470,75.26) and (464.4,77.5) .. (457.5,77.5) .. controls (450.6,77.5) and (445,75.26) .. (445,72.5) -- cycle ;
\draw  [color={rgb, 255:red, 0; green, 0; blue, 255 }  ,draw opacity=1 ] (422.38,127.5) .. controls (422.38,125.71) and (426.01,124.25) .. (430.5,124.25) .. controls (434.99,124.25) and (438.63,125.71) .. (438.63,127.5) .. controls (438.63,129.29) and (434.99,130.75) .. (430.5,130.75) .. controls (426.01,130.75) and (422.38,129.29) .. (422.38,127.5) -- cycle ;
\draw  [draw opacity=0] (428.56,141.69) .. controls (428.44,141.47) and (428.38,141.24) .. (428.38,141) .. controls (428.38,140.76) and (428.44,140.53) .. (428.56,140.31) -- (436.5,141) -- cycle ; \draw  [color={rgb, 255:red, 0; green, 0; blue, 255 }  ,draw opacity=1 ] (428.56,141.69) .. controls (428.44,141.47) and (428.38,141.24) .. (428.38,141) .. controls (428.38,140.76) and (428.44,140.53) .. (428.56,140.31) ;  
\draw  [color={rgb, 255:red, 0; green, 0; blue, 255 }  ,draw opacity=1 ] (437.38,157.5) .. controls (437.38,155.71) and (441.01,154.25) .. (445.5,154.25) .. controls (449.99,154.25) and (453.63,155.71) .. (453.63,157.5) .. controls (453.63,159.29) and (449.99,160.75) .. (445.5,160.75) .. controls (441.01,160.75) and (437.38,159.29) .. (437.38,157.5) -- cycle ;
\draw  [color={rgb, 255:red, 0; green, 0; blue, 255 }  ,draw opacity=1 ] (446.38,174.5) .. controls (446.38,172.71) and (450.01,171.25) .. (454.5,171.25) .. controls (458.99,171.25) and (462.63,172.71) .. (462.63,174.5) .. controls (462.63,176.29) and (458.99,177.75) .. (454.5,177.75) .. controls (450.01,177.75) and (446.38,176.29) .. (446.38,174.5) -- cycle ;
\draw  [color={rgb, 255:red, 0; green, 0; blue, 255 }  ,draw opacity=1 ] (410.38,106.5) .. controls (410.38,104.71) and (414.01,103.25) .. (418.5,103.25) .. controls (422.99,103.25) and (426.63,104.71) .. (426.63,106.5) .. controls (426.63,108.29) and (422.99,109.75) .. (418.5,109.75) .. controls (414.01,109.75) and (410.38,108.29) .. (410.38,106.5) -- cycle ;
\draw  [color={rgb, 255:red, 0; green, 0; blue, 255 }  ,draw opacity=1 ] (401.38,89.5) .. controls (401.38,87.71) and (405.01,86.25) .. (409.5,86.25) .. controls (413.99,86.25) and (417.63,87.71) .. (417.63,89.5) .. controls (417.63,91.29) and (413.99,92.75) .. (409.5,92.75) .. controls (405.01,92.75) and (401.38,91.29) .. (401.38,89.5) -- cycle ;
\draw  [color={rgb, 255:red, 0; green, 0; blue, 255 }  ,draw opacity=1 ] (392.38,72.5) .. controls (392.38,70.71) and (396.01,69.25) .. (400.5,69.25) .. controls (404.99,69.25) and (408.63,70.71) .. (408.63,72.5) .. controls (408.63,74.29) and (404.99,75.75) .. (400.5,75.75) .. controls (396.01,75.75) and (392.38,74.29) .. (392.38,72.5) -- cycle ;
\draw  [draw opacity=0] (430.49,138.81) .. controls (431.98,138.16) and (434.12,137.75) .. (436.5,137.75) .. controls (440.99,137.75) and (444.63,139.21) .. (444.63,141) .. controls (444.63,142.79) and (440.99,144.25) .. (436.5,144.25) .. controls (434.12,144.25) and (431.98,143.84) .. (430.49,143.19) -- (436.5,141) -- cycle ; \draw  [color={rgb, 255:red, 0; green, 0; blue, 255 }  ,draw opacity=1 ] (430.49,138.81) .. controls (431.98,138.16) and (434.12,137.75) .. (436.5,137.75) .. controls (440.99,137.75) and (444.63,139.21) .. (444.63,141) .. controls (444.63,142.79) and (440.99,144.25) .. (436.5,144.25) .. controls (434.12,144.25) and (431.98,143.84) .. (430.49,143.19) ;  
\draw  [draw opacity=0] (427.11,103.72) .. controls (429.36,102.38) and (433.17,101.5) .. (437.5,101.5) .. controls (444.4,101.5) and (450,103.74) .. (450,106.5) .. controls (450,109.26) and (444.4,111.5) .. (437.5,111.5) .. controls (433.17,111.5) and (429.36,110.62) .. (427.11,109.28) -- (437.5,106.5) -- cycle ; \draw  [color={rgb, 255:red, 245; green, 166; blue, 35 }  ,draw opacity=1 ] (427.11,103.72) .. controls (429.36,102.38) and (433.17,101.5) .. (437.5,101.5) .. controls (444.4,101.5) and (450,103.74) .. (450,106.5) .. controls (450,109.26) and (444.4,111.5) .. (437.5,111.5) .. controls (433.17,111.5) and (429.36,110.62) .. (427.11,109.28) ;  
\draw  [color={rgb, 255:red, 0; green, 0; blue, 255 }  ,draw opacity=1 ] (428.38,141) .. controls (428.38,139.21) and (432.01,137.75) .. (436.5,137.75) .. controls (440.99,137.75) and (444.63,139.21) .. (444.63,141) .. controls (444.63,142.79) and (440.99,144.25) .. (436.5,144.25) .. controls (432.01,144.25) and (428.38,142.79) .. (428.38,141) -- cycle ;
\draw  [draw opacity=0] (428.33,143.5) .. controls (426.17,144.99) and (422.13,146) .. (417.5,146) .. controls (410.6,146) and (405,143.76) .. (405,141) .. controls (405,138.24) and (410.6,136) .. (417.5,136) .. controls (422.13,136) and (426.17,137.01) .. (428.33,138.5) -- (417.5,141) -- cycle ; \draw  [color={rgb, 255:red, 245; green, 166; blue, 35 }  ,draw opacity=1 ] (428.33,143.5) .. controls (426.17,144.99) and (422.13,146) .. (417.5,146) .. controls (410.6,146) and (405,143.76) .. (405,141) .. controls (405,138.24) and (410.6,136) .. (417.5,136) .. controls (422.13,136) and (426.17,137.01) .. (428.33,138.5) ;  
\draw  [draw opacity=0] (429.9,140.35) .. controls (429.96,140.56) and (430,140.78) .. (430,141) .. controls (430,141.22) and (429.96,141.44) .. (429.9,141.65) -- (417.5,141) -- cycle ; \draw  [color={rgb, 255:red, 245; green, 166; blue, 35 }  ,draw opacity=1 ] (429.9,140.35) .. controls (429.96,140.56) and (430,140.78) .. (430,141) .. controls (430,141.22) and (429.96,141.44) .. (429.9,141.65) ;  
\draw  [draw opacity=0] (436.75,126.87) .. controls (434.46,127.87) and (431.16,128.5) .. (427.5,128.5) .. controls (426.33,128.5) and (425.2,128.44) .. (424.13,128.32) -- (427.5,123.5) -- cycle ; \draw  [color={rgb, 255:red, 245; green, 166; blue, 35 }  ,draw opacity=1 ] (436.75,126.87) .. controls (434.46,127.87) and (431.16,128.5) .. (427.5,128.5) .. controls (426.33,128.5) and (425.2,128.44) .. (424.13,128.32) ;

\end{tikzpicture}\]
\caption{}
    \end{subfigure}
    \caption{(A): A positive crossing represents a tube threaded through another tube in $\mathbb \R^4$. In the {\em broken surface diagram} a surface is broken where it lies below the other in the fourth coordinate. (B): Wiring diagram crossings (also known as {\em virtual crossings}) represent tubes crossing in front of one another, with no interaction.}
    \label{fig:4DCrossing}
\end{figure}

\begin{figure}
\begin{subfigure}[b]{0.3\textwidth}
\[\begin{tikzpicture}
\draw[thick] (0,1) ellipse (.5cm and .1cm);
\draw[thick, dashed] (0,0) ellipse (.5cm and .1cm);
\begin{scope}
    \clip (-1,0) rectangle (1,-.6);
    \draw[thick] (0,0) ellipse (.5cm and .1cm);
\end{scope}
\draw[thick, dashed] (1,-1) ellipse (.5cm and .1cm);
\begin{scope}
    \clip (0,-1) rectangle (2,-1.6);
    \draw[thick] (1,-1) ellipse (.5cm and .1cm);
\end{scope}
\draw[thick, dashed] (-1,-1) ellipse (.5cm and .1cm);
\begin{scope}
    \clip (-2,-1) rectangle (0,-1.6);
    \draw[thick] (-1,-1) ellipse (.5cm and .1cm);
\end{scope}
\draw[thick] (-.5,1)--(-.5,0)--(-1.5,-1);
\draw[thick] (.5,1)--(.5,0)--(1.5,-1);
\draw[thick] (-.5,-1)--(.5,0);

\draw[thick] (-.1,-.4)--(-.5,0);
\draw[thick] (.1,-.6)--(.5,-1);

 \end{tikzpicture}\]
\end{subfigure}
\raisebox{2mm}{\begin{subfigure}[b]{0.3\textwidth}
\[\begin{tikzpicture}
\draw[thick] (-1,1) ellipse (.5cm and .1cm);
\draw[thick] (-1.5,1)to [in=140, out=270] (-.29,-.5);
\draw[thick] (-.5,1)to [in=130, out=270] (.7,-.49);
\draw[thick, dashed] (.2,-.5) ellipse (.5cm and .1cm);
\begin{scope}
    \clip (-1,-.5) rectangle (1,-1);
    \draw[thick] (.2,-.5) ellipse (.5cm and .1cm);
\end{scope}

\draw[red] (-1.5,-.7)to [in=230, out=30] (-.6,.2);
\draw[red, fill=red] (-.6,.2) circle (.05cm);
 \end{tikzpicture}\]
\end{subfigure}}
    \caption{Examples of the type of  vertices permitted in $w$-foams.}
    \label{fig:flying tubes}
\end{figure}
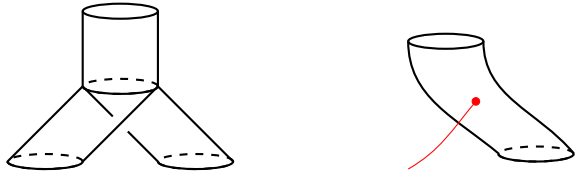

The wheeled prop of $w$-foams is equipped with several {\em auxiliary operations} in addition to circuit algebra compositions. These are {\em orientation switches}, {\em unzips} (doubling a tube along framing), and {\em punctures} (puncture and retract a 2-dimensional tube onto a 1-dimensional string). We suppress the details here (see \cite[Section 4.1.3]{BND:WKO2} and \cite[Section 2.3]{BND:WKO3} for more), but we present a full description of the associated graded versions of these operations in Definition~\ref{def: arrow diagram circuit alg}.


As with braids, every $w$-foam has an underlying $w$-foam \emph{skeleton}. The skeleton of a $w$-foam is the combinatorial data of its features, the \emph{vertices, caps and strands}, and how they are connected, but not the embedding (i.e. crossing) information. In other words, the skeleton of a $w$-foam is obtained by ``flattening'' all of its crossings. The collection of all $w$-foam skeleta is itself a tensor category, as follows: 

\begin{figure}[h]
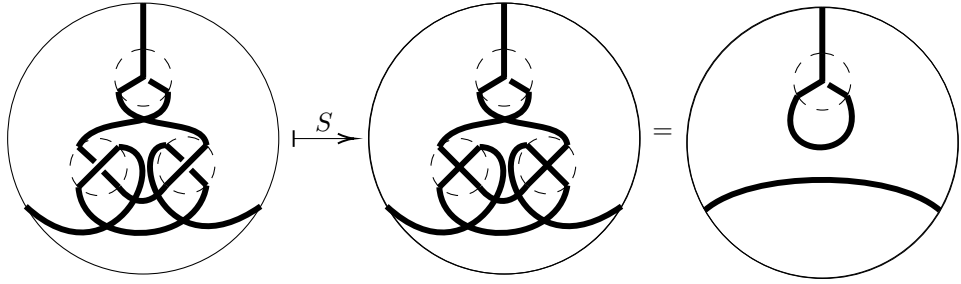

\centering
\[
\]
    \caption{An example of the projection from a $w$-foam to its skeleton.}
    \label{fig:wfoamskeleton}
\end{figure}

\begin{remark}
The {\em wheeled prop} of $w$-foam skeleta is the free wheeled prop given by:
 \[\mathcal{S}=\left<\vertex , \pv, \ppv,
\bcap \right>\] 
 Here generators are understood to come in all possible orientations, and skeleta are equipped with the \emph{orientation switch}, {\em unzip}, and {\em puncture} operations. There is a skeleton projection map $S: \wf \to \mathcal{S}$, given by $S(\overcrossing)=S(\undercrossing)=\virtualcrossing$ (see Figure~\ref{fig:wfoamskeleton}).
\end{remark}

\subsection{The associated graded structure, tangential derivations, and cyclic words}
The analogue of the group ring for $w$-foams is the linear extension of the circuit algebra $\wf$, where morphisms are defined by: $$\mathbb{Q} [\wf] = \bigsqcup_{s\in \mathcal{S}} \mathbb{Q} [\wf](s).$$ Here $\mathbb{Q}[\wf](s)$ is the $\mathbb{Q}$-vector space of formal linear combinations of $w$-foams with skeleton $s\in\calS$. That is, $\mathbb{Q}[\wf](s)$ includes sums of the form $\sum_{S(T_i)=s}\alpha_iT_i,$ $T_i\in \wf(s)$.  
The associated graded structure of $\wf$ is the direct product of successive quotients of powers of its augmentation ideal\footnote{Note, this equivalent to defining the pro-unipotent, or rational, completion of the wheeled prop of $w$-foams. For more details on this, see \cite[Section 2.3]{DHR21}.} \[\calA:=\Pi_0^\infty \calI^n/\calI^{n+1}.\] 
This associated graded $\calA$ is a tensor category whose morphisms have a presentation in terms of a type of generalised Jacobi diagrams which we call \emph{arrow diagrams}. In analogy with chord diagrams, arrow diagrams are closely related to the elements of universal enveloping algebras of Lie algebras of {\em tangential derivations}, and as we will see, homomorphic expansions of $w$-foams translate to solutions to the {\em Kashiwara-Vergne} equations.

\subsubsection{Tangential derivations and automorphisms}\label{sec:tder}
Let $\lie_n$ denote degree completion of the \emph{free Lie algebra} on $n$-letters, $x_1,\ldots, x_n$.
A derivation $u:\lie_n\rightarrow \lie_n$ is a linear map satisfying 
$u([a,b])=[u(a),b]+ [a,u(b)].$  A \textit{tangential derivation} is a derivation $\mathbf a =(a_1,\dots,a_n)$, $a_i\in \mathfrak{lie}_n$, for which $\mathbf{a}(x_i) = [x_i,a_i]$ for $i=1,\dots,n$ (\cite[Definition 3.2, 3.5]{AT12}).  Tangential derivations of $\lie_n$ form a Lie algebra, denoted $\mathfrak{tder}_n$, with bracket $[\mathbf{a},\mathbf{b}](x_i)=\mathbf{a}(\mathbf{b}(x_i))-\mathbf{b}(\mathbf{a}(x_i))=\mathbf{a}([x_i,b_i])-\mathbf{b}([x_i,a_i])$ (\cite[Proposition 3.4]{AT12}). The Lie group of tangential derivations is called \emph{tangential automorphisms}, \[\TAut_n=\exp(\mathfrak{tder}_n).\] These are basis conjugating, or tangential, automorphisms of $\exp(\mathfrak{lie}_n)$. 

\begin{example}\label{example: t}
An important example of a tangential derivation is the \emph{braid derivation} $\mathbf{t}^{1,2} = (y,x) \in \mathfrak{tder}_2$. By definition, $\mathbf{t}(x) = [x,y]$ and $\mathbf{t}(y) = [y,x]$.  It follows that $\mathbf{t}(x+y) = [x,y] + [y,x] = 0$. (In general, the set of derivations with this property is called {\em special}, and denoted $\mathfrak{sder}_n$.) 

The analogously defined derivations $\mathbf{t}^{i,j}$ generate a Lie subalgebra of $\tder_n$ isomorphic to the Lie algebra of infinitesimal braids $\mathfrak{t}_n$ \cite[Proposition 3.11]{AT12}.
\end{example}

A solution to the {\em generalised Kashiwara-Vergne problem} is a tangential automorphism $F$ of $\lie_2$, which satisfies the two {\em KV equations} \cite[Section 5.3]{AT12}:
\begin{multline}\label{eq:SolKV}\tag{KV}
\text{SolKV}:= \Big\{F\in \TAut_2  \Big|\, F(e^xe^y)=e^{x+y}, \ \text{and} \\ 
   J(F)=\trace\Big(r(x+y)-r(x)-r(y)\Big) \ \text{for some} \ r \in u^2\mathbb{Q}[[u]] \ \Big\}.
\end{multline}

To explain the meaning of the second KV equation, note first that, as a completed Hopf algebra, the universal enveloping algebra of $\lie_n$ can be identified with the degree completed free associative algebra $\widehat{\mathfrak{ass}}_n:=\mathbb{Q}\left<\left<x_1,\ldots,x_n\ \right>\right>$. 
The (graded) vector space of \emph{cyclic words}\footnote{What we call $\operatorname{cyc}_n$ is denoted by $\mathfrak{tr}_n$ in \cite{AT12}.} is the linear quotient  $$\cyc_n := \widehat{\mathfrak{ass}}_n/\left<ab-ba\mid \forall a,b\in\widehat{\mathfrak{ass}}_n\right> = \widehat{\mathfrak{ass}}_n/[\widehat{\mathfrak{ass}}_n,\widehat{\mathfrak{ass}}_n].$$  
Thus, there is a natural trace map $\operatorname{tr}: \widehat{\mathfrak{ass}}_n \to \cyc_n$.

\begin{example} \cite[Example 2.1]{AT12}
The vector space $\operatorname{cyc}_1$ is isomorphic to the space $x\mathbb{K}[[x]]$ of formal power series with no constant term. It is spanned by $\operatorname{tr}(x^k)$ for $k \geq 1$. 
\end{example}

The natural action of $\tder_n$ on $\widehat{\mathfrak{ass}}_n$ induces an action of $\tder_n$ on $\cyc_n$ via the trace map. The \emph{divergence} cocycle $\operatorname{div}: \mathfrak{tder}_n \to \operatorname{cyc}_n$ is defined by
\begin{equation}
    u = (a_1,\dots,a_n) \mapsto \sum_{k=1}^n \operatorname{tr}(x_k(\partial_ka_k)),
\end{equation}
where the operator $\partial_k$ acts on $\lie_n \subseteq \widehat{\ass}_n$ by deleting words that do not end in $x_k$, and deleting the letter $x_k$ from words that do. Divergence satisfies the 1-cocycle property $\divcc([u,v])=u\cdot \divcc(v) - v\cdot \divcc(u)$ and thus extends this trace action to a representation of $\tder_n$ on $\cyc_n$. This representation integrates to give a representation of tangential automorphisms on cyclic words. 

\begin{definition}\cite[Section 5.1]{AT12} The non-commutative Jacobian map 
\begin{align*}
    J: \mathsf{TAut}_n \to \operatorname{cyc}_n
\end{align*}
is defined by setting
\begin{align}
    J(1) = 0\;, \quad \frac{d}{dt}\Bigr\rvert_{t=0}j\left(e^{tu}g \right) = \operatorname{div}(u) + u \cdot J(g)
    \end{align}
for $g \in \mathsf{TAut}_n$ and $u \in \mathfrak{tder}_n$. 
\end{definition}

In summary, solutions to the generalised Kashiwara-Vergne problem are tangential automorphisms of $\lie_2$ which take $F(e^xe^y)=e^{x+y}$ and whose Jacobian representation satisfies a particular constraint in $\cyc_1$. We are now ready to describe the morphisms in the tensor category of arrow diagrams which capture similar behaviour.

\subsubsection{Arrow diagrams}\label{sec:arrow diagrams}
 The associated graded structure of $w$-foams is a tensor category, $\arrows$, whose morphisms are combinatorially described as \emph{arrow diagrams}. In analogy with the relationship between chord diagrams and the Lie algebras of infinitesimal braids, certain morphism spaces in $\arrows$ are closely related to tangential automorphisms:  \[\arrows(\uparrow_n)\cong U(\tder_n\oplus \mathfrak{a}_n\ltimes \cyc_n).\] 

Morphisms in the two-coloured wheeled prop of arrow diagrams can be formally described as linear combinations of {\em $w$-Jacobi diagrams} on $w$-foam skeleta.
\begin{definition}\label{def:wJacobiDiag}
A $w$-\emph{Jacobi diagram} on a $w$-foam skeleton $s\in\calS$ is a (possibly infinite) linear combination of directed uni-trivalent {\em arrow graphs} satisfying the following conditions: 
\begin{enumerate}
\item trivalent arrow vertices are equipped with a cyclic orientation;
\item each trivalent arrow vertex is incident to two incoming and one outgoing arrow edge;
\item univalent vertices are attached to the chosen skeleton $s$, considered up to orientation preserving diffeomorphism of $s$. (That is, only the attachment order matters, not the specific locations along $s$.) 
\end{enumerate}

Such diagrams are considered up to the STU, TC, VI, CP, TF and RI relations shown in Figure~\ref{fig:JacobiRelns}. The {\em degree} of a $w$-Jacobi diagram is half the number of (trivalent and univalent) vertices of the arrow graph.
\end{definition}

\begin{figure}
\[ 
\begin{tikzpicture}[scale=.75]

\draw[ ultra thick,->](-.75,0)--(.75,0);
\draw[thick,dotted,blue,->](0,1)--(0,0);
\draw[thick,dotted,blue,->](-.75,1.75)--(0,1);
\draw[thick,dotted,blue,->](.75,1.75)--(0,1);

\node[] at (1.3,1){$=$};
\node[] at (1.3,0.5)[font=\footnotesize] {STU1};

\begin{scope}[xshift=2.5cm]
\draw[ultra thick,->](-.75,0)--(.75,0);
\draw[thick,dotted,blue,->](-.75,1.75)--(-.25,0);
\draw[thick,dotted,blue,->](.75,1.75)--(.25,0);

\node[] at (1.4,1){$-$};
\end{scope}

\begin{scope}[xshift=5cm]
\draw[ultra thick,->](-.75,0)--(.75,0);
\draw[thick,dotted,blue,->](-.75,1.75)--(.25,0);
\draw[thick,dotted,blue,->](.75,1.75)--(-.25,0);
\end{scope}

\begin{scope}[xshift=8cm]
\draw[ultra thick,->](-.75,0)--(.75,0);
\draw[thick,dotted,blue,<-](0,1)--(0,0);
\draw[thick,dotted,blue,->](-.75,1.75)--(0,1);
\draw[thick,dotted,blue,<-](.75,1.75)--(0,1);

\node[] at (1.3,1){$=$};
\node[] at (1.3,0.5)[font=\footnotesize] {STU2};

\begin{scope}[xshift=2.5cm]
\draw[ultra thick,->](-.75,0)--(.75,0);
\draw[thick,dotted,blue,->](-.75,1.75)--(-.25,0);
\draw[thick,dotted,blue,<-](.75,1.75)--(.25,0);

\node[] at (1.4,1){$-$};

\end{scope}

\begin{scope}[xshift=5cm]
\draw[ultra thick,->](-.75,0)--(.75,0);
\draw[thick,dotted,blue,->](-.75,1.75)--(.25,0);
\draw[thick,dotted,blue,<-](.75,1.75)--(-.25,0);
\end{scope}

\end{scope}

\begin{scope}[xshift=16cm]
\draw[ultra thick,->](-.75,0)--(.75,0);
\draw[thick,dotted,blue,<-](-.75,1.75)--(-.25,0);
\draw[thick,dotted,blue,<-](.75,1.75)--(.25,0);

\node[] at (1.4,1){$=$};
\node[] at (1.35,0.5)[font=\footnotesize] {TC};

\begin{scope}[xshift=2.5cm]
\draw[ultra thick,->](-.75,0)--(.75,0);
\draw[thick,dotted,blue,<-](-.75,1.75)--(.25,0);
\draw[thick,dotted,blue,<-](.75,1.75)--(-.25,0);
\end{scope}

\end{scope}
\end{tikzpicture}
\]

\[\begin{tikzpicture}[scale=.75]

\node at (-.75,0){$\pm$};
\draw[thick](-.75,-1)--(0,0)--(0,1)(.75,-1)--(0,0);
\draw[thick,->,blue,dotted] (0,.9)to [out=225, in=0] (-.75,.75);

\node[xshift=1.4cm] at (-.75,0){$\pm$};
\draw[thick,xshift=2cm](-.75,-1)--(0,0)--(0,1)(.75,-1)--(0,0);
\draw[thick,->,blue,dotted,xshift=2cm] (-.25,-.25)to [out=100, in=0] (-.75,.75);

\node[xshift=3cm] at (-.75,0){$\pm$};
\draw[thick,xshift=4cm](-.75,-1)--(0,0)--(0,1)(.75,-1)--(0,0);
\draw[thick,->,dotted,blue,xshift=4cm] (.4,-.5)to [out=190, in=0] (-.75,.75);

\node[]at (5,0){$=$ 0, };
\node[]at (4.9,0.5) [font=\footnotesize]{VI };
\node[]at (6.5,0){ and };

\node[xshift=6.5cm] at (-.75,0){$\pm$};
\draw[thick,xshift=9cm](-.75,-1)--(0,0)--(0,1)(.75,-1)--(0,0);
\draw[thick,->,dotted,blue,xshift=9cm](-.75,.75) to [out=0, in=190] (0,.9);

\node[xshift=8cm] at (-.75,0){$\pm$};
\draw[thick,xshift=11cm](-.75,-1)--(0,0)--(0,1)(.75,-1)--(0,0);
\draw[thick,->,dotted,blue,xshift=11cm](-.75,.75) to [out=0, in=110] (-.25,-.25);

\node[xshift=9.4cm] at (-.75,0){$\pm$};
\draw[thick,xshift=13cm](-.75,-1)--(0,0)--(0,1)(.75,-1)--(0,0);
\draw[thick,->,dotted,blue,xshift=13cm](-.75,.75) to [out=0, in=200] (.4,-.5);

\node[]at (14,0){$=$ 0 };
\node[]at (13.9,0.5)[font=\footnotesize]{VI};

\end{tikzpicture}\]

\[\begin{tikzpicture}[scale=.75]

\draw[ultra thick,->] (-1,-1)--(-1,1);
\draw[thick, dotted,blue](-1,-.5) to[out=0, in=270] (0,0);
\draw[thick,dotted,blue,->](0,0) to[out=90, in=360](-1,.5);

\node[] at (.4,0.1) {$=$};
\node[] at (.4,0.6)[font=\footnotesize] {RI};

\begin{scope}[xshift= 2cm]
\draw[ultra thick,->] (-1,-1)--(-1,1);
\draw[thick, dotted,blue,->](0,0) to[out=270, in=3600] (-1,-.5);
\draw[thick, dotted,blue](-1,.5) to[out=0, in=90](0,0);
\end{scope}


\begin{scope}[xshift=5cm]
\draw[fill=black] (-1,-.95) circle (.17cm);
\draw[ultra thick,] (-1,-1)--(-1,1);
\draw[thick,dotted,blue,<-](-1,0) --(.4,0);

\node[] at (.95,0.1) {$=$ 0};
\node[] at (1,0.6) [font=\footnotesize]{CP};
\end{scope}

\begin{scope}[xshift=9cm]
\draw[red] (-1,-1)--(-1,1);
\draw[thick,dotted,blue,->](-1,0) --(.4,0);

\node[] at (.95,0.1) {$=$ 0};
\node[] at (1,0.6) [font=\footnotesize]{TF};
\end{scope}

\end{tikzpicture}\]

\[ 
 \begin{tikzpicture}[scale=.75]
 \draw[thick,dotted,blue,->](-1,1)--(-.02,.02);
 \draw[thick,dotted,blue,->](1,1)--(0.02,0.02);
 \draw[thick,dotted,blue,->](0,0)--(0,-1);

 \node[] at (1.5,0){$=$ \ \ $-$};
 \node[] at (1.4,0.5)[font=\footnotesize] {AS};

 \begin{scope}[xshift=3cm]
  \draw[thick,dotted,blue,->](-1,1)to [out=360,in=15](.02,.02);
 \draw[thick,dotted,blue,->](1,1)to [out=180,in=160] (-0.02,0.02);
 \draw[thick,dotted,blue,->](0,0)--(0,-1);
 \end{scope}

 \begin{scope}[xshift=6.5cm]

  \draw[thick,dotted,blue,->](-1,1)--(-.02,.02);
 \draw[thick,dotted,blue,->](1,1)--(0.02,0.02);
 \draw[thick,dotted,blue,->](0,0)--(0,-.5);
  \draw[thick,dotted,blue,->](-1.01,-1.4)--(-.02,-.6);
   \draw[thick,dotted,blue,->](.02,-.6)--(1.01,-1.4);
     \node[] at (1.25,0){$=$};
     \node[] at (1.25,0.5)[font=\footnotesize] {IHX};

 \end{scope}
  \begin{scope}[xshift=9.5cm]

  \draw[thick,dotted,blue,->](-1,1)--(-.51,-.23);
 \draw[thick,dotted,blue,->](1,1)--(.52,-.23);
 \draw[thick,dotted,blue,->](-.5,-.25)--(.5,-.25);
  \draw[thick,dotted,blue,->](-1,-1.4)--(-.51,-.28);
   \draw[thick,dotted,blue,->](.5,-.28)--(1.01,-1.4);
        \node[] at (1.65,-.2){$-$};

 \end{scope}

  \begin{scope}[xshift=12.5cm]

  \draw[thick,dotted,blue,->](-1,1)--(.7,-.7);
 \draw[thick,dotted,blue,->](1,1)--(-.7,-.7);

 \draw[thick,dotted,blue,->](-.75,-.75)--(.7,-.75);

  \draw[thick,dotted,blue,->](-1.25,-1.25)--(-.75,-.75);
   \draw[thick,dotted,blue,->](.75,-.75)--(1.25,-1.25);
 \end{scope}

\end{tikzpicture}\]

\caption{The relations for $w$-Jacobi diagrams.}\label{fig:JacobiRelns}
\end{figure}
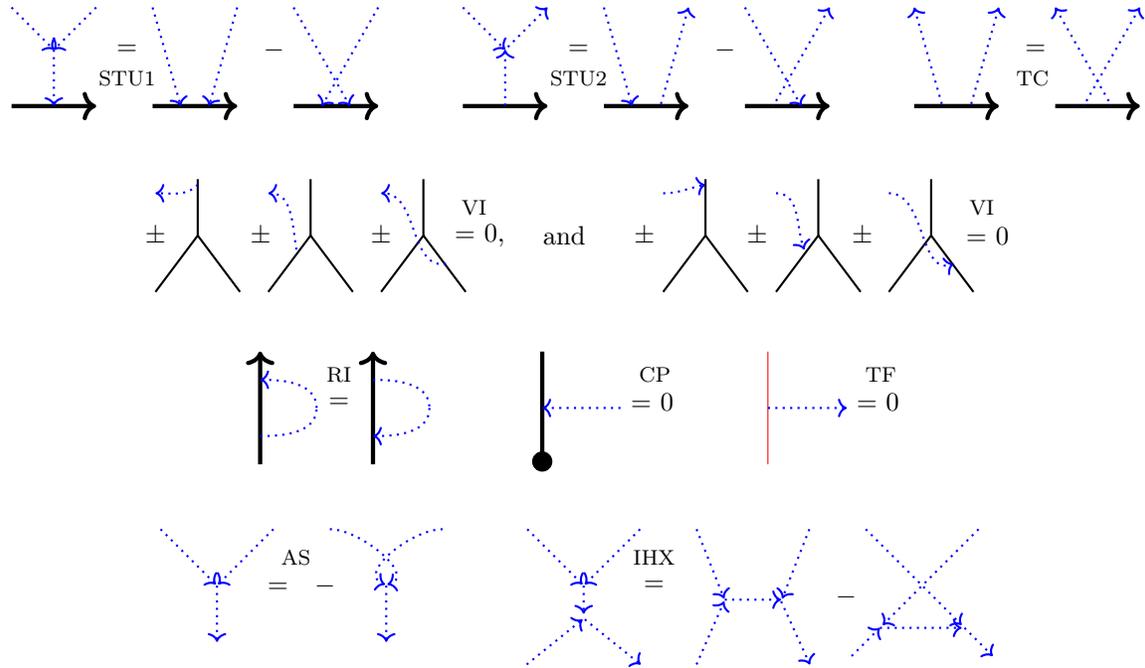

The STU relations for arrow diagrams are the oriented version of the classical STU relation of Jacobi diagrams, and ensures that the trivalent vertices ``behave as Lie brackets\footnote{The relationship between the STU relations and Lie brackets can be made precise using the notion of weight systems, which is not required for this paper. See for example \cite[Section 2.4]{BN_Survey_Knot_Invariants}}''. In particular, the STU relation implies the AS (anti-symmetry) and IHX relations also shown in Figure~\ref{fig:JacobiRelns}. 
An example of a $w$-Jacobi diagram is shown in Figure~\ref{fig: examples of diagrams}. In figures the oriented uni-trivalent graph (arrow graph) is drawn in dashed lines, the $w$-foam skeleton has tubes shown as thick black lines and strings as thin red lines. 

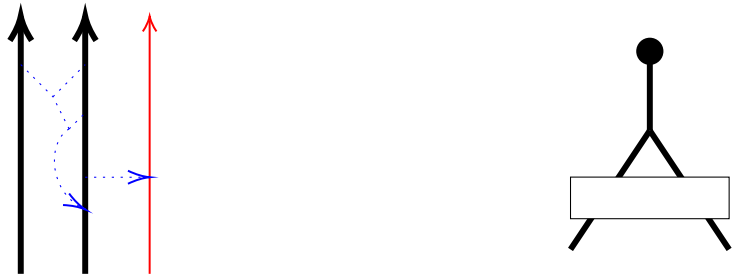
\begin{figure}[h]
\centering
    \begin{subfigure}[c]{0.45\textwidth}
    \centering
\[\begin{tikzpicture}[x=0.75pt,y=0.75pt,yscale=-1,xscale=1]

\draw [line width=2.25]    (70,44) -- (70,170) ;
\draw [shift={(70,40)}, rotate = 90] [color={rgb, 255:red, 0; green, 0; blue, 0 }  ][line width=2.25]    (12.24,-5.49) .. controls (7.79,-2.58) and (3.71,-0.75) .. (0,0) .. controls (3.71,0.75) and (7.79,2.58) .. (12.24,5.49)   ;
\draw [line width=2.25]    (102.5,44) -- (102.5,170) ;
\draw [shift={(102.5,40)}, rotate = 90] [color={rgb, 255:red, 0; green, 0; blue, 0 }  ][line width=2.25]    (12.24,-5.49) .. controls (7.79,-2.58) and (3.71,-0.75) .. (0,0) .. controls (3.71,0.75) and (7.79,2.58) .. (12.24,5.49)   ;
\draw [color={rgb, 255:red, 255; green, 0; blue, 0 }  ,draw opacity=1 ][line width=0.75]    (135,42) -- (135,170) ;
\draw [shift={(135,40)}, rotate = 90] [color={rgb, 255:red, 255; green, 0; blue, 0 }  ,draw opacity=1 ][line width=0.75]    (7.65,-3.43) .. controls (4.86,-1.61) and (2.31,-0.47) .. (0,0) .. controls (2.31,0.47) and (4.86,1.61) .. (7.65,3.43)   ;
\draw [color={rgb, 255:red, 0; green, 0; blue, 255 }  ,draw opacity=1 ] [dash pattern={on 0.84pt off 2.51pt}]  (102.5,121.25) -- (133,121.25) ;
\draw [shift={(135,121.25)}, rotate = 180] [color={rgb, 255:red, 0; green, 0; blue, 255 }  ,draw opacity=1 ][line width=0.75]    (10.93,-3.29) .. controls (6.95,-1.4) and (3.31,-0.3) .. (0,0) .. controls (3.31,0.3) and (6.95,1.4) .. (10.93,3.29)   ;
\draw [color={rgb, 255:red, 0; green, 0; blue, 255 }  ,draw opacity=1 ] [dash pattern={on 0.84pt off 2.51pt}]  (70,64.38) -- (86.25,80.63) ;
\draw [color={rgb, 255:red, 0; green, 0; blue, 255 }  ,draw opacity=1 ] [dash pattern={on 0.84pt off 2.51pt}]  (86.25,80.63) -- (102.5,64.38) ;
\draw [color={rgb, 255:red, 0; green, 0; blue, 255 }  ,draw opacity=1 ] [dash pattern={on 0.84pt off 2.51pt}]  (94.38,96.88) -- (102.5,88.75) ;
\draw [color={rgb, 255:red, 0; green, 0; blue, 255 }  ,draw opacity=1 ] [dash pattern={on 0.84pt off 2.51pt}]  (94.38,96.88) -- (86.25,80.63) ;
\draw [color={rgb, 255:red, 0; green, 0; blue, 255 }  ,draw opacity=1 ] [dash pattern={on 0.84pt off 2.51pt}]  (94.38,96.88) .. controls (83.74,102.79) and (83.42,127.43) .. (100.83,136.69) ;
\draw [shift={(102.5,137.5)}, rotate = 208] [color={rgb, 255:red, 0; green, 0; blue, 255 }  ,draw opacity=1 ][line width=0.75]    (10.93,-3.29) .. controls (6.95,-1.4) and (3.31,-0.3) .. (0,0) .. controls (3.31,0.3) and (6.95,1.4) .. (10.93,3.29)   ;

\end{tikzpicture}
\]
    \end{subfigure}
\begin{subfigure}[c]{0.45\textwidth}
    \centering
\[\begin{tikzpicture}[x=0.75pt,y=0.75pt,yscale=-1,xscale=1]

\draw [line width=2.25]    (90,60) -- (90,100) ;
\draw [shift={(90,60)}, rotate = 90] [color={rgb, 255:red, 0; green, 0; blue, 0 }  ][fill={rgb, 255:red, 0; green, 0; blue, 0 }  ][line width=2.25]      (0, 0) circle [x radius= 5.36, y radius= 5.36]   ;
\draw [line width=2.25]    (90,100) -- (50,160) ;
\draw [line width=2.25]    (90,100) -- (130,160) ;
\draw  [fill={rgb, 255:red, 255; green, 255; blue, 255 }  ,fill opacity=1 ] (50,123.5) -- (130,123.5) -- (130,144.5) -- (50,144.5) -- cycle ;

\end{tikzpicture}
\]
    \end{subfigure}
\caption{On the left is an example of a $w$-Jacobi diagram. Sometimes, we denote an arbitrary $w$-Jacobi diagram located on a certain part of a skeleton by a box, as shown on the right. } \label{fig: examples of diagrams} 
\end{figure}

We have the following finitely presented tensor category of arrow diagrams: 

\begin{definition}\label{def: arrow diagram circuit alg}
The wheeled prop of \emph{arrow diagrams} is:
\[\arrows := \left\langle \rarrow, \rightarrowp, \bcap, \vertex,\negvertex, \pv,\ppv \mid \text{4T, TC, VI, CP, TF, RI}.\right\rangle\] Here the non-skeleton generators are {\em arrows}, drawn in dotted lines. As before, all generators are understood to occur in all possible strand orientations. The 4T relation is shown in Figure~\ref{fig:ArrowRelns}. The {\em degree} of an arrow diagram is given by the number of arrows.
\end{definition}

Morphisms of $\arrows$ are simply called {\em arrow diagrams}. There is an isomorphism between morphism spaces of $\arrows$ and $w$-Jacobi diagrams on the same skeleton, analogous to the classical isomorphism between chord diagrams and Jacobi diagrams \cite[Theorem 6]{BN_Survey_Knot_Invariants}, sometimes called the ``Bracket-Rise Theorem''. In short, arrow diagrams are naturally also $w$-Jacobi diagrams, and this map turns out to descend to an isomorphism. In particular, the STU relations imply the 4T relation. For more detail see \cite[Around Theorem 3.4]{DHR21}, \cite[Theorem 3.13]{BND:WKO2} and \cite[Theorem 3.13]{BND:WKO1}.

\begin{figure}[h]

\[\begin{tikzpicture}[scale=.75]

\draw[ultra thick,->] (-1,-1)--(-1,1);
\draw[ultra thick,->] (0,-1)--(0,1);
\draw[thick,->] (1,-1)--(1,1);
\draw[dotted, thick,blue,->](-1,.5)--(-.07,.5);
\draw[dotted,thick,blue,->] (0,-.5)--(1,-.5);

\node[] at (1.75,0){+};

\begin{scope}[xshift=3.5cm]
\draw[ultra thick,->] (-1,-1)--(-1,1);
\draw[ultra thick,->] (0,-1)--(0,1);
\draw[thick,->] (1,-1)--(1,1);
\draw[dotted,thick,blue,->](-1,.5)--(1.,.5);
\draw[dotted,thick,blue,->] (0,-.5)--(1,-.5);
\node[] at (1.75,0){$=$};
\node[] at (1.75,0.6)[font=\footnotesize]{4T};

\end{scope}

\begin{scope}[xshift=7cm]
\draw[ultra thick,->] (-1,-1)--(-1,1);
\draw[ultra thick,->] (0,-1)--(0,1);
\draw[thick,->] (1,-1)--(1,1);
\draw[dotted,thick,blue,->](0,.5)--(1.,.5);
\draw[dotted,thick,blue,->] (-1,-.5)--(0,-.5);
\node[] at (1.75,0){+};
\end{scope}

\begin{scope}[xshift=10.5cm]
\draw[ultra thick,->] (-1,-1)--(-1,1);
\draw[ultra thick,->] (0,-1)--(0,1);
\draw[thick,->] (1,-1)--(1,1);
\draw[dotted,thick,blue,->](0,.5)--(1.,.5);
\draw[dotted,thick,blue,->] (-1,-.5)--(1,-.5);
\end{scope}
\end{tikzpicture}\]

\caption{The 4T relation for arrow diagrams.}\label{fig:ArrowRelns}
\end{figure}
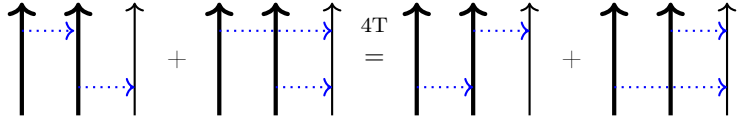

Arrow diagrams are equipped with the following \emph{operations} beyond the wheeled prop structure: 
\begin{enumerate}
	\item  The \emph{orientation switch} $S_e: \arrows(s) \rightarrow \arrows(S_e(s))$ reverses the direction of a strand $e$ of $s$ and multiplies each arrow diagram by $$(-1)^{\#(\text{arrow heads and tails ending on }e)}.$$
	\item The \emph{adjoint} operation $A_e: \arrows(s) \rightarrow \arrows(A_e(s))$ reverses the direction of a strand $e$ in $s$ and multiplies each arrow diagram by
	$$(-1)^{\#(\text{arrow heads on } e)}.$$
	\item The \emph{unzip} (for a strand $e$ ending in two vertices) and \emph{disc unzip} (for a strand $e$ ending in a vertex and a cap) operations are both denoted $u_e: \arrows(s) \rightarrow \arrows(u_e(s))$. The operation doubles the strand $e$, removes the vertex (vertices) at the end(s) of $e$, and maps each arrow ending on $e$ to a sum of two new arrows, each of which ends on one of the two new strands created. This is shown in Figure~\ref{fig:unzip_arrow}
	\begin{figure}[h]
	    \centering
	    \[\begin{tikzpicture}[x=0.75pt,y=0.75pt,yscale=-1,xscale=1]

\draw [line width=2.25]    (44,70) -- (58,91) ;
\draw [line width=2.25]    (58,91) -- (72,70) ;
\draw [line width=2.25]    (58,119) -- (58,91) ;
\draw [line width=2.25]    (44,140) -- (58,119) ;
\draw [line width=2.25]    (58,119) -- (72,140) ;
\draw [color={rgb, 255:red, 0; green, 0; blue, 255 }  ,draw opacity=1 ] [dash pattern={on 0.84pt off 2.51pt}]  (32,105) -- (58,105) ;
\draw [shift={(30,105)}, rotate = 0] [color={rgb, 255:red, 0; green, 0; blue, 255 }  ,draw opacity=1 ][line width=0.75]    (10.93,-3.29) .. controls (6.95,-1.4) and (3.31,-0.3) .. (0,0) .. controls (3.31,0.3) and (6.95,1.4) .. (10.93,3.29)   ;
\draw    (90,111) -- (128,111) ;
\draw [shift={(130,111)}, rotate = 180] [color={rgb, 255:red, 0; green, 0; blue, 0 }  ][line width=0.75]    (10.93,-3.29) .. controls (6.95,-1.4) and (3.31,-0.3) .. (0,0) .. controls (3.31,0.3) and (6.95,1.4) .. (10.93,3.29)   ;
\draw [shift={(90,111)}, rotate = 180] [color={rgb, 255:red, 0; green, 0; blue, 0 }  ][line width=0.75]    (0,5.59) -- (0,-5.59)   ;
\draw [line width=2.25]    (168,140) -- (168,70) ;
\draw [line width=2.25]    (196,70) -- (196,140) ;
\draw [color={rgb, 255:red, 0; green, 0; blue, 255 }  ,draw opacity=1 ] [dash pattern={on 0.84pt off 2.51pt}]  (142,105) -- (168,105) ;
\draw [shift={(140,105)}, rotate = 0] [color={rgb, 255:red, 0; green, 0; blue, 255 }  ,draw opacity=1 ][line width=0.75]    (10.93,-3.29) .. controls (6.95,-1.4) and (3.31,-0.3) .. (0,0) .. controls (3.31,0.3) and (6.95,1.4) .. (10.93,3.29)   ;
\draw [line width=2.25]    (268,140) -- (268,70) ;
\draw [line width=2.25]    (296,70) -- (296,140) ;
\draw [color={rgb, 255:red, 0; green, 0; blue, 255 }  ,draw opacity=1 ] [dash pattern={on 0.84pt off 2.51pt}]  (242,105) -- (296,105) ;
\draw [shift={(240,105)}, rotate = 0] [color={rgb, 255:red, 0; green, 0; blue, 255 }  ,draw opacity=1 ][line width=0.75]    (10.93,-3.29) .. controls (6.95,-1.4) and (3.31,-0.3) .. (0,0) .. controls (3.31,0.3) and (6.95,1.4) .. (10.93,3.29)   ;

\draw (110,107.6) node [anchor=south] [inner sep=0.75pt]    {$u$};
\draw (222.5,100) node    {$+$};
\end{tikzpicture}\]
	    \caption{Unzip acting on an arrow tail ending on the unzipped strand. Unzip behaves similarly for arrow heads. }
	    \label{fig:unzip_arrow}
	\end{figure}
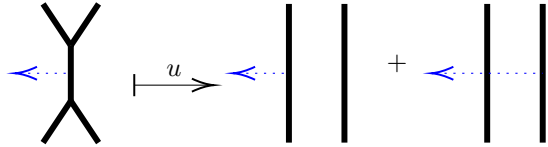

	\item The \emph{puncture} operation $p_e: \arrows(s) \to \arrows(p_e(s))$ punctures the strand $e$, turning it into a string (drawn as a red line). Any arrow diagram that has an arrow tail ending on $e$ is sent to zero.
	\item The \emph{deletion} operation $d_e: \arrows(s) \rightarrow \arrows(d_e(s))$ (for long strands $e$ not incident to any skeleton vertex) deletes $e$, and an arrow diagram with any arrow ending—head or tail—on $e$ is sent to 0.
\end{enumerate}

\begin{remark}
These operations capture structure in $\arrows$ which is not encoded by the operad of wiring diagrams (Definition~\ref{def: wiring diagram}). This is analogous to the type of operations that arise in the category of parenthesised braids, $\PaB$, such as strand deletion and strand doubling.  
\end{remark}



\subsubsection{Arrow diagrams and tangential derivations}
\label{ellmap}
At each skeleton $s\in\calS$, the $\mathbb{Q}$-vector space $\arrows(s)$ can be equipped with a co-unital, co-associative co-product:
 \[\begin{tikzcd}\Delta: \arrows(s)\arrow[r]& \arrows(s)\otimes\arrows(s)\end{tikzcd}\] 
which sends an arrow diagram $\bD\in\arrows(s)$ to the sum of all ways of distributing the connected components of the arrow graph of $\bD$ between the two copies of the skeleton.

\begin{definition}
An arrow diagram $\bD\in\arrows(s)$ is said to be \emph{primitive} if \[\Delta(\bD) = \bD\otimes 1 + 1 \otimes \bD.\] 
\end{definition}

The vector space of primitive arrow diagrams on a skeleton $s$ is denoted by $\calP(s)$, and $\calP=\sqcup_{s\in\calS} \calP(s)$. As it turns out, a single arrow diagram is primitive if, and only if, its arrow graph is connected. Therefore, $\calP$ is spanned by {\em trees} and {\em wheels}: a \emph{tree} is an arrow diagram whose arrow graph is connected and has no cycles (on the right of Figure~\ref{fig: tder arrow diagrams}), and a \emph{wheel} is an arrow diagram whose arrow graph is an $n$-cycle with $n$ ``spokes'' all oriented towards the cycle, as shown in Figure~\ref{fig:wheel example} \cite[Proof of Theorem 3.16]{BND:WKO2}. 

Let $\calP(\uparrow_n)$ denote primitive arrow diagrams on a skeleton of $n$ black strands (that is, $n$ tubes).
As mentioned earlier, $\calP(\uparrow_n)$ is closely related to the lie algebras $\mathfrak{tder}_n$ and $\cyc_n$, which is made precise in the following proposition. 

\begin{prop}\label{prop:SES}\cite[Proposition 3.19]{BND:WKO2}
There is a split short exact sequence of Lie algebras:

\begin{equation}
    \label{eq:primitiveshortexact}
    \begin{tikzcd}
	0 & {\cyc_n} & {\calP(\uparrow_n)} & {\mathfrak{tder}_n \oplus \mathfrak{a}_n} & 0
	\arrow[from=1-1, to=1-2]
	\arrow[hook, from=1-2, to=1-3]
	\arrow["pr", two heads, from=1-3, to=1-4]
	\arrow[from=1-4, to=1-5]
	\arrow["\ell", bend left = 33, dashed, from=1-4, to=1-3]
\end{tikzcd}
\end{equation} 
\end{prop}

Here, $\mathfrak{a}_n$ denotes the $n$-dimensional abelian Lie algebra given by tuples of the form $(\alpha_1 x_1, \alpha_2 x_2,...,\alpha_n x_n)$, for $\alpha_1,...,\alpha_n \in \Q$. This Lie algebra arises as the kernel of the natural map $u:(\mathfrak{lie}_n)^{\oplus n} \to \mathfrak{tder}_n$, which sends a tuple $\mathbf{a}=(a_1,...,a_n)$ to the derivation $x_i\mapsto [x_i,a_i]$.  

Cyclic words embed into $\calP(\uparrow_n)$, as {\em wheels}, as follows. Label the skeleton strands with the letters $x_1,...,x_n$. Given a cyclic word $w$ of length $k$ in $\cyc_n$, associate to it a {\em $k$-wheel} as in Figure~\ref{fig:wheel example}, oriented counter-clockwise, with ``spokes'' (arrow tails) attached to the skeleton according to the letters of $w$. The order in which arrow tails are attached on each skeleton strand does not matter due to the TC relation.

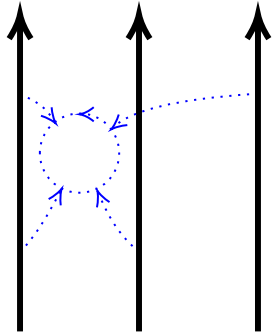
\begin{figure}[h]
\tikzset{every picture/.style={line width=0.75pt}} 

\[\begin{tikzpicture}[x=0.75pt,y=0.75pt,yscale=-1,xscale=1]

\draw [color={rgb, 255:red, 0; green, 0; blue, 255 }  ,draw opacity=1 ] [dash pattern={on 0.84pt off 2.51pt}]  (60,80) .. controls (67.99,84.23) and (71.56,85.22) .. (76.95,93.35) ;
\draw [shift={(78,95)}, rotate = 231.59] [color={rgb, 255:red, 0; green, 0; blue, 255 }  ,draw opacity=1 ][line width=0.75]    (7.65,-3.43) .. controls (4.86,-1.61) and (2.31,-0.47) .. (0,0) .. controls (2.31,0.47) and (4.86,1.61) .. (7.65,3.43)   ;
\draw [color={rgb, 255:red, 0; green, 0; blue, 255 }  ,draw opacity=1 ] [dash pattern={on 0.84pt off 2.51pt}]  (60,160) .. controls (67.68,151.84) and (73.06,143.22) .. (80.11,129.71) ;
\draw [shift={(81,128)}, rotate = 117.35] [color={rgb, 255:red, 0; green, 0; blue, 255 }  ,draw opacity=1 ][line width=0.75]    (7.65,-3.43) .. controls (4.86,-1.61) and (2.31,-0.47) .. (0,0) .. controls (2.31,0.47) and (4.86,1.61) .. (7.65,3.43)   ;
\draw [line width=2.25]    (60,200) -- (60,44) ;
\draw [shift={(60,40)}, rotate = 90] [color={rgb, 255:red, 0; green, 0; blue, 0 }  ][line width=2.25]    (12.24,-5.49) .. controls (7.79,-2.58) and (3.71,-0.75) .. (0,0) .. controls (3.71,0.75) and (7.79,2.58) .. (12.24,5.49)   ;
\draw [color={rgb, 255:red, 0; green, 0; blue, 255 }  ,draw opacity=1 ] [dash pattern={on 0.84pt off 2.51pt}]  (120,160) .. controls (110.88,151.84) and (105.91,145.06) .. (99.77,130.82) ;
\draw [shift={(99,129)}, rotate = 67.25] [color={rgb, 255:red, 0; green, 0; blue, 255 }  ,draw opacity=1 ][line width=0.75]    (7.65,-3.43) .. controls (4.86,-1.61) and (2.31,-0.47) .. (0,0) .. controls (2.31,0.47) and (4.86,1.61) .. (7.65,3.43)   ;
\draw [line width=2.25]    (120,200) -- (120,44) ;
\draw [shift={(120,40)}, rotate = 90] [color={rgb, 255:red, 0; green, 0; blue, 0 }  ][line width=2.25]    (12.24,-5.49) .. controls (7.79,-2.58) and (3.71,-0.75) .. (0,0) .. controls (3.71,0.75) and (7.79,2.58) .. (12.24,5.49)   ;
\draw  [color={rgb, 255:red, 0; green, 0; blue, 255 }  ,draw opacity=1 ][dash pattern={on 0.84pt off 2.51pt}] (70,110) .. controls (70,98.95) and (78.95,90) .. (90,90) .. controls (101.05,90) and (110,98.95) .. (110,110) .. controls (110,121.05) and (101.05,130) .. (90,130) .. controls (78.95,130) and (70,121.05) .. (70,110) -- cycle ;
\draw  [color={rgb, 255:red, 0; green, 0; blue, 255 }  ,draw opacity=1 ] (97.17,86.83) .. controls (94.78,88.82) and (92.39,90.02) .. (90,90.42) .. controls (92.39,90.81) and (94.78,92.01) .. (97.17,94) ;
\draw [color={rgb, 255:red, 0; green, 0; blue, 255 }  ,draw opacity=1 ] [dash pattern={on 0.84pt off 2.51pt}]  (180,80) .. controls (157.8,81.45) and (120.24,84.76) .. (107.31,96.67) ;
\draw [shift={(106,98)}, rotate = 321.04] [color={rgb, 255:red, 0; green, 0; blue, 255 }  ,draw opacity=1 ][line width=0.75]    (7.65,-3.43) .. controls (4.86,-1.61) and (2.31,-0.47) .. (0,0) .. controls (2.31,0.47) and (4.86,1.61) .. (7.65,3.43)   ;
\draw [line width=2.25]    (180,200) -- (180,44) ;
\draw [shift={(180,40)}, rotate = 90] [color={rgb, 255:red, 0; green, 0; blue, 0 }  ][line width=2.25]    (12.24,-5.49) .. controls (7.79,-2.58) and (3.71,-0.75) .. (0,0) .. controls (3.71,0.75) and (7.79,2.58) .. (12.24,5.49)   ;

\end{tikzpicture}\]
\caption{The image of $\trace(x^2yz)$ in $\mathcal{P}(\uparrow_3)$. We use the convention that wheels are oriented counterclockwise, unless otherwise marked.}\label{fig:wheel example}
\end{figure}

The map $pr$ sends wheels to zero. Given a single tree $\bT \in\calP (\uparrow_n)$, $pr(\bT) \in \tder_n \oplus \mathfrak{a}_n$ as follows. 
\begin{enumerate}
    \item If $\bT$ is a short arrow on strand $i$, then $pr(\bT)=x_i\in \mathfrak{a}_n$.
    \item Otherwise, $pr(\bT)\in \tder_n$. Label the skeleton strands with $x_1, x_2,...,x_n$.  $\bT$ determines a Lie word $w_\bT \in \mathfrak{lie}_n$ by reading the generator corresponding to the skeleton strand of each leaf (tail) of the tree, and combining these with brackets corresponding to each trivalent vertex. 
    If the \emph{head} (root) of the tree is on strand $i$, then $pr(\bT)(x_i)=[x_i,w_\bT]$, and $pr(T)(x_j)=0$ for all $j\neq i$. See Figure~\ref{fig: tder arrow diagrams} for two examples.
\end{enumerate}


\begin{figure}[h]

\[\begin{tikzpicture}[x=0.75pt,y=0.75pt,yscale=-.75,xscale=.75]

\draw [<-][line width=1.5]    (75,82) -- (75,225) ;
\draw [<-] [line width=1.5]    (115,82) -- (115,225) ;
\draw [<-][line width=1.5]    (243,82) -- (243,225) ;
\draw [<-][line width=1.5]    (283,82) -- (283,224.41) ;
\draw [->][color={rgb, 255:red, 0; green, 0; blue, 255 }  ,draw opacity=1 ][line width=0.75]  [dotted]  (75,150) -- (115,150) ;
\draw [<-][color={rgb, 255:red, 0; green, 0; blue, 255 }  ,draw opacity=1 ][line width=0.75]  [dotted]  (243,150) -- (283,150) ;
\draw [<-][line width=1.5]    (433,82) -- (433,225) ;
\draw [<-][line width=1.5]    (473,82) -- (473,225) ;
\draw [<-][line width=1.5]    (512,82) -- (512,225) ;
\draw  [color={rgb, 255:red, 0; green, 0; blue, 255 }  ,draw opacity=1 ] [dotted]  (435,92) -- (453,116) ;
\draw [->][color={rgb, 255:red, 0; green, 0; blue, 255 }  ,draw opacity=1 ] [dotted]  (512,103) .. controls (460,140) and (420,159) .. (471,183) ;
\draw [->][color={rgb, 255:red, 0; green, 0; blue, 255 }  ,draw opacity=1 ] [dotted]  (470,100) .. controls (457,113) and (440,117) .. (461,141) ;

\draw (169,133.4) node [anchor=north west][inner sep=0.75pt]    {$+$};

\end{tikzpicture}\]
\caption{The arrow diagram on the left maps by $pr$ to the element $(y,x)\in\tder_2$ from Example~\ref{example: t}. The tree on the right maps to $(0,[[x,y],z],0)\in\tder_3$.
}\label{fig: tder arrow diagrams}
\end{figure}
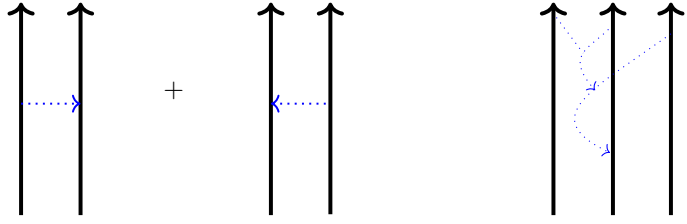

\begin{remark}\label{remark:l}The map $pr$ in \eqref{eq:primitiveshortexact} admits a section $\ell: \mathfrak{tder}_n \oplus \mathfrak{a}_n \to \calP(\uparrow_n)$. For $x_i\in \mathfrak{a}_n$, define $\ell(x_i)$ to be a short arrow on strand $i$, pointing down (though this is equivalent to a short arrow pointing up via the RI relation). For $d\in\tder_n$, write $d=(a_1,...,a_n)$ where $d(x_i)=[x_i,a_i]$. Choose $a_i$ so that the coefficient of the linear term $x_i$ in $a_i$ is 0. To each lie word in $a_i$, associate a binary tree oriented towards a single head, with tails labelled by the letters $x_1,...,x_n$. In $\calP(\uparrow_n)$ attach all tails to the strand labelled by the appropriate letter, and attach the head to strand $i$, below all the tails.
See Figure~\ref{fig:exampleupper} for an example. The element $\ell(d)$ is the sum of all the trees arising from $a_1,...,a_n$ in $\calP(\uparrow_n)$. The reader can check that $\ell$ is a Lie algebra map. We denote the restriction of $\ell$ to $\tder_2$ by $\ell$ also.

Note that the order in which arrow tails (leaves) of each tree are attached along a strand doesn't matter, due to the TC (Tails Commute) relation. However, the attachment position of the head (root) of the tree—in relation to the tails on the same strand—does make a difference. For the map $\ell$, we choose to attach tree heads below tails, but, for example, one could choose to attach the heads at the top instead. These two maps are not the same, but their difference lies in the image of $\cyc_n$, by STU relations. 
\end{remark}

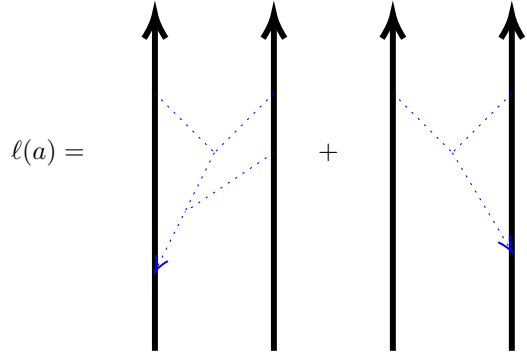
\begin{figure}[h]
\[\begin{tikzpicture}[x=0.75pt,y=0.75pt,yscale=-1,xscale=1]

\draw [line width=2.25]    (130,220) -- (130,54) ;
\draw [shift={(130,50)}, rotate = 90] [color={rgb, 255:red, 0; green, 0; blue, 0 }  ][line width=2.25]    (12.24,-5.49) .. controls (7.79,-2.58) and (3.71,-0.75) .. (0,0) .. controls (3.71,0.75) and (7.79,2.58) .. (12.24,5.49)   ;
\draw [line width=2.25]    (190,220) -- (190,54) ;
\draw [shift={(190,50)}, rotate = 90] [color={rgb, 255:red, 0; green, 0; blue, 0 }  ][line width=2.25]    (12.24,-5.49) .. controls (7.79,-2.58) and (3.71,-0.75) .. (0,0) .. controls (3.71,0.75) and (7.79,2.58) .. (12.24,5.49)   ;
\draw [color={rgb, 255:red, 0; green, 0; blue, 255 }  ,draw opacity=1 ] [dash pattern={on 0.84pt off 2.51pt}]  (130,90) -- (160,120) ;
\draw [color={rgb, 255:red, 0; green, 0; blue, 255 }  ,draw opacity=1 ] [dash pattern={on 0.84pt off 2.51pt}]  (190,90) -- (160,120) ;
\draw [color={rgb, 255:red, 0; green, 0; blue, 255 }  ,draw opacity=1 ] [dash pattern={on 0.84pt off 2.51pt}]  (190,120) -- (145,150) ;
\draw [color={rgb, 255:red, 0; green, 0; blue, 255 }  ,draw opacity=1 ] [dash pattern={on 0.84pt off 2.51pt}]  (160,120) -- (130.89,178.21) ;
\draw [shift={(130,180)}, rotate = 296.57] [color={rgb, 255:red, 0; green, 0; blue, 255 }  ,draw opacity=1 ][line width=0.75]    (7.65,-3.43) .. controls (4.86,-1.61) and (2.31,-0.47) .. (0,0) .. controls (2.31,0.47) and (4.86,1.61) .. (7.65,3.43)   ;
\draw [line width=2.25]    (250,220) -- (250,54) ;
\draw [shift={(250,50)}, rotate = 90] [color={rgb, 255:red, 0; green, 0; blue, 0 }  ][line width=2.25]    (12.24,-5.49) .. controls (7.79,-2.58) and (3.71,-0.75) .. (0,0) .. controls (3.71,0.75) and (7.79,2.58) .. (12.24,5.49)   ;
\draw [line width=2.25]    (310,220) -- (310,54) ;
\draw [shift={(310,50)}, rotate = 90] [color={rgb, 255:red, 0; green, 0; blue, 0 }  ][line width=2.25]    (12.24,-5.49) .. controls (7.79,-2.58) and (3.71,-0.75) .. (0,0) .. controls (3.71,0.75) and (7.79,2.58) .. (12.24,5.49)   ;
\draw [color={rgb, 255:red, 0; green, 0; blue, 255 }  ,draw opacity=1 ] [dash pattern={on 0.84pt off 2.51pt}]  (250,90) -- (280,120) ;
\draw [color={rgb, 255:red, 0; green, 0; blue, 255 }  ,draw opacity=1 ] [dash pattern={on 0.84pt off 2.51pt}]  (310,90) -- (280,120) ;
\draw [color={rgb, 255:red, 0; green, 0; blue, 255 }  ,draw opacity=1 ] [dash pattern={on 0.84pt off 2.51pt}]  (280,120) -- (308.97,168.29) ;
\draw [shift={(310,170)}, rotate = 239.04] [color={rgb, 255:red, 0; green, 0; blue, 255 }  ,draw opacity=1 ][line width=0.75]    (7.65,-3.43) .. controls (4.86,-1.61) and (2.31,-0.47) .. (0,0) .. controls (2.31,0.47) and (4.86,1.61) .. (7.65,3.43)   ;

\draw (217.5,120) node    {$+$};
\draw (76,120) node    {$\ell ( a) =$};

\end{tikzpicture}\]
    \caption{An example of $\ell(\ba)$ for $\ba = ([[x,y],y],[x,y])\in \tder_2$. }
    \label{fig:exampleupper}
\end{figure}

 Restricting our attention to arrow diagrams on the skeleton of $n$ vertical strands enabled us to identify $\tder_n$ and $\cyc_n$ within $\calP$. In fact, arrow diagrams on $n$-strands possess more structure: there is an associative, unital product on each vector space $\arrows(\uparrow_n)$ given by stacking: \[\begin{tikzcd}
    \arrows(\uparrow_n)\otimes \arrows(\uparrow_n)\arrow[rr, "stack"]&& \arrows(\uparrow_n)
\end{tikzcd}\] This product is compatible with the co-product, making $\arrows(\uparrow_n)$ into a Hopf algebra. The following fact is established in \cite[Section 3.2]{BND:WKO2}:

\begin{prop}\label{prop:milnor_moore}
The stacking product and arrow diagram co-product make $\arrows(\uparrow_n)$ into a Hopf algebra, and, by the Milnor-Moore theorem \cite{MR174052}, $\arrows(\uparrow_n)$ is isomorphic to the completed universal enveloping algebra over its primitive elements. In particular, \[\arrows(\uparrow_n)\overset{\Upsilon}{\cong} \hat{U}(\tder_n\oplus \mathfrak{a}_n\ltimes \cyc_n).\]
\end{prop}

As a consequence of Proposition~\ref{prop:milnor_moore}, it makes sense to define \emph{group-like} arrow diagrams as exponentials of primitive elements, formally interpreted as power series: $e^\bD= \sum_{k=0}^\infty \frac{\bD^k}{k!}$, where $\bD\in \calP(\uparrow_n)$.

\begin{notation}
We will write $\calG(\uparrow_n)$ for the group of group-like elements of $\arrows(\uparrow_n)$. 
\end{notation}

In contrast to chord diagrams, arrow diagrams on an arbitrary skeleton do not usually assemble into a Hopf algebra. However, the following lemma is a straightforward application of the $\textnormal{VI}$ relation:  

\begin{lemma}\label{lemma: VI}\cite[Lemma 4.7]{BND:WKO2}
There is a canonical isomorphism of vector spaces \[\arrows(\vertex)\cong\arrows(\uparrow_2) \overset{\Upsilon}{\cong}  \hat{U}(\tder_n\oplus \mathfrak{a}_n\ltimes \cyc_n).\]  Therefore, the Hopf algebra structure on $\arrows(\uparrow_2)$ can be pulled back along this isomorphism to give a Hopf algebra structure on $\arrows(\vertex).$
\end{lemma}

\begin{remark}\label{remark:tree-like arrow diagrams}
The projection  $pr: \calP(\uparrow_n) \to \tder_n\oplus \mathfrak{a}_n$ induces a Hopf algebra map $pr: \arrows(\uparrow_n) \rightarrow \hat{U}(\tder_n \oplus \mathfrak{a}_n)$. Regarding $\cyc_n$ as an abelian Lie algebra, we also have an induced inclusion $\hat{U}(\cyc_n)\hookrightarrow \arrows(\uparrow_n)$. 
Denote by $\arrows(\uparrow_n)^{tr}$ the quotient $\arrows/\im(\hat{U}(\cyc_n))$, we call this quotient the \emph{tree-level} arrow diagrams. We write $\apr$ for the quotient map $\arrows(\uparrow_n) \to \arrows(\uparrow_n)^{tr}$, and for $\bD \in \A(\uparrow_n)$ we denote $\apr(\bD)$ by $\bD^{tr}$. Note that the map $pr: \arrows(\uparrow_n)\to \hat{U}(\tder_n \oplus \mathfrak{a}_n)$ descends to an isomorphism $\arrows(\uparrow_n)^{tr}\cong\hat{U}(\tder_n \oplus \mathfrak{a}_n)$.
\end{remark}

Next, we identify the space of arrow diagrams on one or more capped strands.
Arrow diagrams with an arrow head adjacent to a cap vanish by the $\textnormal{CP}$ relation. When the skeleton consists only of capped strands, arrow heads can be eliminated entirely by successive applications of $\textnormal{STU}$ and $\textnormal{CP}$ relations to bring them adjacent to the cap(s). In turn, arrow diagrams with no arrow heads can be expressed as a linear combination of only {\em wheels}, which are encoded as cyclic words. This leads to the following lemma, which is a straightforward generalisation of \cite[Lemma 4.6]{BND:WKO2}.

\begin{lemma}\label{cyclic_lemma}
There is an isomorphism of graded vector spaces $\arrows(\bcap_n) \cong \cyc_n/\cyc_n^1$. Here $\cyc_n^1$ denotes the degree one component of cyclic words. 
Cyclic words in one letter are power series, in particular,
$ \arrows(\bcap_1)\cong\mathbb Q[[\xi]]/\langle\xi \rangle,$ where the quotient is understood linearly.
\end{lemma} 

\begin{remark}\label{rmk: wheels power series}
In Lemma~\ref{cyclic_lemma} the reason for taking the quotient by the degree one component in  is the RI relation on arrow diagrams forces a degree 1 wheel (i.e. a wheel with one "spoke") to vanish. 
\end{remark}
Since $\cyc_n\cong \trace(\mathbb Q\langle\langle x_1,x_2,...,x_n \rangle\rangle)$, we also have, for example $\arrows(\bcap_2)\cong \trace(\mathbb{Q}\langle\langle \xi, \zeta \rangle\rangle) /\langle \xi, \zeta\rangle$, where the quotient is again linear.

\begin{remark} \label{remark:omega}
There is a well-defined surjective linear map $\omega: \A(\bcap_n) \to \mathbb{Q}[x_1,\cdots,x_n]$ which sends an arrow diagram—expressed in a wheels form, that is, with arrow heads eliminated—to the product of "$x_i$ to the power of the number of tails on strand $i$" for $i=1,...,n$.
\end{remark}

The Jacobian map $J: \TAut_n \to \cyc_n$ also has a diagrammatic interpretation\footnote{Note that our $e^J$ is denoted $J$ in \cite{BND:WKO2}.}, captured by the \emph{adjoint} operation $A_e$. For $\mathbf{D}\in\arrows(\uparrow_n)$ let $\mathbf{D}^*$ denote $A_1...A_n(\mathbf{D})$, that is, the adjoint operation applied to every strand of $\mathbf{D}$. 

\begin{prop}\label{prop:J}\cite[Proposition 3.27]{BND:WKO2}
For any $\mathbf{u}\in\tder_n$, 
$$e^{\ell(\mathbf{u})}\left(e^{\ell(\mathbf{u})}\right)^* = e^{-J(e^\mathbf{u})}.$$
\end{prop}

\subsection{Homomorphic Expansions of \texorpdfstring{$w$}{w}-Foams}\label{sec:FoamExpansions} 
A map of wheeled props $Z:\wf\rightarrow\A$ is \emph{skeleton preserving} if it restricts to a map of graded vector spaces $Z(s): \wf[s]\rightarrow\A[s]$ for each $s \in\calS$. Homomorphic expansions of $w$-foams are isomorphisms of tensor categories $Z:\mathbb{Q}[\hatwf]\rightarrow \arrows$, which are skeleton-preserving and compatible with all operations (including operations external to the wheeled prop structure, such as punctures and unzips) \cite[Definition 2.4, 2.5]{BND:WKO2}. A homomorphic expansion is {\em group-like} if the $Z$-values of generators of $\wf$ are group-like: that is, exponentials of primitive elements in $\arrows$. 

By finite generation, a homomorphic expansion $Z:\mathbb{Q}[\widehat{\wf}]\rightarrow \A$ is determined by its values on the generators: \[\ocrossing, \quad \ucrossing, \quad \vertex, \quad \negvertex, \quad \pv, \quad \ppv, \quad \text{and} \quad \bcap.\] The values $Z$ takes on these generators are subject to equations given by the relations in $\wf$, as well as equations given by the compatibility with the extra operations: orientation switch, adjoint, unzip, puncture and strand deletion. A careful analysis of the equations leads to the following:

\begin{prop}\label{prop:puncturedexps} A group-like homomorphic expansion $Z:\mathbb{Q}[\widehat{\wf}]\to \arrows$ is uniquely determined by the values
\[Z(\ocrossing) = e^{\rightarrowdiagram} = R \in \A(\uparrow_2), \ \quad Z(\vertex)=e^\nu=V \in\A(\vertex)\cong\A(\uparrow_2) \quad \text{and} \] 
\[Z(\bcap) =e^c=C  \in \A(\bcap)\cong \mathbb Q[[\xi]]/\langle \xi \rangle, \] 
which are subject to the following equations: 
\begin{align*}
    V^{12}R^{(12)3} &= R^{23}R^{13}V^{12} \tag{R4} \quad\text{ in }\quad\A(\uparrow_3) \\
    V \cdot A_2A_2(V) &=  1\quad\text{ in }\quad \A(\uparrow_2)\tag{U}\\
    C^{12}(V^{12})^{-1} &= C^1C^2 \quad\text{ in }\quad \A(\bcap_2)\tag{C}. 
\end{align*}
\end{prop}

This proposition is a minor generalisation of \cite[Section 4.3]{BND:WKO2} and a direct consequence of \cite[Theorem 2.6]{BND:WKO3}. As the latter paper has not yet appeared, we summarise the key points of the proof below. See also the discussion in \cite[Section 4]{DHR21}.  

\begin{proof}[Proof (Sketch).]
The main result (Theorem~4.9, building on Theorem~3.30) of \cite{BND:WKO2} is a classification of homomorphic expansions for a class of $w$-foams
\begin{equation*}
\overline{\wf} = \left\langle {\ocrossing, \ucrossing, \vertex, \negvertex, \bcap} \,\Big{|}\, R1^f, R2, R3, OC, CP \right\rangle,    
\end{equation*}
where the operations are orientation switches and unzips, but no \emph{punctures}. In particular, given any group-like homomorphic expansion $\overline{Z}$ of $\overline{\wf}$ the $\overline{Z}$-values of the generators $\ocrossing$, $\vertex$ and $\bcap$ uniquely determine the values $\overline{Z}(\ucrossing)$ and $\overline{Z}(\negvertex)$. For example, the second Reidemeister relation, R2, says that $\overline{Z}(\ocrossing \cdot \ucrossing)=\overline{Z}(\ocrossing)\cdot \overline{Z}(\ucrossing)=1$. In \cite[Section 4.3]{BND:WKO2} it is shown that $\overline{Z}$ is a homomorphic expansion of $\overline{wf}$ if and only if $\overline{Z}(\ocrossing)$, $\overline{Z}(\ucrossing)$ and $\overline{Z}(\vertex)$ satisfy the equations (R4) (U) and (C). 

We deduce the Proposition by constructing a bijection between homomorphic expansions of $\overline{\wf}$ and homomorphic expansions of $\wf$, via restriction and unique extension.
Using the definition of circuit algebra as an algebra over the operad of wiring diagrams (Definition~\ref{def: circut algebra}), it is relatively straightforward to check that $\overline{\wf}$ is a sub-circuit algebra of $\wf$. It follows that any homomorphic expansion $Z:\widehat{\wf}\to \arrows$ restricts to a homomorphic expansion of $\overline{\wf}$. 

Conversely, any homomorphic expansion $\overline{Z}$ of $\overline{\wf}$ uniquely extends to a homomorphic expansion of $\wf$, because all generators of $\wf$ with punctured strands can be obtained by applying the puncture operation to a strand of a generator of $\overline{\wf}$ and any homomorphic expansion $Z$ of $\wf$ is required to commute with the puncture operation. Thus, the $Z$-values of the punctured generators are uniquely determined by $\overline{Z}$.
\end{proof}

The isomorphism $\Upsilon: \arrows(\uparrow_2)\rightarrow \widehat{U}(\cyc_2\rtimes(\tder_2\oplus \mathfrak{a}_2))$ from Lemma~\ref{lemma: VI} identifies the value $Z(\vertex)\in\arrows(\uparrow_2)$ with $$\Upsilon(Z(\vertex))=e^be^\nu \in  \widehat{U}(\cyc_2 \rtimes (\tder_2\oplus \mathfrak{a}_2)),$$ where $b\in \cyc_2$ and $\nu \in \tder_2 \oplus \mathfrak a _2$. 
In \cite[Theorem 4.9]{BND:WKO2}, KV solutions are shown to be in bijection with families of group-like homomorphic expansions of $\wf$ which are parameterised by the $\mathfrak{a}_2$ component of $\nu$.  
\begin{definition}\label{def:vsmall}
A homomorphic expansion is \emph{$v$-small} if the $\mathfrak{a}_2$ component of $\log \Upsilon(Z(\vertex))$ is zero.
\end{definition}

\begin{theorem}\cite[Theorem 4.9]{BND:WKO2}\cite[Theorem 2.6]{BND:WKO3}\label{expansions are solKV}
There is a one-to-one correspondence between the set of $v$-small, group-like homomorphic expansions $Z:\wf\xrightarrow{} \A$, and Kashiwara-Vergne solutions. 
\end{theorem}

\begin{proof}[Proof (sketch)]
In \cite{BND:WKO2} Theorem~4.9 states this for $v$-small, group-like homomorphic expansions of $\overline{\wf}$. The main steps in this proof identify the value $Z(\vertex)$ with a tangential automorphism $e^\nu\in \hat{U}(\tder_2)$ and $Z(\bcap)$ with a power series $e^c\in\Z(\bcap)\in u\mathbb{Q}[[u]]$. The equation R4 forces the value $Z(\vertex)=e^\nu$ to satisfy the first KV equation \eqref{eq:SolKV}, and the equations U and C of Proposition~\ref{prop:puncturedexps} force the Jacobian of $e^\nu$ to satisfy the second KV equation.

By unique extension and restriction (as in Proposition \ref{prop:puncturedexps}), homomorphic expansions for $\overline{\wf}$ are in bijection with homomorphic expansions for $\wf$, and this correspondence preserves the $v$-small property. The statement then follows. 
\end{proof}

\subsection{Automorphisms of Arrow Diagrams and the KV groups}\label{sec:KV groups}

The Kashiwara-Vergne symmetry groups, $\kv$ and $\krv$, act freely and transitively on the set of all KV solutions. The Kashiwara-Vergne group $\kv$ (\cite[after Proposition 7]{ATE10}) is defined as follows:
\begin{multline}
\kv:= \Big\{(a, \sigma)\in \TAut_2(\mathbb{Q})\times u^2\mathbb{Q}[[u]] \Big| \\ 
a(e^{x}e^{y})= e^{x}e^{y} \ \text{and} \  J(a)=\trace\Big(\sigma(\mathfrak{bch}(x,y))-\sigma(x)-\sigma(y)\Big)\Big\}.
\end{multline} The \emph{graded} Kashiwara-Vergne group $\krv$ is the group corresponding to a Lie subalgebra of the tangential derivations of $\lie_2$ which vanish on the sum of the generators (\cite[Remark 50]{ATE10}, \cite[Sec.\ 4.1]{AT12}). Explicitly,
\begin{multline}\label{eq:krvdef}
\krv:= \Big\{(\alpha, s)\in \TAut_2(\mathbb{Q})\times u^2\mathbb{Q}[[u]] \Big| \\ 
\alpha(e^{x+y})=e^{x+y} \ \text{and} \  J(\alpha)=\trace\Big(s(x+y)-s(x)-s(y)\Big)\Big\}. 
\end{multline}
The left action of $\kv$ on SolKV is given by right composition with the inverse in $\TAut_2$, $a\cdot F=F\circ a^{-1}$, and the right action of $\krv$ on SolKV is given by $F\cdot \alpha=\alpha^{-1}\circ F$.

The Kashiwara-Vergne symmetry groups $\kv$ and $\krv$ can be identified with automorphisms of $\mathbb{Q}[\hatwf]$ and $\arrows$, respectively, which are both \emph{skeleton} and \emph{expansion preserving}. Explicitly, skeleton and expansion preserving automorphisms of $\arrows$ are maps $G: \A \to \A$ such that for any expansion $Z: \mathbb{Q}[\widehat{\wf}] \to \A$, the composite $G \circ Z$ is also an expansion. We denote the group of skeleton and expansion preserving automorphisms by $\Aut(\arrows)$.


\begin{prop}\label{prop:SimplifiedEqns} \cite[Proposition 5.6]{DHR21} For a skeleton-preserving map $G: \arrows \to \arrows$, $G\in \Aut(\calA)$ if, and only if, $G(\rightarrowdiagram)=R=e^{\rightarrowdiagram}\in\arrows(\uparrow_2)$ and the values $G(\bcap)=C=e^c $ and $G(\vertex)=N=e^\eta$ satisfy the following equations:
\begin{gather} 
(N^{12})^{-1}R^{(12)3}N^{12}=R^{(12)3} \quad \text{ in } \quad \arrows(\uparrow_3) \tag{R4'}  \label{R4'}\\
N\cdot A_1A_2 (N)=1 \quad \text{ in } \quad \arrows(\uparrow_2)  \tag{U'} \label{U'}\\
C^{12}(N^{12})^{-1} =C^1C^2  \quad \text{ in } \quad \arrows(\bcap_2) \tag{C'} \label{C'}.
\end{gather}
\end{prop}


\begin{remark}
We note that \cite[Proposition 5.6]{DHR21} characterises automorphisms of $\overline{\arrows}=\gr \overline{\wf}$, that is, arrow diagrams on skeleta without punctured strands. Nonetheless, the Proposition stated here is equivalent by the same restriction and unique extension argument that we used in the proof of Proposition~\ref{prop:puncturedexps}.
\end{remark}

An automorphism $G\in \Aut(\arrows)$ is called \emph{$v$-small} if the $\mathfrak{a}_2$ component of $\log \Upsilon(G(\vertex))$ is zero. We denote the group of $v$-small automorphisms by $\Aut_v(\arrows)$.
In \cite[Theorem 5.12; Proposition 5.15]{DHR21}, two of the authors, with Halacheva, show:
\begin{theorem}\label{thm:kv and krv}
There are group isomorphisms \[\krv \stackrel{\Theta}{\cong} \Aut_v(\arrows) \quad \text{and} \quad  \kv\stackrel{\Theta_{Z}}{\cong} \Aut_v(\hatwf).\] 
\end{theorem}

To define the isomorphism 
$
\begin{tikzcd} \Theta: \Aut_v(\arrows) \arrow[r] & \krv,\end{tikzcd}
$
we write, for $G\in \Aut_v(\arrows)$,  
\[pr(\log G(\vertex))=\eta \in \tder_2, \quad 
\text{and} \quad 
G(\bcap) = e^c \quad \text{where} \quad c\in \Q[[\xi]]/\langle\xi\rangle.\] 
Then 
\begin{equation}\label{eq:Theta}
\Theta(G)=\left(e^\eta, 2c\right)\in \krv.
\end{equation}

\subsubsection{Duflo elements and the evaluation of automorphisms of vertices} While elements of $\kv$ and $\krv$ are defined above as pairs, in both 
cases the first component \emph{uniquely} determines the 
second (\cite[Section 6.1]{ATE10}). In particular,  for both $\kv$ and $\krv$ there are Duflo maps given by $a\mapsto \sigma$ and $\alpha \mapsto s$. These Duflo maps are group homomorphisms $\TAut_2 \to u^2\mathbb Q[[u]]$, where the group operation in $\TAut_2$ is composition, and $u^2\mathbb Q[[u]]$ is the additive group \cite[Proposition 34]{ATE10}. For example, if $(\alpha_1, s_1)\in \krv$ and $(\alpha_2,s_2)\in \krv$ then $(\alpha_2\circ \alpha_1, s_1+s_2)\in \krv$. The corresponding fact about automorphisms of $\arrows$ is given in the following Proposition and Corollary, which are proven by standard perturbative arguments:

\begin{prop}\label{vertex_auto}
Any $G\in\Aut_v(\A)$ is uniquely determined by the value $N=G(\vertex)\in\A(\uparrow_2)$. 
\end{prop}

\begin{proof}

Given $G \in \Aut_v(\A)$, we know that $G(\rightarrowdiagram)=\rightarrowdiagram$, so we only need to show that $N$ uniquely determines the value $G(\bcap)=C=e^c$, where $c\in \Q[[\xi]]/\langle\xi \rangle$. 
To this end, assume that there exist 
$G, G'\in \Aut_v(\A)$ 
given by 
\[(G(\vertex), G(\bcap))=(N,C) \quad \text{and} \quad (G'(\vertex), G'(\bcap))=(N,C')\] 
where $C \neq C'$. 

Let $\varepsilon$ denote the lowest degree term of the difference $C-C'$, and assume that $\varepsilon$ is in degree $k$. Note that $\varepsilon = \lambda x^k$, where $0\neq\lambda \in \mathbb{Q}$.
The Cap Equation, \eqref{C'} applies to both $C$ and $C'$. Writing the difference of the two cap equations in degree $k$ we obtain 
\begin{align}
    0 = \lambda((x+y)^r-x^r-y^r).
\end{align}
Since $\lambda \neq 0$, this implies that $r = 1$. But $C, C' \in \Q[[\xi]]/\langle\xi\rangle$ have no degree one terms, a contradiction. 
\end{proof}

Given a $G \in \Aut_v(\A)$ with $N=G(\vertex)$, we write $N^{tr}$ for $\apr(N)\in\A(\vertex)^{tr}$, as in Remark~\ref{remark:tree-like arrow diagrams}. 

\begin{cor} \label{auto_wheels}
Any $G\in\Aut_v(\A)$ is uniquely determined by $N^{tr}$. 
\end{cor}
\begin{proof}
 We aim to recover $G(\vertex)=N$ from $N^{tr}$. 
Assume that $G, G' \in \Aut_v(\A)$ with $G(\vertex) = N$, $G'(\vertex) = N'$ such that $N \neq N'$ and $N^{tr}=(N')^{tr}$. Let $\varepsilon$ denote the lowest degree term of $N-N'$, and assume that $\varepsilon$ is of degree $k$.  Since $\varepsilon$ is in the kernel of $\apr$, it is in the image of $\cyc_n$, and thus it is a scalar multiple of a $k$-wheel. 

By Proposition~\ref{prop:puncturedexps}, $N$ and $N'$ satisfies the Unitarity equation \eqref{U'}: $A_1A_2(N) = N^{-1}$ and $A_1A_2(N') = N'^{-1}$. Modulo degree $k+1$ we have 
\begin{align}
    N = N' + \varepsilon \quad \text{ and } \quad N^{-1} = (N')^{-1}-\varepsilon.
\end{align}
And therefore, subtracting the two \eqref{U'} equations, up to degree $k$ we have
\begin{align}
   A_1A_2( \varepsilon) =A_1A_2(N-N') = N^{-1}-(N')^{-1}= -\varepsilon 
\end{align}

Now, recall that the adjoint operation $A_i$ reverses the direction of strand $i$ and multiplies each arrow diagram by $(-1)^{\# \ \text{heads on that strand}}$. Since $\varepsilon$ is a scalar multiple of a $k$-wheel, it only has arrow tails on each strand. It follows that $A_1A_2(\varepsilon) = \varepsilon$ and thus $\varepsilon = -\varepsilon$ so $\varepsilon = 0$, a contradiction. 
\end{proof}

\subsection{The relationship between chord diagrams and arrow diagrams}\label{sec:PaCDtoA}
For any $n\geq 2$, there is an inclusion of Lie algebras $\mathfrak{t}_n\hookrightarrow \tder_n$. Namely, $\mathfrak{t}_n$ is isomorphic to a Lie subalgebra of $\tder_n$ spanned by tangential derivations of the form $t^{i,j}=(0,\ldots, x_j,\ldots, x_i,\ldots 0)$(\cite[Proposition 3.11]{AT12} and see Example~\ref{example: t} above).  This inclusion of Lie algebras extends to a Hopf algebra homomorphism \[\begin{tikzcd}\mathsf{CD}(n)\cong\widehat{U}(\mathfrak{t}_n)\arrow[r, "\epsilon"] & \widehat{U}(\tder_n\oplus \mathfrak{a}_n \ltimes \cyc_n)  \arrow[r, "\cong"', "\Upsilon^{-1}"] &  \arrows(\uparrow_n).\end{tikzcd}\] Diagrammatically, this composition maps a chord between strands $i$ and $j$ to the sum of two arrows: one going from strand $j$ to strand $i$ (representing the $x_j$ in the $i$th component) and one going from strand $i$ to strand $j$ to (representing  $x_i$ in the $j$th component). For example, the arrow diagram on the left of Figure~\ref{fig: tder arrow diagrams} represents the image of the chord diagram $t^{12}=\chord$. (See also the proof of Lemma 6.3 in \cite{DHR21}). 

This map of Hopf algebras extends to an inclusion of tensor categories, $\begin{tikzcd}\PaCD\arrow[r, "\alpha"]& \arrows\end{tikzcd}$, defined ``one skeleton at a time''. Recall that $\PaCD(s)$ consists of the chord diagrams with underlying skeleton $s$. Equivalently, an element of $\PaCD(s)$ is a morphism in $\PaCD(n)$, where the data of the skeleton provides the underlying permutation and parenthesisation. The example on the right of Figure~\ref{fig:pacdexample} showed such a morphism from $p_1=((13)2)4)$ to $p_2=(12)(34)$. ``Closing up'' the parenthesised permutations with binary trees representing the parenthesisation produces a $\wf$ skeleton, as in Figure~\ref{fig:alphamap}. The map $\alpha$ sends a chord diagram $\bD$ to the arrow diagram $\Upsilon^{-1}\epsilon(\bD)$ on the closed skeleton, as shown in Figure~\ref{fig:alphamap}.

\begin{figure}[h]
    \centering

\tikzset{every picture/.style={line width=0.75pt}} 
\[\begin{tikzpicture}[x=0.75pt,y=0.75pt,yscale=-1,xscale=1]

\draw [line width=2.25]    (60,240) -- (60,84) ;
\draw [shift={(60,80)}, rotate = 90] [color={rgb, 255:red, 0; green, 0; blue, 0 }  ][line width=2.25]    (12.24,-5.49) .. controls (7.79,-2.58) and (3.71,-0.75) .. (0,0) .. controls (3.71,0.75) and (7.79,2.58) .. (12.24,5.49)   ;
\draw [line width=2.25]    (110,240) -- (110,84) ;
\draw [shift={(110,80)}, rotate = 90] [color={rgb, 255:red, 0; green, 0; blue, 0 }  ][line width=2.25]    (12.24,-5.49) .. controls (7.79,-2.58) and (3.71,-0.75) .. (0,0) .. controls (3.71,0.75) and (7.79,2.58) .. (12.24,5.49)   ;
\draw [line width=2.25]    (160,240) -- (160,84) ;
\draw [shift={(160,80)}, rotate = 90] [color={rgb, 255:red, 0; green, 0; blue, 0 }  ][line width=2.25]    (12.24,-5.49) .. controls (7.79,-2.58) and (3.71,-0.75) .. (0,0) .. controls (3.71,0.75) and (7.79,2.58) .. (12.24,5.49)   ;
\draw [color={rgb, 255:red, 0; green, 0; blue, 255 }  ,draw opacity=1 ] [dash pattern={on 0.84pt off 2.51pt}]  (60,160) -- (110,160) ;
\draw [color={rgb, 255:red, 0; green, 0; blue, 255 }  ,draw opacity=1 ] [dash pattern={on 0.84pt off 2.51pt}]  (110,180) -- (160,180) ;
\draw    (190,160) -- (288,160) ;
\draw [shift={(290,160)}, rotate = 180] [color={rgb, 255:red, 0; green, 0; blue, 0 }  ][line width=0.75]    (10.93,-3.29) .. controls (6.95,-1.4) and (3.31,-0.3) .. (0,0) .. controls (3.31,0.3) and (6.95,1.4) .. (10.93,3.29)   ;
\draw [shift={(190,160)}, rotate = 180] [color={rgb, 255:red, 0; green, 0; blue, 0 }  ][line width=0.75]    (0,5.59) -- (0,-5.59)   ;
\draw [line width=2.25]    (330,105.79) -- (330,53.21) ;
\draw [line width=2.25]    (359.47,105.79) -- (359.93,53.21) ;
\draw [line width=2.25]    (389.41,105.79) -- (389.87,54.21) ;
\draw [color={rgb, 255:red, 0; green, 0; blue, 255 }  ,draw opacity=1 ] [dash pattern={on 0.84pt off 2.51pt}]  (330,76.32) -- (357.93,76.32) ;
\draw [shift={(359.93,76.32)}, rotate = 180] [color={rgb, 255:red, 0; green, 0; blue, 255 }  ,draw opacity=1 ][line width=0.75]    (6.56,-2.94) .. controls (4.17,-1.38) and (1.99,-0.4) .. (0,0) .. controls (1.99,0.4) and (4.17,1.38) .. (6.56,2.94)   ;
\draw [color={rgb, 255:red, 0; green, 0; blue, 255 }  ,draw opacity=1 ] [dash pattern={on 0.84pt off 2.51pt}]  (359.47,83.68) -- (387.41,83.68) ;
\draw [shift={(389.41,83.68)}, rotate = 180] [color={rgb, 255:red, 0; green, 0; blue, 255 }  ,draw opacity=1 ][line width=0.75]    (6.56,-2.94) .. controls (4.17,-1.38) and (1.99,-0.4) .. (0,0) .. controls (1.99,0.4) and (4.17,1.38) .. (6.56,2.94)   ;
\draw [line width=2.25]    (330,54.21) -- (344.74,39.47) ;
\draw [line width=2.25]    (359.93,54.21) -- (347.74,42.47) ;
\draw [line width=2.25]    (344.74,39.47) -- (360.47,23.74) ;
\draw [line width=2.25]    (389.95,55.21) -- (362.47,27.74) ;
\draw [line width=2.25]    (359.47,24.74) -- (359.47,10) ;
\draw [line width=2.25]    (359.47,105.79) -- (375.21,121.53) ;
\draw [line width=2.25]    (377.21,117.53) -- (389.95,104.79) ;
\draw [line width=2.25]    (362.47,132.26) -- (374.21,120.53) ;
\draw [line width=2.25]    (359.47,135.26) -- (330,105.79) ;
\draw [line width=2.25]    (359.47,150) -- (359.47,135.26) ;
\draw [line width=2.25]    (451.05,105.79) -- (451.05,53.21) ;
\draw [line width=2.25]    (480.53,105.79) -- (480.99,53.21) ;
\draw [line width=2.25]    (510.46,105.79) -- (510.92,54.21) ;
\draw [color={rgb, 255:red, 0; green, 0; blue, 255 }  ,draw opacity=1 ] [dash pattern={on 0.84pt off 2.51pt}]  (451.05,76.32) -- (478.99,76.32) ;
\draw [shift={(480.99,76.32)}, rotate = 180] [color={rgb, 255:red, 0; green, 0; blue, 255 }  ,draw opacity=1 ][line width=0.75]    (6.56,-2.94) .. controls (4.17,-1.38) and (1.99,-0.4) .. (0,0) .. controls (1.99,0.4) and (4.17,1.38) .. (6.56,2.94)   ;
\draw [color={rgb, 255:red, 0; green, 0; blue, 255 }  ,draw opacity=1 ] [dash pattern={on 0.84pt off 2.51pt}]  (482.53,83.68) -- (510.46,83.68) ;
\draw [shift={(480.53,83.68)}, rotate = 0] [color={rgb, 255:red, 0; green, 0; blue, 255 }  ,draw opacity=1 ][line width=0.75]    (6.56,-2.94) .. controls (4.17,-1.38) and (1.99,-0.4) .. (0,0) .. controls (1.99,0.4) and (4.17,1.38) .. (6.56,2.94)   ;
\draw [line width=2.25]    (451.05,54.21) -- (465.79,39.47) ;
\draw [line width=2.25]    (480.99,54.21) -- (468.79,42.47) ;
\draw [line width=2.25]    (465.79,39.47) -- (481.53,23.74) ;
\draw [line width=2.25]    (511,55.21) -- (483.53,27.74) ;
\draw [line width=2.25]    (480.53,24.74) -- (480.53,10) ;
\draw [line width=2.25]    (480.53,105.79) -- (496.26,121.53) ;
\draw [line width=2.25]    (498.26,117.53) -- (511,104.79) ;
\draw [line width=2.25]    (483.53,132.26) -- (495.26,120.53) ;
\draw [line width=2.25]    (480.53,135.26) -- (451.05,105.79) ;
\draw [line width=2.25]    (480.53,150) -- (480.53,135.26) ;
\draw [line width=2.25]    (330,259.79) -- (330,207.21) ;
\draw [line width=2.25]    (359.47,259.79) -- (359.93,207.21) ;
\draw [line width=2.25]    (389.41,259.79) -- (389.87,208.21) ;
\draw [color={rgb, 255:red, 0; green, 0; blue, 255 }  ,draw opacity=1 ] [dash pattern={on 0.84pt off 2.51pt}]  (332,230.32) -- (359.93,230.32) ;
\draw [shift={(330,230.32)}, rotate = 0] [color={rgb, 255:red, 0; green, 0; blue, 255 }  ,draw opacity=1 ][line width=0.75]    (6.56,-2.94) .. controls (4.17,-1.38) and (1.99,-0.4) .. (0,0) .. controls (1.99,0.4) and (4.17,1.38) .. (6.56,2.94)   ;
\draw [color={rgb, 255:red, 0; green, 0; blue, 255 }  ,draw opacity=1 ] [dash pattern={on 0.84pt off 2.51pt}]  (359.47,237.68) -- (387.41,237.68) ;
\draw [shift={(389.41,237.68)}, rotate = 180] [color={rgb, 255:red, 0; green, 0; blue, 255 }  ,draw opacity=1 ][line width=0.75]    (6.56,-2.94) .. controls (4.17,-1.38) and (1.99,-0.4) .. (0,0) .. controls (1.99,0.4) and (4.17,1.38) .. (6.56,2.94)   ;
\draw [line width=2.25]    (330,208.21) -- (344.74,193.47) ;
\draw [line width=2.25]    (359.93,208.21) -- (347.74,196.47) ;
\draw [line width=2.25]    (344.74,193.47) -- (360.47,177.74) ;
\draw [line width=2.25]    (389.95,209.21) -- (362.47,181.74) ;
\draw [line width=2.25]    (359.47,178.74) -- (359.47,164) ;
\draw [line width=2.25]    (359.47,259.79) -- (375.21,275.53) ;
\draw [line width=2.25]    (377.21,271.53) -- (389.95,258.79) ;
\draw [line width=2.25]    (362.47,286.26) -- (374.21,274.53) ;
\draw [line width=2.25]    (359.47,289.26) -- (330,259.79) ;
\draw [line width=2.25]    (359.47,304) -- (359.47,289.26) ;
\draw [line width=2.25]    (451.05,259.79) -- (451.05,207.21) ;
\draw [line width=2.25]    (480.53,259.79) -- (480.99,207.21) ;
\draw [line width=2.25]    (510.46,259.79) -- (510.92,208.21) ;
\draw [color={rgb, 255:red, 0; green, 0; blue, 255 }  ,draw opacity=1 ] [dash pattern={on 0.84pt off 2.51pt}]  (453.05,230.32) -- (480.99,230.32) ;
\draw [shift={(451.05,230.32)}, rotate = 0] [color={rgb, 255:red, 0; green, 0; blue, 255 }  ,draw opacity=1 ][line width=0.75]    (6.56,-2.94) .. controls (4.17,-1.38) and (1.99,-0.4) .. (0,0) .. controls (1.99,0.4) and (4.17,1.38) .. (6.56,2.94)   ;
\draw [color={rgb, 255:red, 0; green, 0; blue, 255 }  ,draw opacity=1 ] [dash pattern={on 0.84pt off 2.51pt}]  (482.53,237.68) -- (510.46,237.68) ;
\draw [shift={(480.53,237.68)}, rotate = 0] [color={rgb, 255:red, 0; green, 0; blue, 255 }  ,draw opacity=1 ][line width=0.75]    (6.56,-2.94) .. controls (4.17,-1.38) and (1.99,-0.4) .. (0,0) .. controls (1.99,0.4) and (4.17,1.38) .. (6.56,2.94)   ;
\draw [line width=2.25]    (451.05,208.21) -- (465.79,193.47) ;
\draw [line width=2.25]    (480.99,208.21) -- (468.79,196.47) ;
\draw [line width=2.25]    (465.79,193.47) -- (481.53,177.74) ;
\draw [line width=2.25]    (511,209.21) -- (483.53,181.74) ;
\draw [line width=2.25]    (480.53,178.74) -- (480.53,164) ;
\draw [line width=2.25]    (480.53,259.79) -- (496.26,275.53) ;
\draw [line width=2.25]    (498.26,271.53) -- (511,258.79) ;
\draw [line width=2.25]    (483.53,286.26) -- (495.26,274.53) ;
\draw [line width=2.25]    (480.53,289.26) -- (451.05,259.79) ;
\draw [line width=2.25]    (480.53,304) -- (480.53,289.26) ;

\draw (59,67.6) node [anchor=south] [inner sep=0.75pt]    {$( 1$};
\draw (109,67.6) node [anchor=south] [inner sep=0.75pt]    {$2)$};
\draw (161,67.6) node [anchor=south] [inner sep=0.75pt]    {$3$};
\draw (60,243.4) node [anchor=north] [inner sep=0.75pt]    {$1$};
\draw (110,243.4) node [anchor=north] [inner sep=0.75pt]    {$( 2$};
\draw (160,243.4) node [anchor=north] [inner sep=0.75pt]    {$3)$};
\draw (417.5,80) node    {$+$};
\draw (537.37,80) node    {$+$};
\draw (417.5,234) node    {$+$};

\end{tikzpicture}\]
    \caption{The map $\alpha$ applied to an element of $\PaCD(3)$.}
    \label{fig:alphamap}
\end{figure}
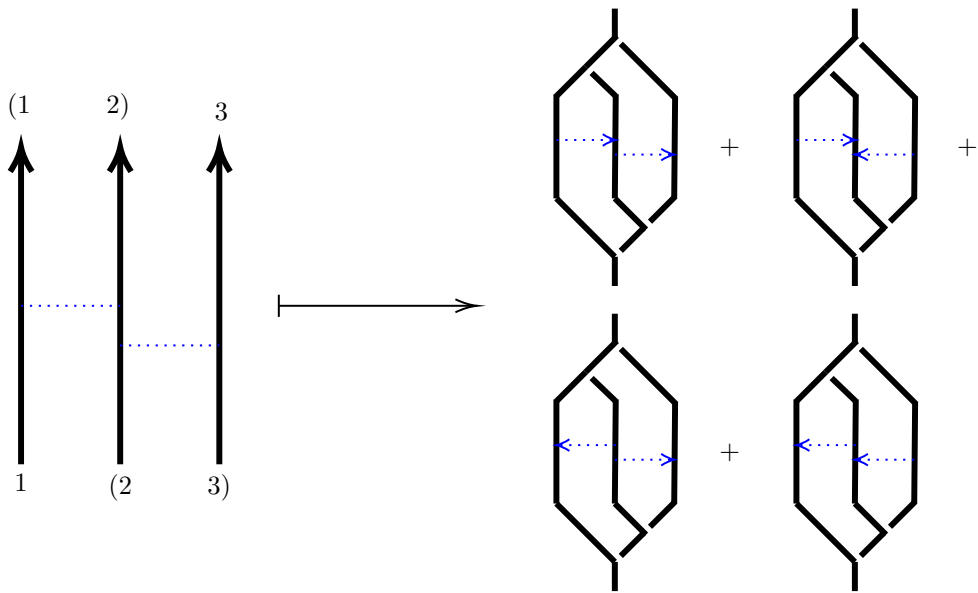

The unzip and stacking operations in arrow diagrams allow us to extend this inclusion of objects to a tensor functor embedding $\PaCD$ into $\arrows$ (Illustrated in Figure~\ref{fig:ClosureComp}, see Theorem 6.4 \cite{DHR21} for full details).

\begin{figure}
\tikzset{every picture/.style={line width=0.75pt}} 
\[\begin{tikzpicture}[x=0.75pt,y=0.75pt,yscale=-1,xscale=1]

\draw [line width=2.25]    (55,161) -- (55,55) ;
\draw [shift={(55,51)}, rotate = 90] [color={rgb, 255:red, 0; green, 0; blue, 0 }  ][line width=2.25]    (12.24,-5.49) .. controls (7.79,-2.58) and (3.71,-0.75) .. (0,0) .. controls (3.71,0.75) and (7.79,2.58) .. (12.24,5.49)   ;
\draw [line width=2.25]    (85,161) -- (85,55) ;
\draw [shift={(85,51)}, rotate = 90] [color={rgb, 255:red, 0; green, 0; blue, 0 }  ][line width=2.25]    (12.24,-5.49) .. controls (7.79,-2.58) and (3.71,-0.75) .. (0,0) .. controls (3.71,0.75) and (7.79,2.58) .. (12.24,5.49)   ;
\draw [line width=2.25]    (115,161) -- (115,55) ;
\draw [shift={(115,51)}, rotate = 90] [color={rgb, 255:red, 0; green, 0; blue, 0 }  ][line width=2.25]    (12.24,-5.49) .. controls (7.79,-2.58) and (3.71,-0.75) .. (0,0) .. controls (3.71,0.75) and (7.79,2.58) .. (12.24,5.49)   ;
\draw  [color={rgb, 255:red, 255; green, 0; blue, 0 }  ,draw opacity=1 ][fill={rgb, 255:red, 255; green, 255; blue, 255 }  ,fill opacity=1 ] (50,95) -- (120,95) -- (120,117) -- (50,117) -- cycle ;
\draw [line width=2.25]    (56,340) -- (56,234) ;
\draw [shift={(56,230)}, rotate = 90] [color={rgb, 255:red, 0; green, 0; blue, 0 }  ][line width=2.25]    (12.24,-5.49) .. controls (7.79,-2.58) and (3.71,-0.75) .. (0,0) .. controls (3.71,0.75) and (7.79,2.58) .. (12.24,5.49)   ;
\draw [line width=2.25]    (86,339) -- (86,233) ;
\draw [shift={(86,229)}, rotate = 90] [color={rgb, 255:red, 0; green, 0; blue, 0 }  ][line width=2.25]    (12.24,-5.49) .. controls (7.79,-2.58) and (3.71,-0.75) .. (0,0) .. controls (3.71,0.75) and (7.79,2.58) .. (12.24,5.49)   ;
\draw [line width=2.25]    (116,339) -- (116,233) ;
\draw [shift={(116,229)}, rotate = 90] [color={rgb, 255:red, 0; green, 0; blue, 0 }  ][line width=2.25]    (12.24,-5.49) .. controls (7.79,-2.58) and (3.71,-0.75) .. (0,0) .. controls (3.71,0.75) and (7.79,2.58) .. (12.24,5.49)   ;
\draw  [color={rgb, 255:red, 0; green, 0; blue, 255 }  ,draw opacity=1 ][fill={rgb, 255:red, 255; green, 255; blue, 255 }  ,fill opacity=1 ] (51,273) -- (121,273) -- (121,295) -- (51,295) -- cycle ;
\draw [line width=2.25]    (265.3,295.79) -- (265.3,243.21) ;
\draw [line width=2.25]    (294.77,295.79) -- (295.23,243.21) ;
\draw [line width=2.25]    (324.7,295.79) -- (325.16,244.21) ;
\draw [line width=2.25]    (265.3,244.21) -- (280.03,229.47) ;
\draw [line width=2.25]    (295.23,244.21) -- (283.03,232.47) ;
\draw [line width=2.25]    (280.03,229.47) -- (295.77,213.74) ;
\draw [line width=2.25]    (325.24,245.21) -- (297.77,217.74) ;
\draw [line width=2.25]    (294.77,214.74) -- (294.77,200) ;
\draw [line width=2.25]    (294.77,295.79) -- (310.51,311.53) ;
\draw [line width=2.25]    (312.51,307.53) -- (325.24,294.79) ;
\draw [line width=2.25]    (297.77,322.26) -- (309.51,310.53) ;
\draw [line width=2.25]    (294.77,325.26) -- (265.3,295.79) ;
\draw [line width=2.25]    (294.77,340) -- (294.77,325.26) ;
\draw  [color={rgb, 255:red, 0; green, 0; blue, 255 }  ,draw opacity=1 ][fill={rgb, 255:red, 255; green, 255; blue, 255 }  ,fill opacity=1 ] (260,258.5) -- (330,258.5) -- (330,280.5) -- (260,280.5) -- cycle ;
\draw [line width=2.25]    (324.73,135.75) -- (324.73,83.13) ;
\draw [line width=2.25]    (295.23,135.75) -- (294.77,83.13) ;
\draw [line width=2.25]    (265.27,135.75) -- (264.81,84.13) ;
\draw [line width=2.25]    (324.73,84.13) -- (309.98,69.38) ;
\draw [line width=2.25]    (294.77,84.13) -- (306.98,72.38) ;
\draw [line width=2.25]    (309.98,69.38) -- (294.23,53.63) ;
\draw [line width=2.25]    (264.73,85.13) -- (292.23,57.63) ;
\draw [line width=2.25]    (295.23,54.63) -- (295.23,39.88) ;
\draw [line width=2.25]    (295.23,135.75) -- (279.48,151.5) ;
\draw [line width=2.25]    (277.48,147.5) -- (264.73,134.75) ;
\draw [line width=2.25]    (292.23,162.25) -- (280.48,150.5) ;
\draw [line width=2.25]    (295.23,165.25) -- (324.73,135.75) ;
\draw [line width=2.25]    (295.23,180) -- (295.23,165.25) ;
\draw  [color={rgb, 255:red, 255; green, 0; blue, 0 }  ,draw opacity=1 ][fill={rgb, 255:red, 255; green, 255; blue, 255 }  ,fill opacity=1 ] (260,98.44) -- (330,98.44) -- (330,120.44) -- (260,120.44) -- cycle ;
\draw [line width=2.25]    (295.23,180) -- (294.77,200) ;
\draw    (350,190) -- (428,190) ;
\draw [shift={(430,190)}, rotate = 180] [color={rgb, 255:red, 0; green, 0; blue, 0 }  ][line width=0.75]    (10.93,-3.29) .. controls (6.95,-1.4) and (3.31,-0.3) .. (0,0) .. controls (3.31,0.3) and (6.95,1.4) .. (10.93,3.29)   ;
\draw    (160,190) -- (238,190) ;
\draw [shift={(240,190)}, rotate = 180] [color={rgb, 255:red, 0; green, 0; blue, 0 }  ][line width=0.75]    (10.93,-3.29) .. controls (6.95,-1.4) and (3.31,-0.3) .. (0,0) .. controls (3.31,0.3) and (6.95,1.4) .. (10.93,3.29)   ;
\draw [line width=2.25]    (524.73,190) -- (524.73,137.37) ;
\draw [line width=2.25]    (495.23,190) -- (494.77,137.37) ;
\draw [line width=2.25]    (465.27,190) -- (464.81,138.38) ;
\draw [line width=2.25]    (524.73,138.38) -- (509.98,123.63) ;
\draw [line width=2.25]    (494.77,138.38) -- (506.98,126.63) ;
\draw [line width=2.25]    (509.98,123.63) -- (494.23,107.88) ;
\draw [line width=2.25]    (464.73,139.38) -- (492.23,111.88) ;
\draw [line width=2.25]    (495.23,108.88) -- (495.23,94.13) ;
\draw  [color={rgb, 255:red, 255; green, 0; blue, 0 }  ,draw opacity=1 ][fill={rgb, 255:red, 255; green, 255; blue, 255 }  ,fill opacity=1 ] (460,152.69) -- (530,152.69) -- (530,174.69) -- (460,174.69) -- cycle ;
\draw [line width=2.25]    (465.3,235.79) -- (465.3,183.21) ;
\draw [line width=2.25]    (494.77,235.79) -- (495.23,183.21) ;
\draw [line width=2.25]    (524.7,235.79) -- (525.16,184.21) ;
\draw [line width=2.25]    (494.77,235.79) -- (510.51,251.53) ;
\draw [line width=2.25]    (512.51,247.53) -- (525.24,234.79) ;
\draw [line width=2.25]    (497.77,262.26) -- (509.51,250.53) ;
\draw [line width=2.25]    (494.77,265.26) -- (465.3,235.79) ;
\draw [line width=2.25]    (494.77,280) -- (494.77,265.26) ;
\draw  [color={rgb, 255:red, 0; green, 0; blue, 255 }  ,draw opacity=1 ][fill={rgb, 255:red, 255; green, 255; blue, 255 }  ,fill opacity=1 ] (460,198.5) -- (530,198.5) -- (530,220.5) -- (460,220.5) -- cycle ;

\draw (55,37.6) node [anchor=south] [inner sep=0.75pt]    {$1$};
\draw (85,37.6) node [anchor=south] [inner sep=0.75pt]    {$( 2$};
\draw (115,37.6) node [anchor=south] [inner sep=0.75pt]    {$3)$};
\draw (55,164.4) node [anchor=north] [inner sep=0.75pt]    {$( 1$};
\draw (85,164.4) node [anchor=north] [inner sep=0.75pt]    {$2)$};
\draw (115,164.4) node [anchor=north] [inner sep=0.75pt]    {$3$};
\draw (56,216.6) node [anchor=south] [inner sep=0.75pt]    {$( 1$};
\draw (86,215.6) node [anchor=south] [inner sep=0.75pt]    {$2)$};
\draw (116,215.6) node [anchor=south] [inner sep=0.75pt]    {$3$};
\draw (56,343.4) node [anchor=north] [inner sep=0.75pt]    {$1$};
\draw (86,342.4) node [anchor=north] [inner sep=0.75pt]    {$( 2$};
\draw (116,342.4) node [anchor=north] [inner sep=0.75pt]    {$3)$};
\draw (200,186.6) node [anchor=south] [inner sep=0.75pt]    {$m\circ \alpha ^{2}$};
\draw (390,186.6) node [anchor=south] [inner sep=0.75pt]    {$u^{2}$};

\end{tikzpicture}
\]
\caption{Closure and composition of chord diagrams, with $m$ denoting the stacking map.}\label{fig:ClosureComp}
\end{figure}
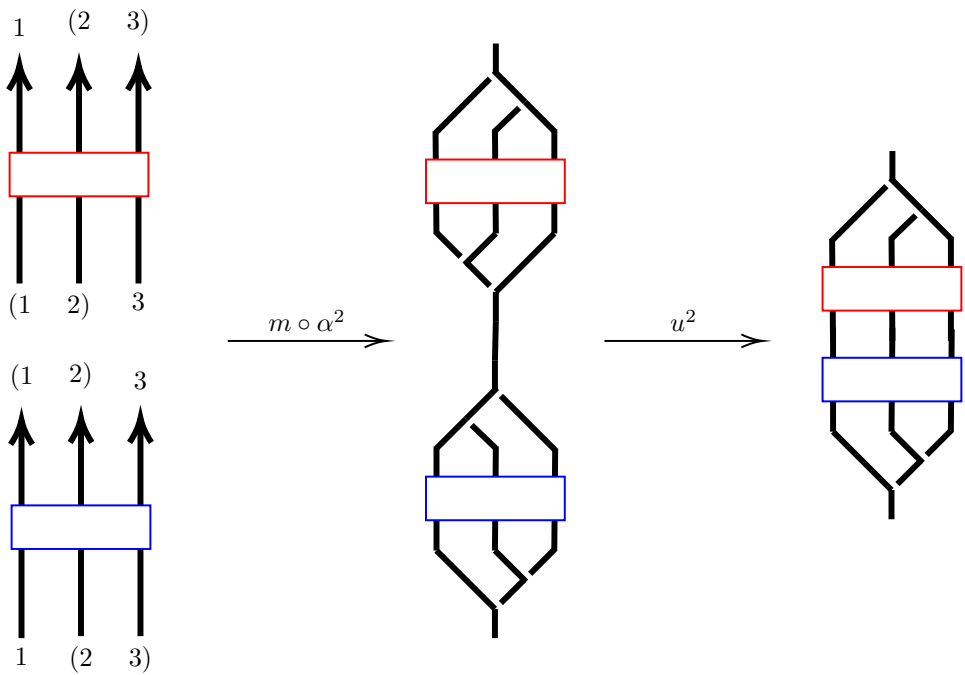

\begin{definition}\label{def: alpha}
The functor $\alpha:\PaCD\rightarrow \arrows$ is defined as follows. 

\begin{itemize}
    \item The map $\alpha$ sends a pair of parenthesised words (i.e. a skeleton as in Figure~\ref{fig:pacdexample}) to the $w$-foam skeleton that is the composition of two oriented binary, rooted trees representing the two parenthesisations. Here, each pair of parentheses corresponds to a vertex, and these are joined together to obtain each tree (see Figure~\ref{fig:alphamap}). 
    \item For a morphism $f\in\PaCD(s) \cong U(\mathfrak{t}_n)$, we define $$\alpha(f) := \epsilon(f)\in \hat{U}(\tder_n\oplus \mathfrak{a}_n \ltimes \cyc_n).$$
\end{itemize}
\end{definition}

The functor $\alpha:\PaCD\rightarrow \arrows$ is a morphism of props (or tensor categories) in that it picks out the prop structure of $\PaCD$ inside the morphism space of $\A$. The symmetric monoidal product on both $\PaCD$ and $\arrows$ is given by concatenation-- visually this is just placing chord (respectively, arrow) diagrams next to each other. In Theorem 6.4 \cite{DHR21}, the authors show that this functor actually has more structure, since the inclusion of morphisms $\PaCD(s)\hookrightarrow\arrows ((s)^{cl})$ preserves co-products. Combining these two facts makes $\alpha$ a morphism of props enriched in co-associative co-algebras. However, we emphasise that $\PaCD$ is not a sub-object of $\arrows$ as the image of $\alpha$ is not preserved under all operations of $\arrows$.

The closure map can also be applied to a braid in $\PaB$ in the same way, and the result can then be interpreted as an element of $\wf$: this defines a set map $a: \PaB \to \wf$, which can also be made into a tensor functor. We don't use this functor for any of the results in this paper, so we refer the reader to \cite[Section 1.2.2]{BND:WKO3}) for a topological description.

Using this relationship between $\PaB$ and $\wf$, it is possible to construct a homomorphic expansion for $\wf$ directly from a homomorphic expansion for $\PaB$, in parallel with the \cite{ATE10} formula for KV solutions in terms of Drinfel'd associators. This is the main result of \cite{BND:WKO3}:

\begin{theorem}\cite[Theorem 1.1]{BND:WKO3}
Every homomorphic expansion $Z^b:\mathbb{Q}[\widehat{\PaB}] \to \PaCD$ extends uniquely to a homomorphic expansion $Z: \mathbb{Q}[\widehat{\wf}] \to \arrows$, in the sense that the following diagram commutes:
\begin{equation}\label{eq:WKO3Main}
\begin{tikzcd}
	{\mathbb{Q}[\widehat{\PaB}]} && {\PaCD} \\
	\\
	{\mathbb{Q}[\widehat{\wf}]} && {\A}
	\arrow["Z^b", from=1-1, to=1-3]
	\arrow["Z"', dashed, from=3-1, to=3-3]
	\arrow["a"', from=1-1, to=3-1]
	\arrow["\alpha", from=1-3, to=3-3]
\end{tikzcd}
\end{equation}
\end{theorem}
Theorem 1.1 of \cite{BND:WKO3} gives an explicit construction for $Z$ from $Z^b$, in particular a formula for $Z(V)$ and $Z(C)$ in terms of the Drinfel'd associator $\Phi=Z^b(\Associator)$. We are not using this result in the current paper, however, the construction used in Section \ref{sec:map} to construct a diagrammatic Alekseev-Torossian map follows the techniques used in \cite{BND:WKO3} to construct $Z$ from $Z^b$.

\section{Deriving a map of expansion-preserving automorphisms}\label{sec:map}

In \cite{AT12} and \cite{ATE10}, Alekseev, Enriquez and Torossian illustrated the deep relationship between Drinfel'd associators and solutions to the Kashiwara-Vergne conjecture.  In particular, they construct an injective group homomorphism between the graded Grothendieck-Teichm\"{u}ller and the graded Kashiwara-Vergne group. 
\begin{theorem}
\cite[Theorem 2.5]{ATE10} The map $\rho: \grt_1 \to \krv$ given by 
\begin{align}\label{eq:rhodef}
    \Psi(x,y) \mapsto (\Psi(-x-y,x),\Psi(-x-y,y))
\end{align}
is an injective group homomorphism. 
\end{theorem}

In other words, the map $\rho$ defines a subgroup of $\krv$ which is \emph{isomorphic} to $\grt_1$. Moreover, when restricted to the set of KV solutions constructed from Drinfel'd associators, the action of the subgroup of $\krv$ isomorphic to $\grt_1$ coincides with the canonical action of $\grt_1$ on Drinfel'd associators (\cite[Prop 9.13]{AT12}). 

 In this section, we construct a map $\rp: \Aut(\PaCD) \to \Aut_v(\A)$ from automorphisms of parenthesised chord diagrams to automorphisms of arrow diagrams which makes the following diagram commute: 
 \begin{equation}\label{eq:rp}
 \begin{tikzcd}
\Aut(\PaCD) \arrow[d, "\tau", swap]\arrow[r,"\rp", dashed] & \Aut_v(\arrows) \arrow[d, "\Theta"] \\
\grt_1 \arrow[r,"\rho"] & \krv
 \end{tikzcd}
 \end{equation}
 In particular, we identify a subgroup of arrow diagram automorphisms isomorphic to $\grt_1$.

\subsection{Comparing automorphisms of chord diagrams and arrow diagrams.} 
In Section~\ref{sec:PaCDtoA} we described a tensor functor $\alpha:\PaCD\rightarrow\arrows$ induced by the inclusion of Lie algebras $\mathfrak{t}_n\hookrightarrow \tder_n$ (Definition~\ref{def: alpha}). The goal for this section is to construct the map $\rp: \Aut(\PaCD) \to \Aut_v(\A)$. The plan is to show that, given $F\in \operatorname{Aut}(\PaCD)$, then we can construct $G \in \operatorname{Aut}_v(\A)$, such that the following diagram commutes:
\begin{equation}\label{eq:GFcommute}
\begin{tikzcd}
	{\PaCD} && {\A} \\
	\\
	{\PaCD} && {\A}.
	\arrow["\alpha", from=1-1, to=1-3]
	\arrow["\alpha"', from=3-1, to=3-3]
	\arrow["F"', from=1-1, to=3-1]
	\arrow["G", dashed, from=1-3, to=3-3]
\end{tikzcd}
\end{equation}
We will then show that setting $\rp(F)=G$ fits into the diagram \eqref{eq:rp}.
Before we give the construction, we need to establish some facts about arrow diagrams.

\begin{lemma}\label{lem:EK}
\cite[Lemma 2.4]{BND:WKO3} There is an isomorphism $$\varphi: \A(\pcv) \stackrel{\cong}{\longrightarrow} \A(\uparrow).$$ 
\end{lemma}

Since familiarity with the construction of this isomorphism will be necessary, we give a sketch of the proof:
\begin{proof}
The isomorphism $\varphi$ is shown in Figure \ref{fig:phi_iso}. Given any arrow diagram $\bD \in \arrows(\pcv)$, by the TF relation, the punctured strand has only arrow heads ending on it. These arrow heads are collectively denoted by $h$. For any element of $\arrows(\pcv)$ there is a representative with only arrow tails on the right strand (achievable via STU and CP relations): denote these arrow endings by $t$. Finally, we denote arrow endings on the top strand (both heads and tails) by $x$. Then, by repeated applications of the VI relation, $t$ can be ``pushed'' to the top strand: VI in this case does not result in a sum, as the other strand is punctured. After clearing the capped strand this way, use repeated VI relations to push $h$ to the top strand as well. Once again, no sums appear by the CP relation, as the capped strand no longer has any arrow endings. The result of this process is $\varphi (\bD)$. 

We leave it to the reader to check that $\varphi$ is well-defined, and that its inverse is the natural inclusion of $\arrows(\uparrow)$ as the top strand in $\arrows(\pcv)$.
\end{proof}

\begin{figure}[h]
    \centering

\tikzset{every picture/.style={line width=0.75pt}} 

\[\begin{tikzpicture}[x=0.75pt,y=0.75pt,yscale=-1,xscale=1]

\draw [line width=2.25]    (94,200) -- (94,64) ;
\draw [shift={(94,60)}, rotate = 90] [color={rgb, 255:red, 0; green, 0; blue, 0 }  ][line width=2.25]    (12.24,-5.49) .. controls (7.79,-2.58) and (3.71,-0.75) .. (0,0) .. controls (3.71,0.75) and (7.79,2.58) .. (12.24,5.49)   ;
\draw [shift={(94,200)}, rotate = 270] [color={rgb, 255:red, 0; green, 0; blue, 0 }  ][fill={rgb, 255:red, 0; green, 0; blue, 0 }  ][line width=2.25]      (0, 0) circle [x radius= 3.75, y radius= 3.75]   ;
\draw [color={rgb, 255:red, 255; green, 0; blue, 0 }  ,draw opacity=1 ]   (44,200) .. controls (44,179.33) and (40,160.33) .. (94,130) ;
\draw [line width=2.25]    (210,200) -- (210,64) ;
\draw [shift={(210,60)}, rotate = 90] [color={rgb, 255:red, 0; green, 0; blue, 0 }  ][line width=2.25]    (12.24,-5.49) .. controls (7.79,-2.58) and (3.71,-0.75) .. (0,0) .. controls (3.71,0.75) and (7.79,2.58) .. (12.24,5.49)   ;
\draw    (120,130) -- (188,130) ;
\draw [shift={(190,130)}, rotate = 180] [color={rgb, 255:red, 0; green, 0; blue, 0 }  ][line width=0.75]    (10.93,-3.29) .. controls (6.95,-1.4) and (3.31,-0.3) .. (0,0) .. controls (3.31,0.3) and (6.95,1.4) .. (10.93,3.29)   ;

\draw (101,92.4) node [anchor=north west][inner sep=0.75pt]    {$x$};
\draw (41,142.4) node [anchor=north west][inner sep=0.75pt]    {$h$};
\draw (101,152.4) node [anchor=north west][inner sep=0.75pt]    {$t$};
\draw (155,126.6) node [anchor=south] [inner sep=0.75pt]    {$\varphi $};
\draw (217,92.4) node [anchor=north west][inner sep=0.75pt]    {$x$};
\draw (217,122.4) node [anchor=north west][inner sep=0.75pt]    {$t$};
\draw (217,152.4) node [anchor=north west][inner sep=0.75pt]    {$h$};

\end{tikzpicture}\]

    \caption{The isomorphism $\varphi$ applied to the arrow diagram $\bD$ on the left.}
    \label{fig:phi_iso}
\end{figure}
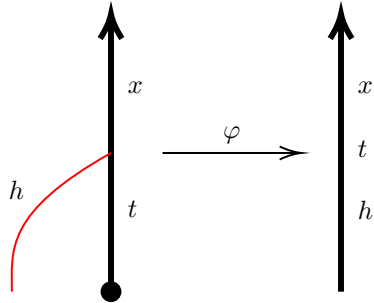

The next Lemma shows that puncturing and capping a vertex trivialises its $G$-value. 

\begin{lemma}\label{lem:Gpuncturecap}
For any $G\in \Aut_v(\arrows)$, $G(\pv) = 1$, that is, $G(\pv)=\pv$ with no arrows. 
\end{lemma}

\begin{proof}
\cite[Lemma 2.8]{BND:WKO3}, states that for any homomorphic expansion $Z$ of $\wf$, $Z(\pv) = 1$. In other words, $Z(\pv)$ is the skeleton element $\pv \in \arrows$, with no arrows. 
Since automorphisms $G\in \Aut_{v}(\arrows)$ are expansion-preserving, the result is immediate.
\end{proof}

We are now ready to construct $G$ from $F$ as in \eqref{eq:GFcommute}. 

\begin{cons}\label{cons:buckle}
By Corollary~\ref{auto_wheels}, $G$ is uniquely determined by $G(\vertex)^{tr}$, and hence the idea is to compute $G(\vertex)^{tr}$ from $F(\beta)$ for a specific, suitably chosen $\beta\in \PaCD$, depicted on the left side of Figure~\ref{fig:beta}. Given and $F\in\Aut(\PaCD)$ with $F(\Associator)=\Psi(t^{12}, t^{23})$, we compute $F(\beta)$:
\begin{align*}
   F(\beta)= (\Psi^{(13)24})^{-1}\, \Psi^{132}\, \virtualcrossing^{23}\, (\Psi^{123})^{-1}\, \Psi^{(12)34}.
\end{align*} This is shown on the left side of Figure~\ref{fig:betap}.
Using $F(\beta)$, we determine $G(\vertex)^{tr}$ as follows: 
\begin{enumerate}
\item Applying $\alpha$ (Definition~\ref{def: alpha}) to $F(\beta)$ we obtain an arrow diagram in $\arrows$ with skeleton $\beta^{cl} = \alpha(\beta)$. Any $G\in\Aut_v(\A)$ fitting into the commutative diagram \eqref{eq:GFcommute} satisfies  $\alpha(F(\beta)) = G(\beta^{cl})$, and thus we know the value $G(\beta^{cl})$. See Figure~\ref{fig:betap}. 
\item To find the value of $G(\vertex)^{tr}$, we perform a series of operations on $G(\beta^{cl})$, as shown in Figure~\ref{fig:Stripping} and explained below.  First, we pre-compose with a capped vertex $\cvtikz$. Then we perform two unzip operations: the first on the top strand of the new vertex, and the second a disk unzip on the capped strand. Finally, we puncture the two bottom uncapped strands to obtain an arrow diagram on a skeleton that is a double capped, double punctured vertex denoted $W$. 
\item Since $G\in \Aut(\arrows)$, it commutes with all the prop and auxiliary operations of $\arrows$, and hence we can compute $G(W)$ from $G(\beta^{cl})$, as shown in the bottom line of Figure~\ref{fig:Stripping}. Due to the composition with $\cvtikz$ in the previous step, $G(W)$ contains involves copies of $N=G(\vertex)$ and $C=G(\bcap)$: this is shown in detail in Figure~\ref{fig:Stripping}; schematically,  $G(\cvtikz\beta^{cl})=G(\bcap)G(\vertex)G(\beta^{cl})$ and thus $G(W) = p^2 u^2\left( G(\bcap)G(\vertex)G(\beta^{cl})\right)$.
\item By Lemma~\ref{lem:Gpuncturecap} we have $G(\pv)$ = 1, and therefore the copy of $G(\vertex)$ in $G(W)$ from the composition with $\cvtikz$ cancels after the puncture operations. This is shown in the rightmost column of Figure~\ref{fig:Stripping}. Thus, $G(W) = u(C) p^2G(\beta^{cl})$. 
\item On the other hand, from the fact that $G$ is a prop automorphism, $G(W)$ is a composite of $G(\vertex)$, two copies of $G(\pv)$ and two copies of $G(\bcap)$. Again by Lemma~\ref{lem:Gpuncturecap}, $G(\pv)=1$, and therefore, schematically, $G(W)=G(\bcap) G(\bcap)G(\vertex) $. Since $G(\bcap)\in \arrows(\bcap)$, by Lemma~\ref{cyclic_lemma} $(G(\bcap))^{tr}=1$, and therefore $(G(W))^{tr}=(G(\vertex))^{tr}=N^{tr}$, which is the value we need to compute.
\item Combining (4) and (5) we have $N^{tr}=\left( u(C) p^2G(\beta^{cl})\right)^{tr}=(p^2G(\beta^{cl}))^{tr}$.
\item In order to express $N^{tr}$ in its standard form as an element of $\arrows(\uparrow_2)\cong \arrows(\vertex)$, rather than in $\arrows(W)$, we apply the isomorphism $\varphi$ of Lemma~\ref{lem:EK}: $\arrows(W)\stackrel{\varphi}{\cong} \arrows(\vertex)$.  Therefore,  $N^{tr}=\varphi(p^2G(\beta^{cl})$. 
\end{enumerate}
\end{cons}


\begin{figure}[h]
    \centering
    \begin{subfigure}[c]{\textwidth}
    \centering
\tikzset{every picture/.style={line width=0.75pt}} 
\[
\]

    \caption{The element $\beta\in \PaCD$ is a skeleton element of $\PaCD$. Its image under $\alpha$ is its closure.}
    \label{fig:beta}
    \end{subfigure}
    \begin{subfigure}[c]{\textwidth}
    \centering
\tikzset{every picture/.style={line width=0.75pt}} 
\[%
\]
    \caption{$F(\beta)$, expressed in terms of $\Psi$, is shown on the left. By the compatibility of $F$ and $G$, we know that $\alpha(F(\beta))=G(\alpha(\beta))$.}
    \label{fig:betap}
    \end{subfigure}
    \caption{}\label{fig:betas}
\end{figure}



\begin{figure}[h]
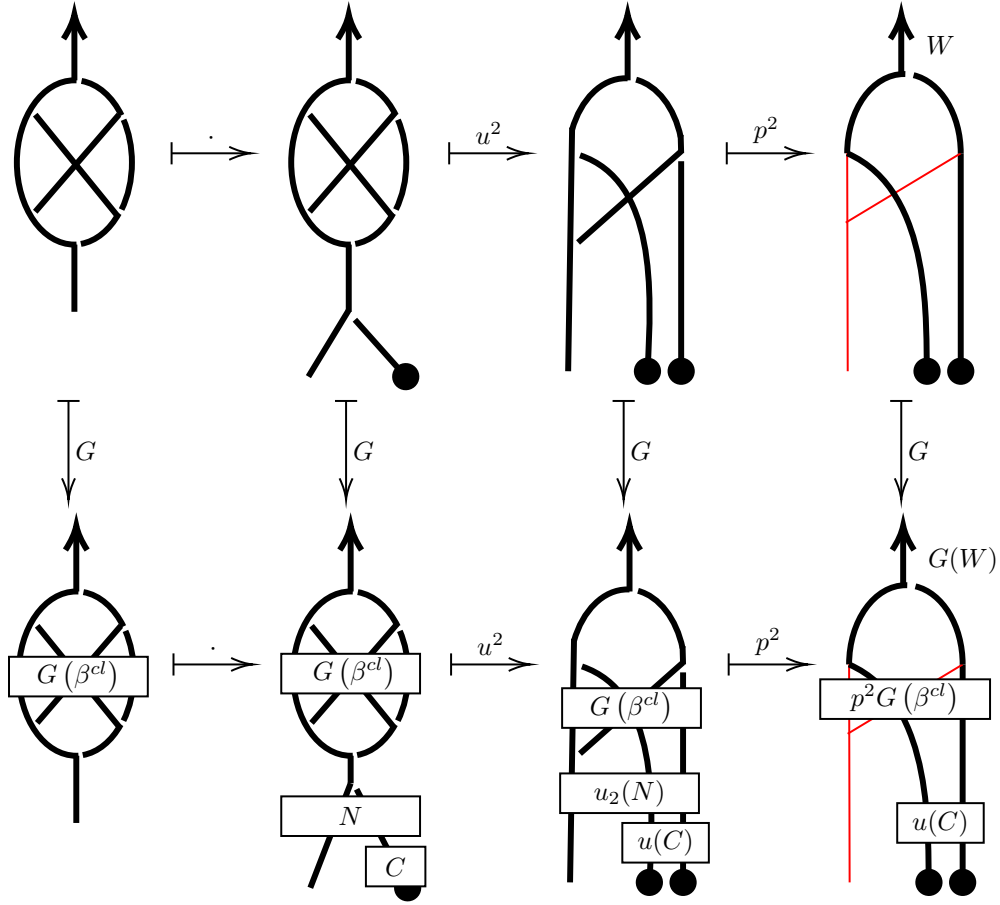

    \centering
\tikzset{every picture/.style={line width=0.75pt}} 
\[%
\]
    \caption{The stripping process.}\label{fig:Stripping}
\end{figure}



Construction~\ref{cons:buckle} shows that given $F\in \Aut(\PaCD)$, if a $G\in \Aut(\arrows)$ fitting into the diagram~\eqref{eq:GFcommute} exists, then $(G(\vertex))^{tr}$ is explicitly computable from $F$. In the following theorem, we compute the complete value of $G(\vertex)$:

\begin{theorem}\label{thm:GfromF}
Given an automorphism $F\in \Aut(\PaCD)$ with $F(\Associator)=\Psi=\Psi(t^{12},t^{23})$, then if there exists $G\in \Aut_v(\arrows)$ which completes the commutative diagram \eqref{eq:GFcommute}, then $G$ is uniquely determined by $F$ and 
    $N=G(\vertex)\in \arrows(\uparrow_2)$ and $C=G(\bcap)\in \Q[[\xi]]/\langle\xi\rangle$ satisfy
    \begin{equation}\label{eq:GfromF}
        N=(C^1)^{-1}(C^2)^{-1}\varphi\left(\Psi^{-1}\left(a^{2(13)},-a^{2(13)}-a^{4(13)}\right)\cdot \Psi(a^{23},a^{43})\right)C^{12},
    \end{equation}
    where $\varphi$ is the isomorphism of Lemma~\ref{lem:EK}, and $a^{ij}$ denotes an arrow from strand $i$ to strand $j$, and the strand numbering is as shown in Figure~\ref{fig:betas}.
\end{theorem}

To prove Theorem~\ref{thm:GfromF}, we first establish a few basic facts about $\grt_1$. 
\begin{lemma}\label{lem:GRT}
For $\Psi=\Psi(x,y)\in \grt_1$, the following properties hold:
\begin{enumerate}
    \item If $[x,y]=0$ then $\Psi(x,y)=1$.
    \item $\Psi(t^{ij},t^{jk})=\Psi(t^{ij},-t^{ij}-t^{ik}),$ where $t^{ij}$ denotes the chord between the $i$-th and $j$-th strands.
\end{enumerate}
\end{lemma}

\begin{proof}
The first property is a direct consequence of the defining relation (\ref{eq:grtinversion}). The second follows from the fact that $(t^{jk}+t^{ij}+t^{ik})$ is central in $\mathfrak{t}_n$ by 4T relation of chord diagrams $[t^{jk}, t^{ij}+t^{ik}]=0$ \cite[Definition 2.7]{BNGT}. Because of this, for any $X \in \mathfrak{t}_n$:
\begin{align*}
    [t^{jk}+t^{ij}+t^{ik}, X] &= 0\\
    \implies [t^{jk},X] &= -[t^{ij}+t^{ik},X]\\
    \implies [t^{jk},X] &= [-t^{ij}-t^{ik},X]
\end{align*}
and therefore, 
\begin{align}
    \Psi(t^{ij},t^{jk})=\Psi(t^{ij},-t^{ij}-t^{ik})
\end{align}
as required. 
\end{proof}

\begin{proof}[Proof of Theorem \ref{thm:GfromF}]
The fact that $F$ uniquely determines $G$ is almost immediate from Construction~\ref{cons:buckle}. By Lemma~\ref{auto_wheels}, $G\in \Aut_v(\arrows)$ is uniquely determined by $N^{tr}=\apr(G(\vertex))$. Step (6) of Construction~\ref{cons:buckle} shows how to compute $N^{tr}$ from $F(\beta)$.
Thus, $G$, is uniquely determined by $F$.

Given $F(\Associator)=\Psi(t^{12}, t^{23})$, we compute $F(\beta)$ as shown in the left-hand side of Figure~\ref{fig:betap}:
\begin{align}
   F(\beta)= (\Psi^{(13)24})^{-1}\, \Psi^{132}\, \virtualcrossing^{23}\, (\Psi^{123})^{-1}\, \Psi^{(12)34}.
\end{align}
From here on we will suppress the skeleton crossing $\virtualcrossing^{23}$, as it is part of the skeleton and commutes with chords. By diagram \eqref{eq:GFcommute}, we have 
\begin{align}\label{eq:Gbeta}
    G(\beta^{cl})=G(\alpha(\beta)) = \alpha(F(\beta)) = \alpha\left((\Psi^{(13)24})^{-1}\, \Psi^{132}\, (\Psi^{123})^{-1}\, \Psi^{(12)34}\right). 
\end{align}

We obtain the formula \eqref{eq:GfromF} from this by the sequence of operations detailed in Construction~\ref{cons:buckle}, and shown in the first row of Figure \ref{fig:Stripping} for the skeleton $\beta^{cl}$,  and in the second row for $G(\beta^{cl})$.
Recall that by step (6) of the construction,
$\arrows(W)\cong \arrows(\vertex)\cong \arrows(\uparrow_2).$

By step (5) of the construction, $G(W)$ is a circuit algebra composition of $G(\vertex)=N$ and two copies of $G(\bcap)=C$. Applying $\varphi: \arrows(W)\to \arrows(\uparrow_2)$, involves sliding the $C$ values onto the top two strands, we obtain $\varphi(G(W))=C^1C^2N$, and therefore,
\begin{equation}
N=(C^1)^{-1}(C^2)^{-1}\varphi(G(W)).
\end{equation}

Since $G$ commutes with all the operations performed throughout the `stripping' process, this allows us to compute $G(W)$, as shown in Figure~\ref{fig:Stripping} and in step (4) of the construction. In the strand numbering of Figure~\ref{fig:betap}, the punctured strands are 1 and 3 and the capped strands are 2 and 4, thus we make step (4) of the construction more precise:
\begin{equation}
    G(W)=C^{(24)}p_1p_3G(\beta^{cl})=C^{24}p_1p_3\left(\alpha(F(\beta))\right).
\end{equation}
Thus, using the formula \eqref{eq:Gbeta},
\begin{multline}\label{eq:GW}
G(W)=C^{24}p_1p_3\left(\alpha\left((\Psi^{(13)24})^{-1}\, \Psi^{132}\, (\Psi^{123})^{-1}\, \Psi^{(12)34}\right)\right)
\\
=C^{24}p_1p_3\left(\alpha\left(\Psi^{-1}(t^{(13)2},t^{24})\, \Psi(t^{13},t^{23})\, \Psi^{-1}(t^{12},t^{23})\, \Psi(t^{(12)3},t^{34})\right)\right),
\end{multline}
where $t^{ij}$ is the chord from $i$ to $j$ and $t^{(ij)k} = t^{ik} + t^{jk}$. 

By definition, $\alpha(t^{ij})=a^{ij}+a^{ji}$, where $a^{ij}$ denotes a horizontal arrow pointing from strand $i$ to strand $j$.
By the TF relation of arrow diagrams (Figure~\ref{fig:ArrowRelns}), arrow tails on a punctured strand are zero. In other words, $p_1p_3(a^{1m})=p_1p_3(a^{3m}))=0$ for any m. 

Then, using property (1) of Lemma~\ref{lem:GRT},
$$
    p_1p_2\alpha\left(\Psi(t^{13}, t^{23})\right) = \Psi(a^{13}+a^{31}, a^{23}+a^{32})=\Psi(0, a^{23}+a^{32})= 1.
$$

Similarly,
$$    p_1p_3\alpha\left(\Psi^{-1}(t^{12},t^{23})\right) = \Psi^{-1}(a^{21},a^{23})=1,
$$
since $[a^{21},a^{23}] = 0$ by the TC relation. 

For $\Psi(t^{(12)3},t^{34})$, we have
$$
    p_1p_3\alpha\left(\Psi(t^{(12)3},t^{34})\right) = \Psi(a^{23},a^{43}),
$$
since $a^{13},a^{31}, a^{32}$ and $a^{34}$ are all in the kernel of $p_1\circ p_3$.

For $\Psi(t^{(13)2},t^{24})^{-1}$, we first apply the symmetry property (2) from Lemma \ref{lem:GRT} to get 
$
    \Psi(t^{(13)2},t^{24})^{-1} = \Psi(t^{(13)2}, -t^{(13)2}-t^{(13)4})^{-1}.
$
Thus,
$$
    p_1p_3\alpha\left(\Psi^{-1}(t^{(13)2}, -t^{(13)2}-t^{(13)4})\right) = \Psi^{-1}(a^{2(13)},-a^{2(13)}-a^{4(13)}). 
$$

Substituting these results into equation~(\ref{eq:GW}), we obtain
\begin{equation}\label{eq:GWresult}
    G(W) = C^{24}\Psi^{-1}(a^{2(13)},-a^{2(13)}-a^{4(13)})\cdot\Psi(a^{23},a^{43}).
\end{equation}

To prove the formula \eqref{eq:GfromF}, it remains to show that
$$ \varphi(G(W))= \varphi\left(\Psi^{-1}\left(a^{2(12)},-a^{2(13)}-a^{4(13)}\right)\cdot \Psi(a^{23},a^{43})\right)C^{12}. $$
Note that in \eqref{eq:GWresult}, $G(W)$ is expressed in a form where there are only arrow tails on the capped strands 2 and 4 (and only arrow heads on the punctured strands 1 and 3). The isomorphism $\varphi$ acts by ``pushing'' all the arrow tails from strands 2 and 4 up across the $\pcv$ vertices. Due to the TC relation, it is possible to start with $C^{24}$, followed by the arrow tails from $\Psi^{-1}\left(a^{2(13)},-a^{2(13)}-a^{4(13)}\right)\cdot \Psi(a^{23},a^{43})$, then—with the capped stands cleared—the arrow heads from strands 1 and 3 are pushed up as well.
Thus, we have $\varphi\left(\Psi^{-1}\left(a^{2(12)},-a^{2(13)}-a^{4(13)}\right)\cdot \Psi(a^{23},a^{43})\right)$, as required. 
\end{proof}

\subsection{An embedding of automorphism groups}
Our goal is to prove the following theorem, the main result of this paper:
\begin{theorem}\label{thm:main} Construction~\ref{cons:buckle} defines a group homomorphism $\rp: \Aut(\PaCD) \to \Aut_v(\arrows)$, given by $\rp(F)=G$. This homomorphism is equivalent to the Alekseev-Torossian map $\rho$ in the sense that $\rp$ fits into the commutative diagram \eqref{eq:rp}.
\end{theorem}
  
The proof of the Theorem comes down to analysing the diagram \eqref{eq:BigDiagram} below, where black arrows are group homomorphisms and red arrows are anti-homomorphisms. We begin with the definitions of the maps.
\begin{equation}\label{eq:BigDiagram}
\begin{tikzcd}
	{\Aut(\PaCD)} &&&& {\Aut_v(\A)} \\
	&&& {\calG(\uparrow_2)^{tr}} \\
	&&& {\TAut_2} \\
	{\grt_1} &&&& {\krv}
	\arrow["\rho"', from=4-1, to=4-5]
	\arrow["\cong","\tau"', from=1-1, to=4-1]
	\arrow["\Theta","\cong"', from=1-5, to=4-5]
	\arrow[""{name=0, anchor=center, inner sep=0}, "\rp", dashed, from=1-1, to=1-5]
	\arrow[""{name=1, anchor=center, inner sep=0},"i", hook', from=4-5, to=3-4]
	\arrow["\gamma", red,from=2-4, to=3-4]
	\arrow[""{name=2, anchor=center, inner sep=0},"v",red, hook, from=1-5, to=2-4]
	\arrow["r",red, from=1-1, to=2-4]
	\arrow["{(1)}"',phantom, from=2-4, to=4-1]
	\arrow["{(2)}"{description}, phantom,shorten <=5pt, Rightarrow, from=0, to=2-4]
	\arrow["{(3)}"{description},phantom, shorten <=9pt, shorten >=9pt, Rightarrow, from=2, to=1]
\end{tikzcd}
\end{equation}


\noindent{\em The map r.} 
From $F\in \Aut(\PaCD)$, set $r(F)=N^{tr}=(G(\vertex))^{tr}$. The map $r$ is a group anti-homomorphism: see Lemma~\ref{lem:ranti}.

\smallskip
\noindent{\em The map $v$.}
Given an automorphism $G\in \Aut_v(\arrows)$, $v(G):=\apr(G(\vertex))\in \calG(\uparrow_2)^{tr}$. $v$ is a group anti-homomorphism since if $v(G_1) = N_1$ and $v(G_2) = N_2$, then—using that expansion-preserving automorphisms of $\arrows$ fix the arrow—we have $v(G_1 \circ G_2) = N_2 \cdot N_1$.  

\smallskip
\noindent{\em The map $\rp$.}
In Theorem~\ref{thm:GfromF} we proved that if there is a $G\in \Aut_v(\arrows)$ fitting into the diagram \eqref{eq:GfromF}, then $\apr(G(\vertex))=r(F)$. In turn, $\apr(G(\vertex))$ uniquely determines $G$. Thus, if there exists a $G\in \Aut_v(\arrows)$ with $\apr(G(\vertex))=r(F)$, we set $\rp(F)=G$. The map $\rp$ exists if and only if $r$ factors through $\Aut_v(\arrows)$, that is, if given any $F\in \Aut(\PaCD)$, the value $r(F)$ is always a valid vertex value for some automorphism $G\in \Aut_v(\arrows)$. Follows from the proof of Theorem~\ref{thm:main} at the end of this section; we also give a more direct proof in Theorem~\ref{rhop_in_aut}. 

\smallskip
\noindent{\em The map $\tau$.}
The map $\tau$ is the isomorphism described in Theorem~\ref{thm: GRT is Aut}. It is given by $\tau(F)=\Psi$, where $F(\Associator)=\Psi(t^{12},t^{23})$.

\smallskip
\noindent{\em The map $\rho$.} The map $\rho$ is the Alekseev-Torossian map defined at the beginning of this section, given by the formula $\rho(\Psi)=(\Psi(-x-y,x),\Psi(-x-y,y))$.

\smallskip
\noindent{\em The map $\gamma$.} There is an action\footnote{This action is described in detail in \cite{BND:WKO2}, see the proof of Proposition 3.19, in particular the ``Conceptual argument''.} of $\calG(\uparrow_2)$ on $\lie_2$, coming from the exponentiation of the projection $\calP(\uparrow_2) \to \tder_2 $ (Proposition~\ref{prop:SES}). There is a bijection between Lie words $w$ in $x$ and $y$, and tree arrow diagrams on three strands whose tails lie on the first two strands and head lies on the third strand, as shown in Figure~\ref{fig:liewordarrows}. Specifically, label the first two strands $x$ and $y$, and represent the word $w$ as a binary tree with tails on the $x$ and $y$ strands, vertices representing the brackets, and head on the third strand. Call this tree $\bT_w\in \arrows(\uparrow_3)$. Given a group-like element $\bD\in \calG(\uparrow_3)$, the conjugate $(\bD^{12})^{-1}\bT_w \bD^{12}$ is an (infinite) sum of trees with tails on the first two strands and heads on the third strand \cite[Proposition 3.19]{BND:WKO2}, and thus corresponds to an element of $\lie_2$ via the same bijection. This defines the action of $\bD$ on $w$. Since wheels in $\calG(\uparrow_2)$ act trivially by the TC relation, the action descends to an action of $\calG(\uparrow_2)^{tr}$ on $\lie_2$.

Furthermore, it turns out that $\calG(\uparrow_2)^{tr}$ acts on $\lie_2$ by {\em tangential} automorphisms. Namely, if $\bD\in \calG(\uparrow_2)^{tr}$, then modulo wheels $\bD=e^\delta$, where $\delta$ is a linear combination of trees. Furthermore, $\delta=\delta_1+\delta_2$, where $\delta_1$ is a sum of the terms with heads ending on strand 1, and $\delta_2$ is a sum of terms with heads on the second strand. Each term of $\delta_1$ corresponds to a Lie word in $x$ and $y$, similarly to Figure~\ref{fig:liewordarrows} but now with head on strand 1, thus $\delta_1$ corresponds to a sum $u_1\in \lie_2$. Similarly, $\delta_2$ corresponds to $l_2\in \lie_2$. Then $\bD$ acts as $(e^{u_1}, e^{u_2})\in \TAut_2$. Note that where the head of each tree in $\bD$ is attached to the strand 1 or 2 in relation to its tails does not make a difference in $\gamma(\bD)$.

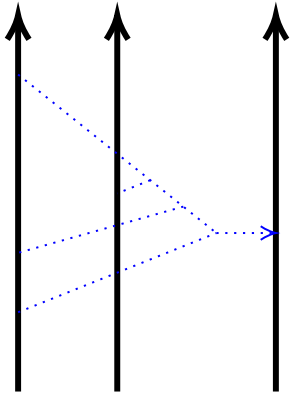
\begin{figure}[h]
    \centering
\tikzset{every picture/.style={line width=0.75pt}} 
\[\begin{tikzpicture}[x=0.75pt,y=0.75pt,yscale=-1,xscale=1]

\draw [line width=2.25]    (60,230) -- (60,44) ;
\draw [shift={(60,40)}, rotate = 90] [color={rgb, 255:red, 0; green, 0; blue, 0 }  ][line width=2.25]    (12.24,-5.49) .. controls (7.79,-2.58) and (3.71,-0.75) .. (0,0) .. controls (3.71,0.75) and (7.79,2.58) .. (12.24,5.49)   ;
\draw [line width=2.25]    (110,230) -- (110,44) ;
\draw [shift={(110,40)}, rotate = 90] [color={rgb, 255:red, 0; green, 0; blue, 0 }  ][line width=2.25]    (12.24,-5.49) .. controls (7.79,-2.58) and (3.71,-0.75) .. (0,0) .. controls (3.71,0.75) and (7.79,2.58) .. (12.24,5.49)   ;
\draw [line width=2.25]    (190,230) -- (190,44) ;
\draw [shift={(190,40)}, rotate = 90] [color={rgb, 255:red, 0; green, 0; blue, 0 }  ][line width=2.25]    (12.24,-5.49) .. controls (7.79,-2.58) and (3.71,-0.75) .. (0,0) .. controls (3.71,0.75) and (7.79,2.58) .. (12.24,5.49)   ;
\draw [color={rgb, 255:red, 0; green, 0; blue, 255 }  ,draw opacity=1 ] [dash pattern={on 0.84pt off 2.51pt}]  (160,150) -- (188,150) ;
\draw [shift={(190,150)}, rotate = 180] [color={rgb, 255:red, 0; green, 0; blue, 255 }  ,draw opacity=1 ][line width=0.75]    (7.65,-3.43) .. controls (4.86,-1.61) and (2.31,-0.47) .. (0,0) .. controls (2.31,0.47) and (4.86,1.61) .. (7.65,3.43)   ;
\draw [color={rgb, 255:red, 0; green, 0; blue, 255 }  ,draw opacity=1 ] [dash pattern={on 0.84pt off 2.51pt}]  (60,190) -- (160,150) ;
\draw [color={rgb, 255:red, 0; green, 0; blue, 255 }  ,draw opacity=1 ] [dash pattern={on 0.84pt off 2.51pt}]  (60,70) -- (76.67,83.33) -- (93.33,96.67) -- (110,110) -- (126.67,123.33) -- (143.33,136.67) -- (160,150) ;
\draw [color={rgb, 255:red, 0; green, 0; blue, 255 }  ,draw opacity=1 ] [dash pattern={on 0.84pt off 2.51pt}]  (126.67,123.33) -- (110,130) ;
\draw [color={rgb, 255:red, 0; green, 0; blue, 255 }  ,draw opacity=1 ] [dash pattern={on 0.84pt off 2.51pt}]  (143.33,136.67) -- (60,160) ;

\end{tikzpicture}\]
    \caption{The lie word $[x,[x,[y,x]]]$ represented as an arrow diagram on 3 strands.}
    \label{fig:liewordarrows}
\end{figure}

The map $\gamma$ is a group anti-homomorphism since if $\gamma(N_1) = a$ and $\gamma(N_2) = b$, then $\gamma(N_1 \cdot N_2) = b \circ a = \gamma(N_2) \circ \gamma(N_1)$ in $\TAut_2$.  

\smallskip
\noindent{\em The map $i$.} By definition, $\krv$ is a subgroup of $\TAut_2$, this is the inclusion $i$.

\smallskip
\noindent{\em The map $\Theta$.} The map $\Theta$ is the isomorphism of \cite[Theorem 5.12]{DHR21}, recalled as Theorem~\ref{thm:kv and krv}. Namely, for $G\in \Aut_v(\arrows)$, let $G(\vertex)=N\in \calG(\uparrow_2)$. Then, modulo wheels, we have $\apr(N)=N^{tr}\in \calG(\uparrow_2)^{tr}$. By \cite[Theorem 5.12]{DHR21}, $\gamma(N^{tr})$ then satisfies the $\krv$ equations and $\Theta(G):=\gamma(N^{tr})\in \krv$.

To prove Theorem~\ref{thm:main}, we need to show that the map $\rp$ exists, and that the diagram \eqref{eq:BigDiagram} commutes. A significant step is to describe the composite $\gamma\circ r$, and prove that $r$ is an anti-homomorphism.

\begin{lemma}\label{lem:Nxy} Given $F\in \Aut(\PaCD)$ with $F(\Associator)=\Psi(t^{12},t^{13})$, we have
\begin{align}
    \left(\gamma \circ r(F)\right)(x) &= \Psi(-x-y,x)^{-1}\; x\;\Psi(-x-y,x),\\
    \left(\gamma \circ r(F)\right)(y) &= \Psi(-x-y,y)^{-1}\; y \;\Psi(-x-y,y). 
\end{align}
\end{lemma}

\begin{proof}Recall that, by definition of $r$ and $\gamma$, 
\[
\left(\gamma\circ r(F)\right)(x)= (N^{tr})^{-1}\;x\;N^{tr}, \quad  \text{and} \quad \left(\gamma\circ r(F)\right)(y)= (N^{tr})^{-1}\;y\;N^{tr},\] 
where $N=G(\vertex)$ is the value obtained by the ``stripping'' process and $N^{tr}$ is the tree projection $\apr(N)$ of $N$. 
Specifically, by Theorem~\ref{thm:GfromF},
\begin{align}
    N^{tr} = \varphi\left(\Psi^{-1}\left(a^{2(12)},-a^{2(13)}-a^{4(13)}\right)\cdot \Psi(a^{23},a^{43})\right). 
\end{align}

The main step still missing is to compute the value of the isomorphism $\varphi$ of Lemma~\ref{lem:EK}. To do this, we introduce a diagrammatic {\em placement notation}, as shown in Figures \ref{fig:placementa} and \ref{fig:placementb}. In this notation, the two $\Psi$-factors are named $\zeta$ and $\chi$ for short, and their arrow endings on each strand are numbered by the strand numbers: $\zeta_1,\zeta_2,$ etc, for arrow endings of $\zeta$ on the first, second, etc strand. The value of $\varphi\left(\Psi^{-1}\left(a^{2(12)},-a^{2(13)}-a^{4(13)}\right)\cdot \Psi(a^{23},a^{43})\right)$ in $\arrows(\uparrow_2)$ is shown in Figure~\ref{fig:placement_notation2}.
\begin{figure}
    \centering
    \begin{subfigure}[c]{0.45\textwidth}
    \centering\captionsetup{width=.8\linewidth}
\[\begin{tikzpicture}[x=0.75pt,y=0.75pt,yscale=-1,xscale=1]

\draw [line width=2.25]    (90,44) -- (90,80) ;
\draw [shift={(90,40)}, rotate = 90] [color={rgb, 255:red, 0; green, 0; blue, 0 }  ][line width=2.25]    (12.24,-5.49) .. controls (7.79,-2.58) and (3.71,-0.75) .. (0,0) .. controls (3.71,0.75) and (7.79,2.58) .. (12.24,5.49)   ;
\draw [color={rgb, 255:red, 255; green, 0; blue, 30 }  ,draw opacity=1 ]   (90,80) .. controls (51,81.33) and (51,155.33) .. (50,260) ;
\draw [line width=2.25]    (90,80) .. controls (130,80.33) and (131,155.33) .. (130,260) ;
\draw [shift={(130,260)}, rotate = 90.55] [color={rgb, 255:red, 0; green, 0; blue, 0 }  ][fill={rgb, 255:red, 0; green, 0; blue, 0 }  ][line width=2.25]      (0, 0) circle [x radius= 5.36, y radius= 5.36]   ;
\draw [line width=2.25]    (199.79,45) -- (199.79,81) ;
\draw [shift={(199.79,41)}, rotate = 90] [color={rgb, 255:red, 0; green, 0; blue, 0 }  ][line width=2.25]    (12.24,-5.49) .. controls (7.79,-2.58) and (3.71,-0.75) .. (0,0) .. controls (3.71,0.75) and (7.79,2.58) .. (12.24,5.49)   ;
\draw [color={rgb, 255:red, 255; green, 0; blue, 30 }  ,draw opacity=1 ]   (199.79,81) .. controls (160.79,82.33) and (160.79,156.33) .. (159.79,261) ;
\draw [line width=2.25]    (199.79,81) .. controls (239.79,81.33) and (240.79,156.33) .. (239.79,261) ;
\draw [shift={(239.79,261)}, rotate = 90.55] [color={rgb, 255:red, 0; green, 0; blue, 0 }  ][fill={rgb, 255:red, 0; green, 0; blue, 0 }  ][line width=2.25]      (0, 0) circle [x radius= 5.36, y radius= 5.36]   ;
\draw  [fill={rgb, 255:red, 255; green, 255; blue, 255 }  ,fill opacity=1 ] (100,130) -- (260,130) -- (260,160) -- (100,160) -- cycle ;
\draw  [fill={rgb, 255:red, 255; green, 255; blue, 255 }  ,fill opacity=1 ] (30,190) -- (260,190) -- (260,220) -- (30,220) -- cycle ;

\draw (90,26.6) node [anchor=south] [inner sep=0.75pt]    {$1$};
\draw (199.79,27.6) node [anchor=south] [inner sep=0.75pt]    {$2$};
\draw (180,145) node    {$\eta ( a_{23} ,a_{43})$};
\draw (145,205) node    {$\eta ^{-1}( a_{2( 13)} ,-a_{2( 13)} -a_{4( 13)})$};
\draw (262,133.4) node [anchor=north west][inner sep=0.75pt]    {$=\chi $};
\draw (262,193.4) node [anchor=north west][inner sep=0.75pt]    {$=\zeta $};

\end{tikzpicture}\]
    \caption{$N^{tr}$ shown in the usual box notation, with names $\chi$ and $\zeta$ assigned to each arrow diagram component.}\label{fig:placementa}
    \end{subfigure}
        \begin{subfigure}[c]{0.45\textwidth}
    \centering\captionsetup{width=.8\linewidth}
    \[\begin{tikzpicture}[x=0.75pt,y=0.75pt,yscale=-1,xscale=1]

\draw [line width=2.25]    (110,64) -- (110,100) ;
\draw [shift={(110,60)}, rotate = 90] [color={rgb, 255:red, 0; green, 0; blue, 0 }  ][line width=2.25]    (12.24,-5.49) .. controls (7.79,-2.58) and (3.71,-0.75) .. (0,0) .. controls (3.71,0.75) and (7.79,2.58) .. (12.24,5.49)   ;
\draw [color={rgb, 255:red, 255; green, 0; blue, 30 }  ,draw opacity=1 ]   (110,100) .. controls (71,101.33) and (71,175.33) .. (70,280) ;
\draw [line width=2.25]    (110,100) .. controls (150,100.33) and (151,175.33) .. (150,280) ;
\draw [shift={(150,280)}, rotate = 90.55] [color={rgb, 255:red, 0; green, 0; blue, 0 }  ][fill={rgb, 255:red, 0; green, 0; blue, 0 }  ][line width=2.25]      (0, 0) circle [x radius= 5.36, y radius= 5.36]   ;
\draw [line width=2.25]    (219.79,65) -- (219.79,101) ;
\draw [shift={(219.79,61)}, rotate = 90] [color={rgb, 255:red, 0; green, 0; blue, 0 }  ][line width=2.25]    (12.24,-5.49) .. controls (7.79,-2.58) and (3.71,-0.75) .. (0,0) .. controls (3.71,0.75) and (7.79,2.58) .. (12.24,5.49)   ;
\draw [color={rgb, 255:red, 255; green, 0; blue, 30 }  ,draw opacity=1 ]   (219.79,101) .. controls (180.79,102.33) and (180.79,176.33) .. (179.79,281) ;
\draw [line width=2.25]    (219.79,101) .. controls (259.79,101.33) and (260.79,176.33) .. (259.79,281) ;
\draw [shift={(259.79,281)}, rotate = 90.55] [color={rgb, 255:red, 0; green, 0; blue, 0 }  ][fill={rgb, 255:red, 0; green, 0; blue, 0 }  ][line width=2.25]      (0, 0) circle [x radius= 5.36, y radius= 5.36]   ;
\draw  [color={rgb, 255:red, 0; green, 0; blue, 0 }  ,draw opacity=1 ][fill={rgb, 255:red, 255; green, 255; blue, 255 }  ,fill opacity=1 ] (130,154) -- (160,154) -- (160,180) -- (130,180) -- cycle ;
\draw  [color={rgb, 255:red, 0; green, 0; blue, 0 }  ,draw opacity=1 ][fill={rgb, 255:red, 255; green, 255; blue, 255 }  ,fill opacity=1 ] (170,154) -- (200,154) -- (200,180) -- (170,180) -- cycle ;
\draw  [color={rgb, 255:red, 0; green, 0; blue, 0 }  ,draw opacity=1 ][fill={rgb, 255:red, 255; green, 255; blue, 255 }  ,fill opacity=1 ] (240,154) -- (270,154) -- (270,180) -- (240,180) -- cycle ;
\draw  [color={rgb, 255:red, 0; green, 0; blue, 0 }  ,draw opacity=1 ][fill={rgb, 255:red, 255; green, 255; blue, 255 }  ,fill opacity=1 ] (130,220) -- (160,220) -- (160,246) -- (130,246) -- cycle ;
\draw  [color={rgb, 255:red, 0; green, 0; blue, 0 }  ,draw opacity=1 ][fill={rgb, 255:red, 255; green, 255; blue, 255 }  ,fill opacity=1 ] (240,220) -- (270,220) -- (270,246) -- (240,246) -- cycle ;
\draw  [color={rgb, 255:red, 0; green, 0; blue, 0 }  ,draw opacity=1 ][fill={rgb, 255:red, 255; green, 255; blue, 255 }  ,fill opacity=1 ] (170,220) -- (200,220) -- (200,246) -- (170,246) -- cycle ;
\draw  [color={rgb, 255:red, 0; green, 0; blue, 0 }  ,draw opacity=1 ][fill={rgb, 255:red, 255; green, 255; blue, 255 }  ,fill opacity=1 ] (60,220) -- (90,220) -- (90,246) -- (60,246) -- cycle ;

\draw (110,46.6) node [anchor=south] [inner sep=0.75pt]    {$1$};
\draw (219.79,47.6) node [anchor=south] [inner sep=0.75pt]    {$2$};
\draw (145,167) node    {$\chi _{1}$};
\draw (185,167) node    {$\chi _{2}$};
\draw (255,167) node    {$\chi _{3}$};
\draw (75,233) node    {$\zeta _{1}$};
\draw (145,233) node    {$\zeta _{2}$};
\draw (185,233) node    {$\zeta _{3}$};
\draw (255,233) node    {$\zeta _{4}$};

\end{tikzpicture}
\]
    \caption{$N^{tr}$ shown in placement notation, with arrow endings numbered according to the strand numbering.}\label{fig:placementb}
    \end{subfigure}
    
    \begin{subfigure}[c]{\textwidth}
    \centering      
\[\begin{tikzpicture}[x=0.75pt,y=0.75pt,yscale=-1,xscale=1]

\draw [line width=2.25]    (90,44) -- (90,80) ;
\draw [shift={(90,40)}, rotate = 90] [color={rgb, 255:red, 0; green, 0; blue, 0 }  ][line width=2.25]    (12.24,-5.49) .. controls (7.79,-2.58) and (3.71,-0.75) .. (0,0) .. controls (3.71,0.75) and (7.79,2.58) .. (12.24,5.49)   ;
\draw [color={rgb, 255:red, 255; green, 0; blue, 30 }  ,draw opacity=1 ]   (90,80) .. controls (51,81.33) and (51,155.33) .. (50,260) ;
\draw [line width=2.25]    (90,80) .. controls (130,80.33) and (131,155.33) .. (130,260) ;
\draw [shift={(130,260)}, rotate = 90.55] [color={rgb, 255:red, 0; green, 0; blue, 0 }  ][fill={rgb, 255:red, 0; green, 0; blue, 0 }  ][line width=2.25]      (0, 0) circle [x radius= 5.36, y radius= 5.36]   ;
\draw [line width=2.25]    (199.79,45) -- (199.79,81) ;
\draw [shift={(199.79,41)}, rotate = 90] [color={rgb, 255:red, 0; green, 0; blue, 0 }  ][line width=2.25]    (12.24,-5.49) .. controls (7.79,-2.58) and (3.71,-0.75) .. (0,0) .. controls (3.71,0.75) and (7.79,2.58) .. (12.24,5.49)   ;
\draw [color={rgb, 255:red, 255; green, 0; blue, 30 }  ,draw opacity=1 ]   (199.79,81) .. controls (160.79,82.33) and (160.79,156.33) .. (159.79,261) ;
\draw [line width=2.25]    (199.79,81) .. controls (239.79,81.33) and (240.79,156.33) .. (239.79,261) ;
\draw [shift={(239.79,261)}, rotate = 90.55] [color={rgb, 255:red, 0; green, 0; blue, 0 }  ][fill={rgb, 255:red, 0; green, 0; blue, 0 }  ][line width=2.25]      (0, 0) circle [x radius= 5.36, y radius= 5.36]   ;
\draw  [color={rgb, 255:red, 245; green, 166; blue, 35 }  ,draw opacity=1 ][fill={rgb, 255:red, 255; green, 255; blue, 255 }  ,fill opacity=1 ] (110,134) -- (140,134) -- (140,160) -- (110,160) -- cycle ;
\draw  [color={rgb, 255:red, 74; green, 144; blue, 226 }  ,draw opacity=1 ][fill={rgb, 255:red, 255; green, 255; blue, 255 }  ,fill opacity=1 ] (150,134) -- (180,134) -- (180,160) -- (150,160) -- cycle ;
\draw  [color={rgb, 255:red, 245; green, 166; blue, 35 }  ,draw opacity=1 ][fill={rgb, 255:red, 255; green, 255; blue, 255 }  ,fill opacity=1 ] (220,134) -- (250,134) -- (250,160) -- (220,160) -- cycle ;
\draw  [color={rgb, 255:red, 245; green, 166; blue, 35 }  ,draw opacity=1 ][fill={rgb, 255:red, 255; green, 255; blue, 255 }  ,fill opacity=1 ] (110,200) -- (140,200) -- (140,226) -- (110,226) -- cycle ;
\draw  [color={rgb, 255:red, 245; green, 166; blue, 35 }  ,draw opacity=1 ][fill={rgb, 255:red, 255; green, 255; blue, 255 }  ,fill opacity=1 ] (220,200) -- (250,200) -- (250,226) -- (220,226) -- cycle ;
\draw  [color={rgb, 255:red, 74; green, 144; blue, 226 }  ,draw opacity=1 ][fill={rgb, 255:red, 255; green, 255; blue, 255 }  ,fill opacity=1 ] (150,200) -- (180,200) -- (180,226) -- (150,226) -- cycle ;
\draw  [color={rgb, 255:red, 74; green, 144; blue, 226 }  ,draw opacity=1 ][fill={rgb, 255:red, 255; green, 255; blue, 255 }  ,fill opacity=1 ] (40,200) -- (70,200) -- (70,226) -- (40,226) -- cycle ;
\draw [line width=2.25]    (390,260) -- (390,54) ;
\draw [shift={(390,50)}, rotate = 90] [color={rgb, 255:red, 0; green, 0; blue, 0 }  ][line width=2.25]    (12.24,-5.49) .. controls (7.79,-2.58) and (3.71,-0.75) .. (0,0) .. controls (3.71,0.75) and (7.79,2.58) .. (12.24,5.49)   ;
\draw [line width=2.25]    (490,260) -- (490,54) ;
\draw [shift={(490,50)}, rotate = 90] [color={rgb, 255:red, 0; green, 0; blue, 0 }  ][line width=2.25]    (12.24,-5.49) .. controls (7.79,-2.58) and (3.71,-0.75) .. (0,0) .. controls (3.71,0.75) and (7.79,2.58) .. (12.24,5.49)   ;
\draw  [dash pattern={on 0.84pt off 2.51pt}]  (370,160) -- (510,160) ;
\draw  [color={rgb, 255:red, 245; green, 166; blue, 35 }  ,draw opacity=1 ][fill={rgb, 255:red, 255; green, 255; blue, 255 }  ,fill opacity=1 ] (375,84) -- (405,84) -- (405,110) -- (375,110) -- cycle ;
\draw  [color={rgb, 255:red, 245; green, 166; blue, 35 }  ,draw opacity=1 ][fill={rgb, 255:red, 255; green, 255; blue, 255 }  ,fill opacity=1 ] (475,84) -- (505,84) -- (505,110) -- (475,110) -- cycle ;
\draw  [color={rgb, 255:red, 245; green, 166; blue, 35 }  ,draw opacity=1 ][fill={rgb, 255:red, 255; green, 255; blue, 255 }  ,fill opacity=1 ] (375,120) -- (405,120) -- (405,146) -- (375,146) -- cycle ;
\draw  [color={rgb, 255:red, 245; green, 166; blue, 35 }  ,draw opacity=1 ][fill={rgb, 255:red, 255; green, 255; blue, 255 }  ,fill opacity=1 ] (475,120) -- (505,120) -- (505,146) -- (475,146) -- cycle ;
\draw  [color={rgb, 255:red, 74; green, 144; blue, 226 }  ,draw opacity=1 ][fill={rgb, 255:red, 255; green, 255; blue, 255 }  ,fill opacity=1 ] (375,174) -- (405,174) -- (405,200) -- (375,200) -- cycle ;
\draw  [color={rgb, 255:red, 74; green, 144; blue, 226 }  ,draw opacity=1 ][fill={rgb, 255:red, 255; green, 255; blue, 255 }  ,fill opacity=1 ] (475,174) -- (505,174) -- (505,200) -- (475,200) -- cycle ;
\draw  [color={rgb, 255:red, 74; green, 144; blue, 226 }  ,draw opacity=1 ][fill={rgb, 255:red, 255; green, 255; blue, 255 }  ,fill opacity=1 ] (475,210) -- (505,210) -- (505,236) -- (475,236) -- cycle ;
\draw    (270,150) -- (348,150) ;
\draw [shift={(350,150)}, rotate = 180] [color={rgb, 255:red, 0; green, 0; blue, 0 }  ][line width=0.75]    (10.93,-4.9) .. controls (6.95,-2.3) and (3.31,-0.67) .. (0,0) .. controls (3.31,0.67) and (6.95,2.3) .. (10.93,4.9)   ;

\draw (90,26.6) node [anchor=south] [inner sep=0.75pt]    {$1$};
\draw (199.79,27.6) node [anchor=south] [inner sep=0.75pt]    {$2$};
\draw (125,147) node    {$\chi _{1}$};
\draw (165,147) node    {$\chi _{2}$};
\draw (235,147) node    {$\chi _{3}$};
\draw (55,213) node    {$\zeta _{1}$};
\draw (125,213) node    {$\zeta _{2}$};
\draw (165,213) node    {$\zeta _{3}$};
\draw (235,213) node    {$\zeta _{4}$};
\draw (390,36.6) node [anchor=south] [inner sep=0.75pt]    {$1$};
\draw (490,36.6) node [anchor=south] [inner sep=0.75pt]    {$2$};
\draw (390,97) node    {$\chi _{1}$};
\draw (490,97) node    {$\chi _{3}$};
\draw (390,133) node    {$\zeta _{2}$};
\draw (490,133) node    {$\zeta _{4}$};
\draw (390,187) node    {$\zeta _{1}$};
\draw (490,187) node    {$\chi _{2}$};
\draw (490,223) node    {$\zeta _{3}$};
\draw (310,146.6) node [anchor=south] [inner sep=0.75pt]    {$\varphi $};
\draw (539,102.4) node [anchor=north west][inner sep=0.75pt]    {$tails$};
\draw (541,182.4) node [anchor=north west][inner sep=0.75pt]    {$heads$};

\end{tikzpicture}
\]
    \caption{Applying the isomorphism $\varphi$ using the placement notation. Arrow tails are shown in blue, and heads in orange.}\label{fig:placement_notation2}
    \end{subfigure}
    \caption{}
\end{figure}

Due to the TC relation on arrow diagrams, only the tree components of $N^{tr}$ which have heads on the first strand (after applying $\varphi$) act non-trivially on $x$. The only such component -- written in the placement notation -- is $\zeta_1$, coming from $$\zeta=\Psi^{-1}(a^{2(13)},-a^{2(13)}-a^{4(13)})=\Psi^{-1}(a^{21}+a^{23}, -a^{21}-a^{23}-a^{41}-a^{43}).$$ The arrows $\varphi(a^{23})$ and $\varphi(a^{43})$ have heads lying on strand 2: this is where $\zeta_3$ arises from in Figure~\ref{fig:placement_notation2}. Therefore, these arrows do not act on $x$, so $x$ is only conjugated by
$
 \varphi(\Psi^{-1}(a^{21},-a^{21}-a^{41})), 
$
 which, by the anti-symmetry relation of $\grt_1$ equals
$
 \varphi(\Psi(-a^{21}-a^{41},a^{21})). 
$

In summary 
\begin{align}
    \left(\gamma \circ r (F)\right)(x) &= \varphi(\Psi(-a^{21}-a^{41},a^{21}))^{-1}\; x\; \varphi(\Psi(-a^{21}-a^{41},a^{21})).
\end{align}

Now recall that the value of $\gamma$—that is, the conjugation action—only depends on the lie words represented by the tree components of $\varphi(\Psi(-a^{21}-a^{41},a^{21}))$, not the specific placement of arrow heads and tails.
Let $\Psi=e^\psi$ for some Lie element $\psi\in \mathfrak{grt}_1$, and let $w=w(-a^{21}-a^{41}, a^{21})$ be a Lie word component of $\psi$. Since $\varphi(a^{21})=a^{11}$ and $\varphi(a^{41})=a^{21}$, the Lie word represented by  $\varphi(w)$ is\footnote{Note that $\varphi(w(-a^{21}-a^{41},a^{21})) \neq w(-a^{11}-a^{21}, a^{11})$: at the level of diagrams the tail/head placements matter, and in fact a short arrow $a^{11}$ (with no other arrow endings between its tail and head) is central in $\arrows(\uparrow_2)$. However, the Lie word represented by the tree $\varphi(w(-a^{21}-a^{41},a^{21}))$ does have an $x$ for each $a^{21}$ and a $y$ for each $a^{41}$.} $w(-x-y,y)$. Exponentiating, we obtain
\begin{align}
   \left(\gamma \circ r (F)\right)(x) = \Psi(-x-y,x)^{-1} \; x \; \Psi(-x-y,x).
\end{align}

We proceed similarly to compute the action on $y$.
Due to the TC relation, only components of $N^{tr}$ with heads on the second strand act non-trivially on $y$. There are two such components: $\chi_3$ and $\zeta_3$ in placement notation, where 
$\chi_3$ comes from the term $\Psi(a^{23},a^{43})$, and $\zeta_3$ comes from $\Psi^{-1}(a^{2(13)},-a^{2(13)}-a^{4(13)})$. Cancelling arrows which act trivially on $y$ (i.e. those with heads on strand 1), we get that $y$ is conjugated by 
\begin{align}\label{eq:yaction}
    \varphi\left(\Psi^{-1}(a^{23},-a^{23}-a^{43}) \cdot \Psi(a^{23},a^{43})\right). 
\end{align}
To simplify, we apply the hexagon relation (\ref{eq:grthexagon}) of $GRT_1$ with $x = -a^{23}-a^{43}, y = a^{23}, z = a^{43}$ and obtain that $y$ is conjugated by 
\begin{align}
    \varphi(\Psi(-a^{23}-a^{43},a^{43})).
\end{align} 
And therefore we have 
\begin{align}
    (N^{12})^{-1}\; y\; N^{12} &= \varphi(\Psi(-a^{23}-a^{43},a^{43}))^{-1}\; y\; \varphi(\Psi(-a^{23}-a^{43},a^{43})).
\end{align}
Then, computing the value of $\gamma$ as before, we have 
\begin{align}
     \left(\gamma \circ r (F)\right)(y) = \Psi(-x-y,y)^{-1} \; y \; \Psi(-x-y,y),
\end{align}
as required.
\end{proof}

\begin{lemma}\label{lem:ShortArrows}
The map $\gamma$ is injective\footnote{Equivalently, $r(F)$ is free of short arrows for any $F\in \Aut(\PaCD)$. That is, $r(F)=e^\zeta$ and $\zeta$ is an (infinite) sum of tree arrow diagrams with no short arrow summands.} on the image of $r$.  
\end{lemma}
\begin{proof}
To see that $\gamma$ is injective on the image of $r$, we recall that the kernel of $\gamma$ is $\exp(\mathfrak{a_2})$, where $\mathfrak{a_2}$ is the abelian lie algebra spanned by single {\em short arrows} $a^{11}$ and $a^{22}$ \cite[Proposition 3.19]{BND:WKO2}. On the other hand, $F\in \Aut(\PaCD)$ is given by $F(\Associator)=\Psi=e^\psi\in\grt_1$, and $\psi$ vanishes in degree one \cite[Remark after Proposition 6.2]{Drin90}. Hence, $r(F)$ does not include short arrows. 
\end{proof}

The following Lemma, due to Drinfel'd, will be useful several times through the rest of the argument:
\begin{lemma}\label{lem:Drinfeld}\cite[Equations 5.14 and 5.19]{Drin90}
For $\Psi = e^{\psi} \in \grt_1$, if $X + Y + Z = 0$ then 
\begin{align}
    [X,\psi(Z,X)] + [Y,\psi(Z,Y)] = 0 \label{eq:drin_lie}
\end{align}
and 
\begin{align}
    X + \Psi(X,Y)^{-1} \; Y \; \Psi(X,Y) + \Psi(X,Z)^{-1} \; Z \; \Psi(X,Z) = 0.\label{eq:drin_grp}  
\end{align}
\end{lemma}

\begin{lemma}\label{lem:ranti}
The map $r: \Aut(\PaCD) \to \calG(\uparrow_2)^{tr}$ is a group anti-homomorphism. 
\end{lemma}
\begin{proof}
Since $\gamma$ is a group anti-homomorphism and injective on the image of $r$, we can prove that $r$ is an anti-homomorphism by showing that $\gamma \circ r$, as computed in Lemma~\ref{lem:Nxy}, is a group homomorphism.

First, recall that composition in $\Aut(\PaCD)$ does not translate to multiplication in $\PaCD$, but is given by the following formula \cite[Proposition 5.1]{BNGT}: if $F_1(\Associator)=\Psi_1(t^{12},t^{23})$ and $F_2(\Associator)=\Psi_2(t^{12}, t^{23})$, then 
\[
(F_1 \cdot F_2)(\Associator)= \Psi_1(t^{12},t^{23})\Psi_2\left(t^{12}\, , \, \Psi_1^{-1}(t^{12},t^{23})\, t^{23}\, \Psi_1(t^{12},t^{23})\right).
\]

By Lemma~\ref{lem:Nxy}$, \gamma \circ r(F_1)$ conjugates $x$ by 
$\Psi_1(-x-y,x)=:\Psi_{1,x}$ and $y$ by $\Psi_1(-x-y,y)=:\Psi_{1,y}$. Similarly, $\gamma \circ r(F_2)$ conjugates $x$ by $\Psi_2(-x-y,x)=:\Psi_{2,x}$ and $y$ by $\Psi_2(-x-y,y)=:\Psi_{2,y}$. 

When acting with $\left(\gamma \circ r(F_1)\right) \circ \left(\gamma \circ r (F_2)\right)$ on $x$, this is first a conjugation by $\Psi_2(-x-y,x)$, followed by the action of $\gamma \circ r(F_1)$, which acts on each letter of the result. Thus, the composite action on $x$ is conjugation by
\begin{equation}
     \Psi_{1,x} \cdot \Psi_2(-\Psi_{1,x}^{-1}\, x\, \Psi_{1,x}\, -\, \Psi_{1,y}^{-1}\, y\, \Psi_{1,y} \, , \,
    \Psi_{1,x}^{-1}\, x\, \Psi_{1,x}). 
\end{equation}
Applying Lemma~\ref{lem:Drinfeld} with the substitution $X=-x-y, Y=x, Z=y$, we obtain
$$-\Psi_{1,x}^{-1} x \Psi_{1,x} - \Psi_{1,y}^{-1} y \Psi_{1,y} =-x-y$$
Thus,  $\left(\gamma \circ r(F_1)\right) \circ \left(\gamma \circ r (F_2)\right)$ conjugates $x$ by 
\begin{align}
    \Psi_1(-x-y,x)\cdot \Psi_2(-x-y,\Psi_1^{-1}(-x-y,x)x\Psi_1(-x-y,x)).
\end{align}

On the other hand, by Lemma~\ref{lem:Nxy} we also have that $ \gamma \circ r(F_1 \cdot F_2)$ conjugates $x$ by
\begin{align}
    \Psi_1(-x-y,x)\cdot \Psi_2(-x-y,\Psi_1^{-1}(-x-y,x)x\Psi_1(-x-y,x)),
\end{align}
therefore 
$$\left(\gamma \circ r(F_1)\right) \circ \left(\gamma \circ r (F_2)\right)(x)=\left( \gamma \circ r(F_1 \cdot F_2)\right)(x).$$

After an analogous computation on $y$, we have that $\left( \gamma \circ r(F_1 \cdot F_2)\right) = \left(\gamma \circ r(F_1)\right) \cdot \left(\gamma \circ r (F_2)\right)$. Hence, $\gamma \circ r$ is a group homomorphism, and it follows from Lemma~\ref{lem:ShortArrows} that $r$ is a group anti-homomorphism as required. 
\end{proof}

Now, given that we know the value of $\gamma \circ r(F)$ and that $r$ is an anti-homomorphism, we can prove the following:

\begin{lemma}\label{rhoequiv}
The pentagon (1) in diagram~\ref{eq:BigDiagram} commutes:
\begin{equation}
\begin{tikzcd}
	{\Aut(\PaCD)} && {\calG(\uparrow_2)^{tr}} \\
	&& {\TAut_2} \\
	{\grt_1} && {\krv_2}
	\arrow["\rho"', from=3-1, to=3-3]
	\arrow["{\tau}", from=1-1, to=3-1]
	\arrow[hook', from=3-3, to=2-3]
	\arrow["\gamma", red,from=1-3, to=2-3]
	\arrow["r",red, from=1-1, to=1-3]
\end{tikzcd}
\end{equation}
\end{lemma}
\begin{proof}
Let $F$ be an element of $\Aut(\PaCD)$ such that $F(\Associator) = \Psi(t^{12},t^{23})$. Then $\tau(F) = \Psi(x,y)$, and $\rho(\tau(F)) = (\Psi(-x-y,x),\Psi(-x-y,y))$. 

In Lemma~\ref{lem:Nxy}, we showed that $\gamma(r(F))$ acts on $x$ and $y$ by conjugation with $(\Psi(-x-y,x)$, $\Psi(-x-y,y))$, which indeed agrees with $\rho(\Psi)$.  
\end{proof}

\begin{lemma}\label{lem:3commutes}
The square (3) in the diagram~\ref{eq:BigDiagram} commutes: 
\begin{equation}
\begin{tikzcd}
	{\calG(\uparrow_2)^{tr}} && {\Aut_v(\A)} \\
	\\
	{\TAut_2} && {\krv_2}
	\arrow["\gamma", red,from=1-1, to=3-1]
	\arrow["v", red, hook', swap, from=1-3, to=1-1]
	\arrow["\Theta", from=1-3, to=3-3]
	\arrow[hook', from=3-3, to=3-1]
\end{tikzcd}
\end{equation}
\end{lemma}
\begin{proof}
For $G\in \Aut_v(\arrows)$ given by $G(\vertex)=N$,  $\Theta(G)$ is the conjugation action by $N^{tr}$ on $\lie_2$, which is a priori an element of $\TAut_2$, but proven in \cite[Theorem 5.12]{DHR21} to lie in $\krv$. This agrees with $\gamma(v(G))$ by definition.
\end{proof}

This implies the Main Theorem:

\begin{proof}[Proof of Theorem~\ref{thm:main}]
Given Lemmas~\ref{rhoequiv}~and~\ref{lem:3commutes}, the only statement we need to show is that $r$ factors through $\Aut_v(\arrows)$. Indeed, by the commutativity of the pentagon (1) and the square (3), $r=v\circ(\Theta)^{-1}\circ\rho\circ\tau$. Thus, for any $F\in \Aut(\PaCD)$, the value $r(F)$ is the tree projection of a valid vertex value\footnote{Readers of \cite{DHR21} may note that $N^{tr}$ in this paper does not look the same as a vertex value of an automorphism in the image of $\Theta^{-1}$. However, by Lemma~\ref{lem:Ntr_elleta}, $N^{tr}$ is equal to a vertex value of the correct form.} for an expansion preserving, v-small automorphism of arrow diagrams, namely $G=\Theta^{-1}\rho\tau(F)$.
\end{proof}

In Appendix~\ref{sec:DirectProof} we present an alternative, more direct proof that $r$ factors through $\Aut_v(\arrows)$, which relies less on the Alekseev-Torossian map, and in Remark~\ref{rmk:direct} we explain an avenue for making the proof entirely self-contained.

\begin{cor}\label{rp_grouphomo}
$\rp: \operatorname{Aut}(\PaCD) \to \operatorname{Aut}_v(\A)$ is an injective group homomorphism.
\end{cor}
\begin{proof}
It's enough to show that both $v$ and $r$ are injective. 

The map $v$ is injective since by Corollary~\ref{auto_wheels} an automorphism is uniquely determined by its vertex value modulo wheels. 

The injectivity of $r$ follows from the in injectivity of $\gamma \circ r$. If we have $\gamma \circ r(F_1) = \gamma \circ r(F_2)$ with $F_1(\Associator) = \Psi_1(t^{12},t^{23})$ and $F_2(\Associator) = \Psi_2(t^{12},t^{23})$ then we must have, 
\begin{align}\label{eq:psi1psi2}
    (\Psi_1(-x-y,x),\Psi_1(-x-y,y)) = (\Psi_2(-x-y,x),\Psi_2(-x-y,y)).
\end{align} 
By Lemma~\ref{rhoequiv}, $\gamma\circ r$ is in the image of $\rho$. Then, by \cite[Theorem 4.6]{AT12}, $\rho$ is injective, and therefore Equation~\ref{eq:psi1psi2} is only possible if $\Psi_1 = \Psi_2$ which gives $F_1 = F_2$. By Lemma~\ref{lem:ShortArrows} it follows that $r$ is injective, and therefore so is $\rp$.
\end{proof}

Corollary~\ref{rp_grouphomo} states that there exists a subgroup of arrow diagram automorphisms isomorphic to $\grt_1$. As each homomorphic expansion of $w$-foams $Z:\widehat{\mathbb{Q}[\wf]}\rightarrow \arrows$ is an isomorphism of (completed) wheeled props, each one induces an isomorphism \[\Aut_v(\widehat{\mathbb{Q}[\wf]})\overset{Z}{\cong}\Aut_v(\arrows).\] 
\begin{cor}
There exists a subgroup of $\Aut_v(\widehat{\mathbb{Q}[\wf]})$ which is (non-canonically) isomorphic to $\gt_1$.
\end{cor}

Ideally, one would construct a topological map $\tilde{\rho}:\Aut(\widehat{\mathbb{Q}[\PaB]})\hookrightarrow \Aut_{v}(\widehat{\mathbb{Q}[\wf]})$ which agrees with the embedding of $\gt_1\hookrightarrow\kv$ from \cite{ATE10}. This construction is more subtle due to the completions, and is the topic of ongoing work of the second author.

\appendix

\section{Alternative proof of the Main Theorem}\label{sec:DirectProof}

 To prove directly that $r$ factors through $\Aut_v(\arrows)$, we need to show that given $F \in \Aut(\PaCD)$, the value $N^{tr}=r(F)$ lifts to a pair $(N,C)\in \arrows(\uparrow_2)\times \arrows(\bcap)$ satisfying the equations $R4', C'$ and $U'$ of Proposition~\ref{prop:SimplifiedEqns}. The value $N=G(\vertex)$ is computed in Theorem~\ref{thm:GfromF}, and given in terms of $N^{tr}$ and a yet-to-be-determined cap value $C=G(\bcap)$ in equation \eqref{eq:GfromF}, which we recall here for convenience:
 $$N=(C^1)^{-1}(C^2)^{-1}N^{tr}C^{12}.$$

\begin{lemma}\label{lem:rhopr4}
The value $N$ of Theorem~\ref{thm:GfromF} satisfies the R4' equation: 
\begin{align}\tag{R4'}
    N^{12}R^{(12)3}(N^{12})^{-1} &= R^{(12)3}  \quad \in \A(\uparrow_3)
\end{align}
where $R^{(12)3} = e^{a^{13}+a^{23}}$. 
\end{lemma}
\begin{proof}
Note that for the conjugation action of $N$ on $R^{(12)3}$, only the tree-projection of $N$ matters. That is, $N$ satisfies $R4'$ if and only if $N^{tr}$ does. Note that $(N^{tr})^{-1} e^{a^{13}+a^{23}}N^{tr}$ maps to $(N^{tr})\cdot e^{x+y}\in \lie_2$ via the map $\gamma$, where $\cdot$ denotes the conjugation action.
Thus, the proof that $N$ satisfies the R4' equation reduces to the following computation, where we denote $\Psi_x = \Psi(-x-y,x)$ and $\Psi_y = \Psi(-x-y,y)$, and explain the steps below.
\begin{align*}
    (N^{12})^{-1}\exp(a^{13} + a^{23}) N^{12}
    &= \exp\left((N^{12})^{-1}a^{13}N^{12} + (N^{12})^{-1}a^{23}N^{12}\right) \\
    &= \exp\left(\Psi_x^{-1} \; x \; \Psi_x + \Psi_y^{-1}\; y\; \Psi_y\right)\\
    &= \exp\left(x + y\right)\\
    &= e^{a^{13} + a^{23}}
\end{align*}
Here the first equality is by the definition of the exponential, the second is by Lemma~\ref{lem:Nxy}, and the third is by Lemma~\ref{lem:Drinfeld} with the substitution $X=-x-y, Y=x, Z=y$, and the fourth is by definition of the map $\gamma$. 
\end{proof}

\begin{lemma}\label{lem:rhopcap}
For any choice of $C\in \arrows(\bcap)$ the value $N$ of \eqref{eq:GfromF} satisfies the cap equation, 
\begin{equation}\tag{C'}
     C^{12}(N^{12})^{-1} = C^1C^2 \quad \text{in} \quad \A(\bcap_2).
\end{equation}
\end{lemma}
\begin{proof}
Substituting the value of $N$ \eqref{eq:GfromF} into the Cap equation \eqref{C'} we obtain:
\begin{align}
    C^{12}(C^{12})^{-1}(N^{tr})^{-1}C^2C^1 = C^1C^2.
\end{align}
We cancel $C^{12}(C^{12})^{-1}$ on the left. Since $\arrows(\bcap_2)$ is a right $\arrows(\uparrow_2)$-module by stacking, we multiply on the right by $(C^1)^{-1}(C^2)^{-1}$. Note that $C^1$ and $C^2$ commute. 

Thus, we have reduced the equation to 
$$N^{tr} = 1 \quad \text{in} \quad \A(\bcap_2).$$ 
Then, observe that $N^{tr}$ is in the image of $\varphi$, and as such, all arrow heads in every summand in $N^{tr}$ are below all arrow tails. Thus, all positive degree components of $N^{tr}$ vanish in $\arrows(\bcap)$ by the CP relation, and this completes the proof.
\end{proof}

In the proof of the unitarity equation U', the following three Lemmas are useful. 
\begin{lemma}\label{lem:WNtr}
$N=W \cdot N^{tr}$, where $W = C^{12}(C^1)^{-1}(C^2)^{-1}$. 
\end{lemma}
\begin{proof}
By repeated applications of the STU relation, the $C^{12}$ factor at the right-hand side of $N$ can be commuted to the left-hand side, at the possible "cost" of additional wheels.  Thus, $N$ can be written in the form $W \cdot N^{tr}$ for some product of wheels $W$. 
To calculate $W$, we apply that the pair $(N,C)$ satisfies the cap equation by Lemma~\ref{lem:rhopcap} and hence in $\A(\bcap_2)$ we can deduce (explained below):
\begin{align*}
    C^{12}N^{-1} &= C^1C^2\\
    C^{12}(N^{tr})^{-1}W^{-1} &= C^1C^2\\
    C^{12}(N^{tr})^{-1} &= C^1C^2W\\
    N^{tr}C^{12}(N^{tr})^{-1} &= C^1C^2W. 
\end{align*}
Here the first step is substituting $N=WN^{tr}$; the second step is right multiplication by $W$ using the right $\arrows(\uparrow_2)$-module structure, and the last step inserts $N^{tr}$ at the bottom caps, using that $N^{tr}=1\in\arrows(\bcap_2)$ by the CP relation.

Now recall that $C=e^{c(x)}$, where $c\in \Q[[\xi]]/\langle\xi\rangle$ (Lemma~\ref{cyclic_lemma}), and thus $C^{12}=e^{c(x+y)}$. By the \eqref{R4'} equation,
$N^{tr}(x+y)(N^{tr})^{-1}=x+y$. It is tempting to use this to deduce that $N^{tr}C^{12}(N^{tr})^{-1}=C^{12}$, but unfortunately in this case $e^{c(x+y)}$ denotes a power series in wheels, rather than a tree element corresponding to $\lie_2$ as in the \eqref{R4'} equation.

To overcome this issue, we use a linear trace map $\sigma: \A(\uparrow_3) \to \A(\bcap_2)$ defined in \cite[Section 4.4]{BND:WKO2} and illustrated in Figure~\ref{fig:sigmamap}.
The trace $\sigma$ ``closes up'' the third strand and converts it from a skeletal to an internal strand, then caps the first two strands at the bottom, as shown.  This map is well-defined (it kills almost all relations, and converts one STU relation into an IHX relation). If $f(x,y)$ is a word in $\hat{U}(\mathfrak{lie}_2)$ represented by an arrow diagram where $x=a^{13} $ and $y=a^{23}$, then $\sigma(f(x,y))$ is the wheel arrow diagram representing the cyclic word $\operatorname{tr}(f(x,y))$. In particular, for a power series $c\in \Q[[\xi]]/\langle\xi\rangle$, we have that $\sigma(c(x+y))=c(x+y)$.

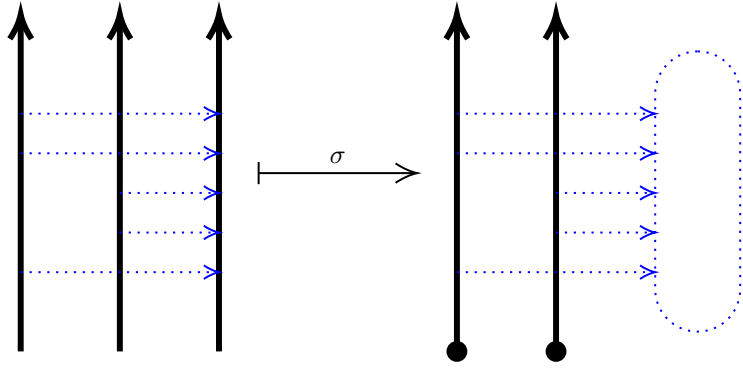
\begin{figure}[h]
    \centering
\tikzset{every picture/.style={line width=0.75pt}} 
\[\begin{tikzpicture}[x=0.75pt,y=0.75pt,yscale=-1,xscale=1]

\draw [line width=2.25]    (60,220) -- (60,54) ;
\draw [shift={(60,50)}, rotate = 90] [color={rgb, 255:red, 0; green, 0; blue, 0 }  ][line width=2.25]    (12.24,-5.49) .. controls (7.79,-2.58) and (3.71,-0.75) .. (0,0) .. controls (3.71,0.75) and (7.79,2.58) .. (12.24,5.49)   ;
\draw [line width=2.25]    (110,220) -- (110,54) ;
\draw [shift={(110,50)}, rotate = 90] [color={rgb, 255:red, 0; green, 0; blue, 0 }  ][line width=2.25]    (12.24,-5.49) .. controls (7.79,-2.58) and (3.71,-0.75) .. (0,0) .. controls (3.71,0.75) and (7.79,2.58) .. (12.24,5.49)   ;
\draw [line width=2.25]    (160,220) -- (160,54) ;
\draw [shift={(160,50)}, rotate = 90] [color={rgb, 255:red, 0; green, 0; blue, 0 }  ][line width=2.25]    (12.24,-5.49) .. controls (7.79,-2.58) and (3.71,-0.75) .. (0,0) .. controls (3.71,0.75) and (7.79,2.58) .. (12.24,5.49)   ;
\draw [line width=2.25]    (280,220) -- (280,54) ;
\draw [shift={(280,50)}, rotate = 90] [color={rgb, 255:red, 0; green, 0; blue, 0 }  ][line width=2.25]    (12.24,-5.49) .. controls (7.79,-2.58) and (3.71,-0.75) .. (0,0) .. controls (3.71,0.75) and (7.79,2.58) .. (12.24,5.49)   ;
\draw [shift={(280,220)}, rotate = 270] [color={rgb, 255:red, 0; green, 0; blue, 0 }  ][fill={rgb, 255:red, 0; green, 0; blue, 0 }  ][line width=2.25]      (0, 0) circle [x radius= 3.75, y radius= 3.75]   ;
\draw [line width=2.25]    (330,220) -- (330,54) ;
\draw [shift={(330,50)}, rotate = 90] [color={rgb, 255:red, 0; green, 0; blue, 0 }  ][line width=2.25]    (12.24,-5.49) .. controls (7.79,-2.58) and (3.71,-0.75) .. (0,0) .. controls (3.71,0.75) and (7.79,2.58) .. (12.24,5.49)   ;
\draw [shift={(330,220)}, rotate = 270] [color={rgb, 255:red, 0; green, 0; blue, 0 }  ][fill={rgb, 255:red, 0; green, 0; blue, 0 }  ][line width=2.25]      (0, 0) circle [x radius= 3.75, y radius= 3.75]   ;
\draw [color={rgb, 255:red, 0; green, 0; blue, 255 }  ,draw opacity=1 ] [dash pattern={on 0.84pt off 2.51pt}]  (60,180) -- (158,180) ;
\draw [shift={(160,180)}, rotate = 180] [color={rgb, 255:red, 0; green, 0; blue, 255 }  ,draw opacity=1 ][line width=0.75]    (7.65,-3.43) .. controls (4.86,-1.61) and (2.31,-0.47) .. (0,0) .. controls (2.31,0.47) and (4.86,1.61) .. (7.65,3.43)   ;
\draw    (180,130) -- (258,130) ;
\draw [shift={(260,130)}, rotate = 180] [color={rgb, 255:red, 0; green, 0; blue, 0 }  ][line width=0.75]    (10.93,-4.9) .. controls (6.95,-2.3) and (3.31,-0.67) .. (0,0) .. controls (3.31,0.67) and (6.95,2.3) .. (10.93,4.9)   ;
\draw [shift={(180,130)}, rotate = 180] [color={rgb, 255:red, 0; green, 0; blue, 0 }  ][line width=0.75]    (0,5.59) -- (0,-5.59)   ;
\draw  [color={rgb, 255:red, 0; green, 0; blue, 255 }  ,draw opacity=1 ][dash pattern={on 0.84pt off 2.51pt}] (401.43,68.57) .. controls (413.26,68.57) and (422.86,78.17) .. (422.86,90) -- (422.86,188.57) .. controls (422.86,200.41) and (413.26,210) .. (401.43,210) -- (401.43,210) .. controls (389.59,210) and (380,200.41) .. (380,188.57) -- (380,90) .. controls (380,78.17) and (389.59,68.57) .. (401.43,68.57) -- cycle ;
\draw [color={rgb, 255:red, 0; green, 0; blue, 255 }  ,draw opacity=1 ] [dash pattern={on 0.84pt off 2.51pt}]  (110,160) -- (158,160) ;
\draw [shift={(160,160)}, rotate = 180] [color={rgb, 255:red, 0; green, 0; blue, 255 }  ,draw opacity=1 ][line width=0.75]    (7.65,-3.43) .. controls (4.86,-1.61) and (2.31,-0.47) .. (0,0) .. controls (2.31,0.47) and (4.86,1.61) .. (7.65,3.43)   ;
\draw [color={rgb, 255:red, 0; green, 0; blue, 255 }  ,draw opacity=1 ] [dash pattern={on 0.84pt off 2.51pt}]  (110,140) -- (158,140) ;
\draw [shift={(160,140)}, rotate = 180] [color={rgb, 255:red, 0; green, 0; blue, 255 }  ,draw opacity=1 ][line width=0.75]    (7.65,-3.43) .. controls (4.86,-1.61) and (2.31,-0.47) .. (0,0) .. controls (2.31,0.47) and (4.86,1.61) .. (7.65,3.43)   ;
\draw [color={rgb, 255:red, 0; green, 0; blue, 255 }  ,draw opacity=1 ] [dash pattern={on 0.84pt off 2.51pt}]  (60,120) -- (84,120) -- (158,120) ;
\draw [shift={(160,120)}, rotate = 180] [color={rgb, 255:red, 0; green, 0; blue, 255 }  ,draw opacity=1 ][line width=0.75]    (7.65,-3.43) .. controls (4.86,-1.61) and (2.31,-0.47) .. (0,0) .. controls (2.31,0.47) and (4.86,1.61) .. (7.65,3.43)   ;
\draw [color={rgb, 255:red, 0; green, 0; blue, 255 }  ,draw opacity=1 ] [dash pattern={on 0.84pt off 2.51pt}]  (60,100) -- (158,100) ;
\draw [shift={(160,100)}, rotate = 180] [color={rgb, 255:red, 0; green, 0; blue, 255 }  ,draw opacity=1 ][line width=0.75]    (7.65,-3.43) .. controls (4.86,-1.61) and (2.31,-0.47) .. (0,0) .. controls (2.31,0.47) and (4.86,1.61) .. (7.65,3.43)   ;
\draw [color={rgb, 255:red, 0; green, 0; blue, 255 }  ,draw opacity=1 ] [dash pattern={on 0.84pt off 2.51pt}]  (280,100) -- (378,100) ;
\draw [shift={(380,100)}, rotate = 180] [color={rgb, 255:red, 0; green, 0; blue, 255 }  ,draw opacity=1 ][line width=0.75]    (7.65,-3.43) .. controls (4.86,-1.61) and (2.31,-0.47) .. (0,0) .. controls (2.31,0.47) and (4.86,1.61) .. (7.65,3.43)   ;
\draw [color={rgb, 255:red, 0; green, 0; blue, 255 }  ,draw opacity=1 ] [dash pattern={on 0.84pt off 2.51pt}]  (280,120) -- (378,120) ;
\draw [shift={(380,120)}, rotate = 180] [color={rgb, 255:red, 0; green, 0; blue, 255 }  ,draw opacity=1 ][line width=0.75]    (7.65,-3.43) .. controls (4.86,-1.61) and (2.31,-0.47) .. (0,0) .. controls (2.31,0.47) and (4.86,1.61) .. (7.65,3.43)   ;
\draw [color={rgb, 255:red, 0; green, 0; blue, 255 }  ,draw opacity=1 ] [dash pattern={on 0.84pt off 2.51pt}]  (280,180) -- (378,180) ;
\draw [shift={(380,180)}, rotate = 180] [color={rgb, 255:red, 0; green, 0; blue, 255 }  ,draw opacity=1 ][line width=0.75]    (7.65,-3.43) .. controls (4.86,-1.61) and (2.31,-0.47) .. (0,0) .. controls (2.31,0.47) and (4.86,1.61) .. (7.65,3.43)   ;
\draw [color={rgb, 255:red, 0; green, 0; blue, 255 }  ,draw opacity=1 ] [dash pattern={on 0.84pt off 2.51pt}]  (330,140) -- (378,140) ;
\draw [shift={(380,140)}, rotate = 180] [color={rgb, 255:red, 0; green, 0; blue, 255 }  ,draw opacity=1 ][line width=0.75]    (7.65,-3.43) .. controls (4.86,-1.61) and (2.31,-0.47) .. (0,0) .. controls (2.31,0.47) and (4.86,1.61) .. (7.65,3.43)   ;
\draw [color={rgb, 255:red, 0; green, 0; blue, 255 }  ,draw opacity=1 ] [dash pattern={on 0.84pt off 2.51pt}]  (330,160) -- (378,160) ;
\draw [shift={(380,160)}, rotate = 180] [color={rgb, 255:red, 0; green, 0; blue, 255 }  ,draw opacity=1 ][line width=0.75]    (7.65,-3.43) .. controls (4.86,-1.61) and (2.31,-0.47) .. (0,0) .. controls (2.31,0.47) and (4.86,1.61) .. (7.65,3.43)   ;

\draw (220,126.6) node [anchor=south] [inner sep=0.75pt]    {$\sigma $};

\end{tikzpicture}\]
    \caption{The map $\sigma$ applied to the of $\A(\uparrow_3)$ representing $xy^2x^2$.}
    \label{fig:sigmamap}
\end{figure}

Thus, $\sigma(R4')$ states that $N^{12}C^{12}(N^{12})^{-1}=C^{12}$, and thus we have $C^{12}=C^{1}C^{2}W$.
Since all factors on both sides of this equation are now wheel arrow diagrams, they all commute, and thus we can use the right $\arrows(\uparrow_2)$-module structure to obtain
\begin{align}
    W = C^{12}(C^{1})^{-1}(C^{2})^{-1}=e^{c^{12}-c^{1}-c^2}. 
\end{align}
We have identified the value of $W$ in $\arrows(\bcap_2)$, rather than in $\arrows(\uparrow_2)$. Although $\arrows(\bcap_2)$ is a factor of $\arrows(\uparrow_2)$, wheels embed in $\arrows(\bcap_2)$, and therefore if $ W = C^{12}(C^{1})^{-1}(C^{2})^{-1}$ in $\arrows(\bcap_2)$, then the same is true in $\arrows(\uparrow_2)$. 
This completes the proof.
\end{proof}

\begin{lemma}\label{lem:eta_self}
Let $\psi \in \mathfrak{grt}_1$ and $\eta = (\psi(-x-y,x),\psi(-x-y,y)) \in \mathfrak{tder}_2$. Then, the action of $\eta$ (as a tangential derivation) on itself (as an element of $\lie_2^{\oplus 2}$) is trivial.
\end{lemma}
\begin{proof}
Let $\psi_x = \psi(-x-y,x)$ and $\psi_y = \psi(-x-y,y)$. The action of $\eta$ on the Lie series $\psi_x$ and $\psi_y$ is -- by definition -- given by:
\begin{equation}\label{eq:etaeta}
\eta(\psi_x)=\psi([\psi_x,-x]+[\psi_y,-y],[\psi_x,x]) \quad \text{and} \quad \eta(\psi_y)=\psi([\psi_x,-x]+[\psi_y,-y],[\psi_y,y])
\end{equation}
Now applying Equation~\ref{eq:drin_lie} of Lemma~\ref{lem:Drinfeld} we have, 
$$[\psi(-x-y,x),x] + [\psi(-x-y,y),y] = 0,$$ 
 and thus the first component of both $\eta(\psi_x)$ and $\eta(\psi_y)$ given in \eqref{eq:etaeta} is zero.
Thus, $\eta(\psi_x)=\eta(\psi_y)=0$, as stated.
\end{proof}

\begin{lemma}\label{lem:Ntr_elleta}Given $F\in \Aut(\PaCD)$ with $F(\Associator)=\Psi(t^{12},t^{13})=e^{\psi(t^{12},t^{23})}$, we have
    \[G(\vertex)^{tr}=N^{tr} = e^{\ell(\eta)} \in \A(\uparrow_2)\]
where $\eta = (\psi(-x-y,x),\psi(-x-y,y)) \in \tder_2$, and $\ell: \tder_n \to \calP(\uparrow_n)$ is the map described in Remark~\ref{remark:l}. 
\end{lemma}
\begin{proof}
Even though $N^{tr}$ and $e^{\ell(\eta)}$ look different as arrow diagrams, $\gamma(N^{tr}) = \gamma(e^{\ell(\eta)})$ (Lemma~\ref{lem:Nxy}). Therefore, the difference $N^{tr}-e^{\ell(\eta)}$ is in the kernel of $\gamma$, and is hence given by some number of short arrows. The point of this proof is to show that these short arrows cancel.

The distinction between $N^{tr}$ and $e^{\ell(\eta)}$ is that each term of $N^{tr}$ is a combination of several trees with {\em all} heads attached below {\em all} tails, as on the left side of Figure~\ref{fig:treeseparate}. On the other hand, terms of $e^{\ell(\eta)}$ are products of trees, where the head of each individual tree is attached below its tails, but tree factors are separated, as the first term on the right-hand side of Figure~\ref{fig:treeseparate}. Converting from the $N^{tr}$ form to the $e^{\ell(\eta)}$ form requires separating the trees using STU relations. This results in extra tree summands, as shown in Figure~\ref{fig:treeseparate}. In fact, performing each STU relation amounts to a higher tree acting on a leaf of a lower one by the adjoint action (equivalently, as a tangential derivation). 

Consider the sum of terms of $N^{tr}$ with $2$ trees. Separating the upper tree from the lower in each term is equivalent to $\eta$ acting on itself. 
By Lemma~\ref{lem:eta_self}, the action of $\eta$ on itself is trivial. Therefore, the sum of the resulting additional tree summands vanish. 

Iterating this process for the sums of terms with $n$ trees thus completes the proof that $N^{tr} = e^{\ell(\eta)}$. 
\end{proof}

\begin{figure}[h]
    \centering
\tikzset{every picture/.style={line width=0.75pt}} 
\[\begin{tikzpicture}[x=0.75pt,y=0.75pt,yscale=-1,xscale=1]

\draw [line width=2.25]    (70,230) -- (70,64) ;
\draw [shift={(70,60)}, rotate = 90] [color={rgb, 255:red, 0; green, 0; blue, 0 }  ][line width=2.25]    (12.24,-5.49) .. controls (7.79,-2.58) and (3.71,-0.75) .. (0,0) .. controls (3.71,0.75) and (7.79,2.58) .. (12.24,5.49)   ;
\draw [line width=2.25]    (130,230) -- (130,64) ;
\draw [shift={(130,60)}, rotate = 90] [color={rgb, 255:red, 0; green, 0; blue, 0 }  ][line width=2.25]    (12.24,-5.49) .. controls (7.79,-2.58) and (3.71,-0.75) .. (0,0) .. controls (3.71,0.75) and (7.79,2.58) .. (12.24,5.49)   ;
\draw [color={rgb, 255:red, 0; green, 0; blue, 255 }  ,draw opacity=1 ] [dash pattern={on 0.84pt off 2.51pt}]  (70,110) -- (100,140) ;
\draw [color={rgb, 255:red, 0; green, 0; blue, 255 }  ,draw opacity=1 ] [dash pattern={on 0.84pt off 2.51pt}]  (100,140) -- (130,110) ;
\draw [color={rgb, 255:red, 0; green, 0; blue, 255 }  ,draw opacity=1 ] [dash pattern={on 0.84pt off 2.51pt}]  (100,140) .. controls (97.05,155.76) and (107.67,170.55) .. (71.68,179.59) ;
\draw [shift={(70,180)}, rotate = 346.68] [color={rgb, 255:red, 0; green, 0; blue, 255 }  ,draw opacity=1 ][line width=0.75]    (7.65,-3.43) .. controls (4.86,-1.61) and (2.31,-0.47) .. (0,0) .. controls (2.31,0.47) and (4.86,1.61) .. (7.65,3.43)   ;
\draw [color={rgb, 255:red, 0; green, 0; blue, 255 }  ,draw opacity=1 ] [dash pattern={on 0.84pt off 2.51pt}]  (70,145) -- (110,170) ;
\draw [color={rgb, 255:red, 0; green, 0; blue, 255 }  ,draw opacity=1 ] [dash pattern={on 0.84pt off 2.51pt}]  (110,170) -- (130,145) ;
\draw [color={rgb, 255:red, 0; green, 0; blue, 255 }  ,draw opacity=1 ] [dash pattern={on 0.84pt off 2.51pt}]  (110,170) .. controls (107.05,185.76) and (107.97,205.4) .. (71.68,214.59) ;
\draw [shift={(70,215)}, rotate = 346.68] [color={rgb, 255:red, 0; green, 0; blue, 255 }  ,draw opacity=1 ][line width=0.75]    (7.65,-3.43) .. controls (4.86,-1.61) and (2.31,-0.47) .. (0,0) .. controls (2.31,0.47) and (4.86,1.61) .. (7.65,3.43)   ;
\draw [line width=2.25]    (220,230) -- (220,64) ;
\draw [shift={(220,60)}, rotate = 90] [color={rgb, 255:red, 0; green, 0; blue, 0 }  ][line width=2.25]    (12.24,-5.49) .. controls (7.79,-2.58) and (3.71,-0.75) .. (0,0) .. controls (3.71,0.75) and (7.79,2.58) .. (12.24,5.49)   ;
\draw [line width=2.25]    (280,230) -- (280,64) ;
\draw [shift={(280,60)}, rotate = 90] [color={rgb, 255:red, 0; green, 0; blue, 0 }  ][line width=2.25]    (12.24,-5.49) .. controls (7.79,-2.58) and (3.71,-0.75) .. (0,0) .. controls (3.71,0.75) and (7.79,2.58) .. (12.24,5.49)   ;
\draw [color={rgb, 255:red, 0; green, 0; blue, 255 }  ,draw opacity=1 ] [dash pattern={on 0.84pt off 2.51pt}]  (220,165) -- (250,190) ;
\draw [color={rgb, 255:red, 0; green, 0; blue, 255 }  ,draw opacity=1 ] [dash pattern={on 0.84pt off 2.51pt}]  (250,190) -- (280,170) ;
\draw [color={rgb, 255:red, 0; green, 0; blue, 255 }  ,draw opacity=1 ] [dash pattern={on 0.84pt off 2.51pt}]  (250,190) .. controls (247.05,205.76) and (257.67,206) .. (221.68,214.6) ;
\draw [shift={(220,215)}, rotate = 346.68] [color={rgb, 255:red, 0; green, 0; blue, 255 }  ,draw opacity=1 ][line width=0.75]    (7.65,-3.43) .. controls (4.86,-1.61) and (2.31,-0.47) .. (0,0) .. controls (2.31,0.47) and (4.86,1.61) .. (7.65,3.43)   ;
\draw [color={rgb, 255:red, 0; green, 0; blue, 255 }  ,draw opacity=1 ] [dash pattern={on 0.84pt off 2.51pt}]  (220,100) -- (250,125) ;
\draw [color={rgb, 255:red, 0; green, 0; blue, 255 }  ,draw opacity=1 ] [dash pattern={on 0.84pt off 2.51pt}]  (250,125) -- (280,105) ;
\draw [color={rgb, 255:red, 0; green, 0; blue, 255 }  ,draw opacity=1 ] [dash pattern={on 0.84pt off 2.51pt}]  (250,125) .. controls (247.05,140.76) and (257.67,141) .. (221.68,149.6) ;
\draw [shift={(220,150)}, rotate = 346.68] [color={rgb, 255:red, 0; green, 0; blue, 255 }  ,draw opacity=1 ][line width=0.75]    (7.65,-3.43) .. controls (4.86,-1.61) and (2.31,-0.47) .. (0,0) .. controls (2.31,0.47) and (4.86,1.61) .. (7.65,3.43)   ;
\draw [line width=2.25]    (350,230) -- (350,64) ;
\draw [shift={(350,60)}, rotate = 90] [color={rgb, 255:red, 0; green, 0; blue, 0 }  ][line width=2.25]    (12.24,-5.49) .. controls (7.79,-2.58) and (3.71,-0.75) .. (0,0) .. controls (3.71,0.75) and (7.79,2.58) .. (12.24,5.49)   ;
\draw [line width=2.25]    (410,230) -- (410,64) ;
\draw [shift={(410,60)}, rotate = 90] [color={rgb, 255:red, 0; green, 0; blue, 0 }  ][line width=2.25]    (12.24,-5.49) .. controls (7.79,-2.58) and (3.71,-0.75) .. (0,0) .. controls (3.71,0.75) and (7.79,2.58) .. (12.24,5.49)   ;
\draw [color={rgb, 255:red, 0; green, 0; blue, 255 }  ,draw opacity=1 ] [dash pattern={on 0.84pt off 2.51pt}]  (350,110) -- (380,140) ;
\draw [color={rgb, 255:red, 0; green, 0; blue, 255 }  ,draw opacity=1 ] [dash pattern={on 0.84pt off 2.51pt}]  (380,140) -- (410,110) ;
\draw [color={rgb, 255:red, 0; green, 0; blue, 255 }  ,draw opacity=1 ] [dash pattern={on 0.84pt off 2.51pt}]  (350,145) -- (356.67,149.17) -- (363.33,153.33) -- (370,157.5) -- (376.67,161.67) -- (383.33,165.83) -- (390,170) ;
\draw [color={rgb, 255:red, 0; green, 0; blue, 255 }  ,draw opacity=1 ] [dash pattern={on 0.84pt off 2.51pt}]  (390,170) -- (410,145) ;
\draw [color={rgb, 255:red, 0; green, 0; blue, 255 }  ,draw opacity=1 ] [dash pattern={on 0.84pt off 2.51pt}]  (390,170) .. controls (387.05,185.76) and (387.97,205.4) .. (351.68,214.59) ;
\draw [shift={(350,215)}, rotate = 346.68] [color={rgb, 255:red, 0; green, 0; blue, 255 }  ,draw opacity=1 ][line width=0.75]    (7.65,-3.43) .. controls (4.86,-1.61) and (2.31,-0.47) .. (0,0) .. controls (2.31,0.47) and (4.86,1.61) .. (7.65,3.43)   ;
\draw [color={rgb, 255:red, 0; green, 0; blue, 255 }  ,draw opacity=1 ] [dash pattern={on 0.84pt off 2.51pt}]  (380,140) -- (376.67,161.67) ;

\draw (161,132.4) node [anchor=north west][inner sep=0.75pt]    {$=$};
\draw (311,132.4) node [anchor=north west][inner sep=0.75pt]    {$-$};

\end{tikzpicture}\]
    \caption{A typical term of $N^{tr}$ is shown on the left. We perform an STU relation to separate the trees, producing a term of $e^{\ell(\eta)}$, and a term which cancels in the sum.}
    \label{fig:treeseparate}
\end{figure}
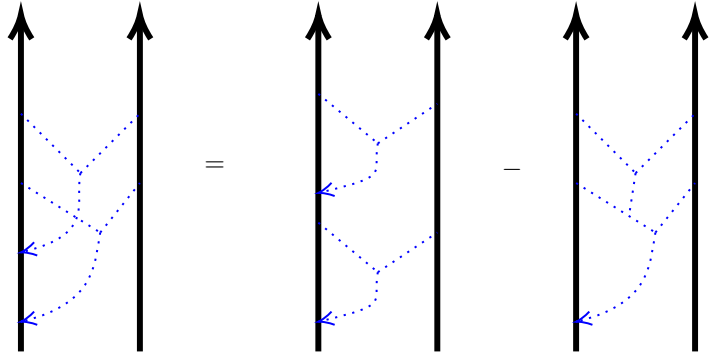

We are now ready to prove the unitarity equation \eqref{U'}: 

\begin{lemma}\label{lem:rhopunitarity}
There is an appropriate choice of $C\in \Q[[\xi]]/\langle\xi\rangle$, so that $N=C^{12}C^1C^2N^{tr}$ satisfies the unitarity equation \eqref{U'}:
\begin{align}
     N \cdot A_2A_2(N) &= 1 \quad \text{in} \quad \A(\uparrow_2)
\end{align}
\end{lemma}
\begin{proof}
First, we prove that unitarity is satisfied modulo wheels, that is,
\begin{align}
    N^{tr}\cdot A_1A_2(N^{tr})=1 \quad \text{in} \quad \apr(\arrows(\uparrow_2)).
\end{align}
Since by Lemma~\ref{lem:Ntr_elleta}, $N^{tr} = e^{\ell(\eta)}$, the left-hand side of the equation reads as 
\begin{align}
    e^{\ell(\eta)} \cdot A_1A_2(e^{\ell(\eta)}) = e^{\ell(\eta)} \cdot e^{A_1A_2(\ell(\eta))}.
\end{align}
By definition, $A_1A_2$ acts on a single tree by reversing the order of the arrow endings along the strands, and multiplying with (-1), as each tree has a single arrow head. Reversing the order of leaves means that the head of the tree is now attached above its tails. Using STU relations, the head can be commuted back to its original position below its tails, at the ``cost'' of wheels. Thus,  $A_1A_2(\ell(\eta))=-\ell(\eta)$ in $\arrows(\uparrow_2)^{tr}$, and therefore
\begin{align}
      e^{\ell(\eta)} \cdot e^{A_1A_2(\ell(\eta))} =  e^{\ell(\eta)} \cdot e^{-\ell(\eta)} = 1 \quad \text{in} \quad \arrows(\uparrow_2)^{tr}
\end{align}
Therefore, $N$ satisfies \eqref{U'} modulo wheels. 
Showing that $N$ satisfies \eqref{U'} in general is slightly more complex. 

By applying Lemma~\ref{lem:WNtr} and Lemma~\ref{lem:Ntr_elleta}, the left-hand side of the unitarity equation reads
\begin{align}
    W\cdot e^{\ell(\eta)}\cdot A_1A_2(W\cdot e^{\ell(\eta)}) =  W\cdot e^{\ell(\eta)} \cdot A_1A_2(e^{\ell(\eta)})\cdot W,
\end{align}
using that $A_1A_2$ reverses multiplication order, and acts trivially on wheels.

By Proposition~\ref{prop:J}, 
$$e^{\ell(\eta)} \cdot A_1A_2(e^{\ell(\eta)})=e^{-J(e^\eta)}.$$

Thus, in order to satisfy the \eqref{U'} equation, we must be able to choose a cap value $C = e^c$ such that 
\begin{align}
   e^{J(e^\eta)}=W^2=e^{2(c^{12}-c^1-c^2)}. 
\end{align}
$e^\eta$ is equal to $\rho(\Psi)$, and hence $e^{\eta} \in \krv$ \cite[Proposition 4.8]{AT12}. Thus, by definition (see~\eqref{eq:krvdef}), $J(e^\eta)$ is of the form $\operatorname{tr}(s(x) + s(y) - s(x+y))$ for some $s \in x^2\mathbb{Q}[[x]]$. Choose $C = e^c$ with $c=\frac{1}{2}s(x)$, and with this choice $N$ satisfies the \eqref{U'} equation.
\end{proof}

The following corollary concludes the alternative proof on the Main Theorem~\ref{thm:main}:

\begin{cor}\label{rhop_in_aut}
The map $r: \Aut(\PaCD) \to \calG(\uparrow_2)^{tr}$ factors through $\Aut_v(\A)$. That is, there exists a map $\rp: \Aut(\PaCD) \to \Aut_v(\A)$ such that the triangle (2) in diagram~\ref{eq:BigDiagram} commutes:
\begin{equation}
\begin{tikzcd}
	{\Aut(\PaCD)} && {\Aut_v(\A)} \\
	&& {\calG(\uparrow_2)^{tr}}
	\arrow["\rp", dashed, from=1-1, to=1-3]
	\arrow["v",red,hook, from=1-3, to=2-3]
	\arrow["r",red, from=1-1, to=2-3]
\end{tikzcd}
\end{equation}
\end{cor}
\begin{proof}
Define $\rp(F)$ by sending $F$ to the automorphism $G$ defined by the pair $(N,C)$ where $N$ is as in Theorem~\ref{thm:GfromF} and $C$ is as defined in the proof of Lemma~\ref{lem:rhopunitarity}. 

By Lemmas~\ref{lem:rhopr4},~\ref{lem:rhopcap}, and~\ref{lem:rhopunitarity} the pair $(N,C)$ satisfies the equations of Proposition~\ref{prop:SimplifiedEqns}, and by Lemma~\ref{lem:ShortArrows} $N$ is free of short arrows. Therefore, $G \in \Aut_v(\A)$. 
It is immediate that $r(F) = v(\rp(F))$ since $v(G) = N^{tr} = r(F)$. 
\end{proof}

\begin{remark}\label{rmk:direct}
The fact that the image of the Alekseev-Torossian map $\rho$ lies in $\krv$ was used in order to prove that the \eqref{U'} equation is satisfied in general (with wheels taken into account). It would be possible to prove this directly, by replacing the ``stripping'' construction with a more involved ``double tree'' construction along the lines of \cite{BND:WKO3}. This produces a planar algebra map compatible with unzips, and thus makes the unitarity equation almost immediate. However, introducing the construction and proving its necessary properties adds a large amount of complexity and length, which is covered in \cite{BND:WKO3} and thus we chose not to duplicate it here.
\end{remark}

\section{Circuit Algebras}\label{sec: circuit algebras}
An essential ingredient in the topological interpretation of the Kashiwara-Vergne problem is the construction of homomorphic expansions for a class 
of $w$-foams (\cite{BND:WKO2}, \cite{DHR21}). This construction relies on the fact that $w$-foams—and, in general, classes of knotted objects—can be 
described as finitely presented algebraic structures called (oriented) \emph{circuit algebras}. Circuit algebras are a generalisation of Jones' planar algebras, where one drops the condition of ``planarity''. In this appendix, we give a brief overview of circuit algebras and recommend \cite[Section 2]{DHR1} for full details.

In a circuit algebra, operations are parameterised by {\em wiring diagrams}. 

\begin{definition}\label{def: wiring diagram}
An \emph{oriented wiring diagram} is a triple $D=(\calA,M,f)$ consisting of:
	\begin{enumerate}
		\item A set $\calA=\{A_0^{\text{out}}, A_0^{\text{in}},A_1^{\text{out}},A_1^{\text{in}},\hdots,A_r^{\text{out}},A_r^{\text{in}}\}$ of  sets of labels, for some non-negative integer $r$. The elements of the sets $A^{\text{out}}_i$ are referred to as \emph{outgoing labels} and the elements of $A^{\text{in}}_i$ are \emph{incoming labels}. The sets $A_0^{\tout}$ and $A_{0}^{\tin}$ play a distinguished role: their elements are called the \emph{output labels} of the diagram, while the sets $A_1^{{\text{out}}}, A_{1}^{\text{in}},...,A_r^{\text{out}},  A_r^{\text{in}}$ contain {\em input labels} of the diagram. 
		
		\item An oriented compact $1$-manifold $M$, with boundary $\partial M$, regarded up to orientation-preserving homeomorphism. The connected components of $M$ are homeomorphic to either an oriented circle (with no boundary) or an oriented interval with one beginning and one ending point. We write $\partial M^{\tout}$ for the set of {\em beginning} boundary points of $M$, and  $\partial M^{\text{in}}$ for the set of {\em ending}  boundary points, so $\partial M = \partial M^{\text{out}} \sqcup \partial M^{\tin}$.
		
		\item Bijections\footnote{If the sets $\{A_i^{\text{out}/\text{in}}\}$ are not pairwise disjoint, replace the unions $(\cup_{i=0}^r A_i^{\text{out}})$ and  $(\cup_{i=0}^r A_i^{\text{in}})$ by the set of triples $\{(a,i,\text{out}/\text{in})\, | \, a\in A_i^{\text{out}/\text{in}}, 0\leq i \leq r\}$.} \label{fn:Triples}
		$$\partial M^{\text{out}} \xrightarrow{f} \cup^r_{i=0} {A^{\text{out}}_i} \quad \textrm{and}\quad  \partial M^{\text{in}} \xrightarrow{f} \cup^r_{i=0} {A^{\text{in}}_i}.$$
	\end{enumerate}

\end{definition}

Wiring diagrams can be pictorially represented as ``disks with $r$ holes,'' as in Figure~\ref{fig:WD}. A disc with $r$ holes is a closed disc with the interiors of $r$ disjoint small discs removed $D_0\setminus (\mathring{D}_1\cup \ldots \cup \mathring{D}_r)$. A set of marked points on the disc boundaries $\mathcal{L}=\{\mathcal{L}_0, \mathcal{L}_1,\ldots, \mathcal{L}_r\}$ are arranged so that $\mathcal{L}_i$ lies on $\partial D_i$, and $\mathcal{L}_i$ is in bijection with $A_i^{\tout} \cup A_{i}^{\tin}$. The manifold is represented as an immersion $M \subset D$, where $\partial M$ is mapped to the set of marked points on the boundary of $D$, so that interval components of $M$ connect the points paired by the bijection $f$. The set $\partial M$ is partitioned to incoming and outgoing boundary points: incoming points are mapped to marked points corresponding to $\bigcup_i A_i^\tin$ and outgoing points to $\bigcup_i A_i^\tout$. 

Note that the choice of a particular immersion of $M$ is only for notational convenience, it is not part of the data. Different immersions of $M$ represent the same wiring diagram as long as they induce the same bijection between incoming and outgoing marked points. See \cite[Lemma 5.3]{DHR1} for further details. 

\begin{example}

Figure~\ref{fig:WD} depicts the wiring diagram represented by the triple
\begin{multline*}
    \WD=\{\{\{l_1,l_3\},\{l_2,l_4,l_5\},\{a_1,a_3,a_4\},\{a_2\},\{b_4,b_5\},\{b_1,b_2,b_3\} \},\sqcup_{i=1}^7 I,\\
    \{(l_1,l_2),(l_3,b_2),(a_1,l_4),(a_3,b_1),(a_4,l_5),(b_4,b_3),(b_5,a_2)\} \}
\end{multline*}

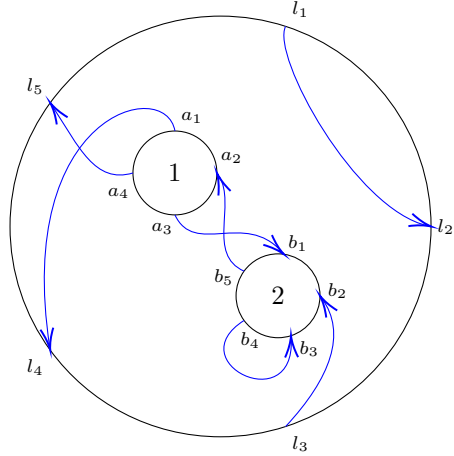
\begin{figure}[h]
\[\begin{tikzpicture}[x=0.75pt,y=0.75pt,yscale=-1,xscale=1]

\draw  [draw opacity=0] (238.33,138.17) .. controls (238.33,138.17) and (238.33,138.17) .. (238.33,138.17) .. controls (238.33,185.35) and (207.56,225.34) .. (164.98,239.16) -- (132.17,138.17) -- cycle ; \draw    (238.33,138.17) .. controls (238.33,185.35) and (207.56,225.34) .. (164.98,239.16) ;  
\draw  [draw opacity=0] (130.33,111.17) .. controls (130.33,111.17) and (130.33,111.17) .. (130.33,111.17) .. controls (130.33,122.86) and (120.86,132.33) .. (109.17,132.33) -- (109.17,111.17) -- cycle ; \draw   (130.33,111.17) .. controls (130.33,111.17) and (130.33,111.17) .. (130.33,111.17) .. controls (130.33,122.86) and (120.86,132.33) .. (109.17,132.33) ;  
\draw  [draw opacity=0] (182.33,173.17) .. controls (182.33,173.17) and (182.33,173.17) .. (182.33,173.17) .. controls (182.33,173.17) and (182.33,173.17) .. (182.33,173.17) .. controls (182.33,182.57) and (176.2,190.55) .. (167.71,193.3) -- (161.17,173.17) -- cycle ; \draw   (182.33,173.17) .. controls (182.33,173.17) and (182.33,173.17) .. (182.33,173.17) .. controls (182.33,173.17) and (182.33,173.17) .. (182.33,173.17) .. controls (182.33,182.57) and (176.2,190.55) .. (167.71,193.3) ;  
\draw  [draw opacity=0] (164.98,239.16) .. controls (154.65,242.52) and (143.62,244.33) .. (132.17,244.33) .. controls (96.86,244.33) and (65.57,227.09) .. (46.27,200.57) -- (132.17,138.17) -- cycle ; \draw    (164.98,239.16) .. controls (154.65,242.52) and (143.62,244.33) .. (132.17,244.33) .. controls (96.86,244.33) and (65.57,227.09) .. (46.27,200.57) ;  
\draw  [draw opacity=0] (46.27,200.57) .. controls (33.52,183.06) and (26,161.49) .. (26,138.17) .. controls (26,114.84) and (33.52,93.28) .. (46.27,75.76) -- (132.17,138.17) -- cycle ; \draw    (46.27,200.57) .. controls (33.52,183.06) and (26,161.49) .. (26,138.17) .. controls (26,114.84) and (33.52,93.28) .. (46.27,75.76) ;  
\draw  [draw opacity=0] (46.27,75.76) .. controls (65.57,49.24) and (96.86,32) .. (132.17,32) .. controls (143.62,32) and (154.65,33.81) .. (164.98,37.17) -- (132.17,138.17) -- cycle ; \draw    (46.27,75.76) .. controls (65.57,49.24) and (96.86,32) .. (132.17,32) .. controls (143.62,32) and (154.65,33.81) .. (164.98,37.17) ;  
\draw  [draw opacity=0] (164.98,37.17) .. controls (207.56,50.99) and (238.33,90.99) .. (238.33,138.17) -- (132.17,138.17) -- cycle ; \draw    (164.98,37.17) .. controls (207.56,50.99) and (238.33,90.99) .. (238.33,138.17) ;  
\draw  [draw opacity=0] (109.17,132.33) .. controls (109.17,132.33) and (109.17,132.33) .. (109.17,132.33) .. controls (109.17,132.33) and (109.17,132.33) .. (109.17,132.33) .. controls (97.48,132.33) and (88,122.86) .. (88,111.17) -- (109.17,111.17) -- cycle ; \draw   (109.17,132.33) .. controls (109.17,132.33) and (109.17,132.33) .. (109.17,132.33) .. controls (109.17,132.33) and (109.17,132.33) .. (109.17,132.33) .. controls (97.48,132.33) and (88,122.86) .. (88,111.17) ;  
\draw  [draw opacity=0] (88,111.17) .. controls (88,111.17) and (88,111.17) .. (88,111.17) .. controls (88,99.48) and (97.48,90) .. (109.17,90) -- (109.17,111.17) -- cycle ; \draw   (88,111.17) .. controls (88,111.17) and (88,111.17) .. (88,111.17) .. controls (88,99.48) and (97.48,90) .. (109.17,90) ;  
\draw  [draw opacity=0] (109.17,90) .. controls (109.17,90) and (109.17,90) .. (109.17,90) .. controls (109.17,90) and (109.17,90) .. (109.17,90) .. controls (120.86,90) and (130.33,99.48) .. (130.33,111.17) -- (109.17,111.17) -- cycle ; \draw   (109.17,90) .. controls (109.17,90) and (109.17,90) .. (109.17,90) .. controls (109.17,90) and (109.17,90) .. (109.17,90) .. controls (120.86,90) and (130.33,99.48) .. (130.33,111.17) ;  
\draw [color={rgb, 255:red, 0; green, 0; blue, 255 }  ,draw opacity=1 ]   (109.17,90) .. controls (104.36,62.14) and (28.1,80.81) .. (45.99,198.79) ;
\draw [shift={(46.27,200.57)}, rotate = 261] [color={rgb, 255:red, 0; green, 0; blue, 255 }  ,draw opacity=1 ][line width=0.75]    (10.93,-3.29) .. controls (6.95,-1.4) and (3.31,-0.3) .. (0,0) .. controls (3.31,0.3) and (6.95,1.4) .. (10.93,3.29)   ;
\draw [color={rgb, 255:red, 0; green, 0; blue, 255 }  ,draw opacity=1 ]   (88,111.17) .. controls (68.82,116.85) and (60.74,97.98) .. (47.31,77.35) ;
\draw [shift={(46.27,75.76)}, rotate = 56.49] [color={rgb, 255:red, 0; green, 0; blue, 255 }  ,draw opacity=1 ][line width=0.75]    (10.93,-3.29) .. controls (6.95,-1.4) and (3.31,-0.3) .. (0,0) .. controls (3.31,0.3) and (6.95,1.4) .. (10.93,3.29)   ;
\draw  [draw opacity=0] (167.71,193.3) .. controls (165.65,193.97) and (163.45,194.33) .. (161.17,194.33) .. controls (154.13,194.33) and (147.89,190.9) .. (144.04,185.61) -- (161.17,173.17) -- cycle ; \draw   (167.71,193.3) .. controls (165.65,193.97) and (163.45,194.33) .. (161.17,194.33) .. controls (154.13,194.33) and (147.89,190.9) .. (144.04,185.61) ;  
\draw  [draw opacity=0] (144.04,185.61) .. controls (141.5,182.12) and (140,177.82) .. (140,173.17) .. controls (140,168.52) and (141.5,164.22) .. (144.04,160.72) -- (161.17,173.17) -- cycle ; \draw   (144.04,185.61) .. controls (141.5,182.12) and (140,177.82) .. (140,173.17) .. controls (140,168.52) and (141.5,164.22) .. (144.04,160.72) ;  
\draw  [draw opacity=0] (144.04,160.72) .. controls (147.89,155.44) and (154.13,152) .. (161.17,152) .. controls (162.17,152) and (163.15,152.07) .. (164.11,152.2) -- (161.17,173.17) -- cycle ; \draw   (144.04,160.72) .. controls (147.89,155.44) and (154.13,152) .. (161.17,152) .. controls (162.17,152) and (163.15,152.07) .. (164.11,152.2) ;  
\draw  [draw opacity=0] (164.11,152.2) .. controls (174.41,153.64) and (182.33,162.48) .. (182.33,173.17) .. controls (182.33,173.17) and (182.33,173.17) .. (182.33,173.17) -- (161.17,173.17) -- cycle ; \draw   (164.11,152.2) .. controls (174.41,153.64) and (182.33,162.48) .. (182.33,173.17) .. controls (182.33,173.17) and (182.33,173.17) .. (182.33,173.17) ;  
\draw [color={rgb, 255:red, 0; green, 0; blue, 255 }  ,draw opacity=1 ]   (131.52,112.89) .. controls (140.03,127.13) and (125.12,152.35) .. (144.04,160.72) ;
\draw [shift={(130.33,111.17)}, rotate = 69.48] [color={rgb, 255:red, 0; green, 0; blue, 255 }  ,draw opacity=1 ][line width=0.75]    (10.93,-3.29) .. controls (6.95,-1.4) and (3.31,-0.3) .. (0,0) .. controls (3.31,0.3) and (6.95,1.4) .. (10.93,3.29)   ;
\draw [color={rgb, 255:red, 0; green, 0; blue, 255 }  ,draw opacity=1 ]   (109.17,132.33) .. controls (119.13,153.57) and (144.3,129.98) .. (162.98,150.87) ;
\draw [shift={(164.11,152.2)}, rotate = 221.47] [color={rgb, 255:red, 0; green, 0; blue, 255 }  ,draw opacity=1 ][line width=0.75]    (10.93,-3.29) .. controls (6.95,-1.4) and (3.31,-0.3) .. (0,0) .. controls (3.31,0.3) and (6.95,1.4) .. (10.93,3.29)   ;
\draw [color={rgb, 255:red, 0; green, 0; blue, 255 }  ,draw opacity=1 ]   (144.04,185.61) .. controls (110.84,206.68) and (169.52,235.13) .. (167.83,195.18) ;
\draw [shift={(167.71,193.3)}, rotate = 88.42] [color={rgb, 255:red, 0; green, 0; blue, 255 }  ,draw opacity=1 ][line width=0.75]    (10.93,-3.29) .. controls (6.95,-1.4) and (3.31,-0.3) .. (0,0) .. controls (3.31,0.3) and (6.95,1.4) .. (10.93,3.29)   ;
\draw [color={rgb, 255:red, 0; green, 0; blue, 255 }  ,draw opacity=1 ]   (164.98,239.16) .. controls (187.75,214.63) and (195.92,192.15) .. (183.34,174.51) ;
\draw [shift={(182.33,173.17)}, rotate = 57.21] [color={rgb, 255:red, 0; green, 0; blue, 255 }  ,draw opacity=1 ][line width=0.75]    (10.93,-3.29) .. controls (6.95,-1.4) and (3.31,-0.3) .. (0,0) .. controls (3.31,0.3) and (6.95,1.4) .. (10.93,3.29)   ;
\draw [color={rgb, 255:red, 0; green, 0; blue, 255 }  ,draw opacity=1 ]   (164.98,37.17) .. controls (157.45,51.78) and (200.02,129.61) .. (236.66,137.84) ;
\draw [shift={(238.33,138.17)}, rotate = 197.55] [color={rgb, 255:red, 0; green, 0; blue, 255 }  ,draw opacity=1 ][line width=0.75]    (10.93,-3.29) .. controls (6.95,-1.4) and (3.31,-0.3) .. (0,0) .. controls (3.31,0.3) and (6.95,1.4) .. (10.93,3.29)   ;

\draw (109.17,111.17) node    {$1$};
\draw (161.17,173.17) node    {$2$};
\draw (166.98,33.77) node [anchor=south west] [inner sep=0.75pt]  [font=\scriptsize]  {$l_{1}$};
\draw (240.33,138.17) node [anchor=west] [inner sep=0.75pt]  [font=\scriptsize]  {$l_{2}$};
\draw (166.98,242.56) node [anchor=north west][inner sep=0.75pt]  [font=\scriptsize]  {$l_{3}$};
\draw (44.27,203.97) node [anchor=north east] [inner sep=0.75pt]  [font=\scriptsize]  {$l_{4}$};
\draw (44.27,72.36) node [anchor=south east] [inner sep=0.75pt]  [font=\scriptsize]  {$l_{5}$};
\draw (111,79.4) node [anchor=north west][inner sep=0.75pt]  [font=\scriptsize]  {$a_{1}$};
\draw (131,99.4) node [anchor=north west][inner sep=0.75pt]  [font=\scriptsize]  {$a_{2}$};
\draw (96,135.4) node [anchor=north west][inner sep=0.75pt]  [font=\scriptsize]  {$a_{3}$};
\draw (74,117.4) node [anchor=north west][inner sep=0.75pt]  [font=\scriptsize]  {$a_{4}$};
\draw (165,141.4) node [anchor=north west][inner sep=0.75pt]  [font=\scriptsize]  {$b_{1}$};
\draw (185,165.4) node [anchor=north west][inner sep=0.75pt]  [font=\scriptsize]  {$b_{2}$};
\draw (170.71,194.7) node [anchor=north west][inner sep=0.75pt]  [font=\scriptsize]  {$b_{3}$};
\draw (141,190.4) node [anchor=north west][inner sep=0.75pt]  [font=\scriptsize]  {$b_{4}$};
\draw (140.04,158.62) node [anchor=north east] [inner sep=0.75pt]  [font=\scriptsize]  {$b_{5}$};

\end{tikzpicture}
\]
    \caption{An example wiring diagram}
    \label{fig:WD}
\end{figure}
\end{example}

\begin{definition} Given a set of colours $\mathfrak{C}$, a \emph{$\mathfrak{C}$-coloured operad} $\mathsf{P}=\{\mathsf{P}(c_1,\ldots,c_r;c_0)\}$ consists of a collection of sets (or vector spaces) $\mathsf{P}(c_1,\ldots,c_r;c_0)$, one for each sequence $(c_1,\ldots,c_r;c_0) \in\mathfrak{C}$, which are equipped with a right action of the symmetric group $\calS_r$ which permutes the ``inputs'' $c_1,\ldots,c_r$. Moreover, $\mathsf{P}$ is equipped with a family of equivariant, associative and unital partial compositions:
\[\begin{tikzcd}\circ_i:\mathsf{P}(c_1,\ldots,c_r;c_0)\times\mathsf{P}(d_1,\ldots,d_s;d_0)\arrow[r]& \mathsf{P}(c_1,\ldots,c_{i-1}, d_1,\ldots, d_s, c_{i+1},\ldots, c_r;c_0), \end{tikzcd}\] whenever $d_0=c_i$. \end{definition}

For more details on coloured operads see, for example, \cite[Definition 1.1]{bm_resolutions}. An \emph{algebra} over an $\mathfrak{C}$-coloured operad $\mathsf{P}$ is a $\mathfrak{C}$-indexed family of vector spaces $\mathsf{A}=\{\mathsf{A}(c)\}_{c\in\mathfrak{C}}$ together with an action by $\mathsf{P}$ (\cite[Definition 1.2]{bm_resolutions}). 
 
\begin{definition}\label{def:operad_WD}
The collection of all wiring diagrams $\mathsf{WD} =\{\mathsf{WD}(A^{\tin/\tout}_0; A^{\tin / \tout}_1,\ldots,A^{\tin/\tout}_r)\}$ forms a discrete operad called the \emph{operad of oriented wiring diagrams}. For $\WD=(\calA,M,f)$ and $\WD'=(\mathcal B,N,g)$, if $A_i^{\tin}=B_0^{\tout}$ and  $A_i^{\tout}=B_0^{\tin}$ then the partial composition $\WD\circ_i \WD'$ is defined by the label set \linebreak $\{A_1^{\tin/\tout},...,A_{i-1}^{\tin/\tout}, B_1^{\tin/\tout},...,B_s^{\tin/\tout},A_{i+1}^{\tin/\tout}...A_r^{\tin/\tout}\}$; the manifold $M\sqcup N/ \sim $ obtained from $M$ and $N$ by gluing along the boundary identification $A_i^{\tin}=B_0^{\tout}$ and  $A_i^{\tout}=B_0^{\tin}$; and the set bijection $f \sqcup g /\sim $ induced by $f$ and $g$ in the natural way. 
\end{definition}

Pictorially, wiring diagram composition shrinks the wiring diagram $\WD'$ and glues it into the $i$th input circle of $\WD$ in such a way that the labels match and the boundary points of the $1$-manifolds are identified: see Figure~\ref{fig:composition}. Note that composition in the operad $\mathsf{WD}$ is similar to the operad of planar tangles in \cite{Jones:PA} and \cite[Definition  2.1]{HPT16}. 

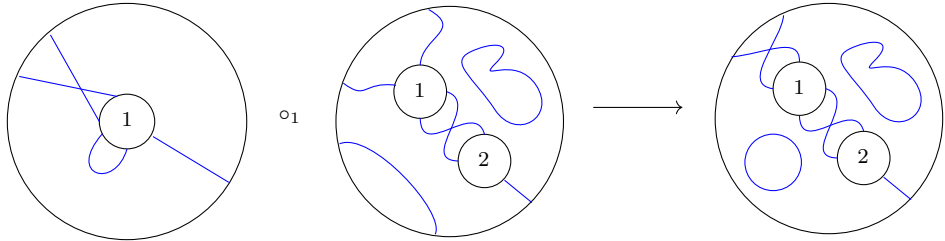
\begin{figure}
\[\begin{tikzpicture}[x=0.75pt,y=0.75pt,yscale=-.75,xscale=.75]

\draw   (9,139.5) .. controls (9,95.59) and (45.04,60) .. (89.5,60) .. controls (133.96,60) and (170,95.59) .. (170,139.5) .. controls (170,183.41) and (133.96,219) .. (89.5,219) .. controls (45.04,219) and (9,183.41) .. (9,139.5) -- cycle ;
\draw   (70.87,139.5) .. controls (70.87,129.34) and (79.21,121.1) .. (89.5,121.1) .. controls (99.79,121.1) and (108.13,129.34) .. (108.13,139.5) .. controls (108.13,149.66) and (99.79,157.9) .. (89.5,157.9) .. controls (79.21,157.9) and (70.87,149.66) .. (70.87,139.5) -- cycle ;
\draw [color={rgb, 255:red, 0; green, 0; blue, 255 }  ,draw opacity=1 ]   (107,149.62) -- (158,180.45) ;
\draw   (230,139.5) .. controls (230,96.7) and (264.25,62) .. (306.5,62) .. controls (348.75,62) and (383,96.7) .. (383,139.5) .. controls (383,182.3) and (348.75,217) .. (306.5,217) .. controls (264.25,217) and (230,182.3) .. (230,139.5) -- cycle ;
\draw   (312.17,166.05) .. controls (312.17,156.14) and (320.09,148.11) .. (329.88,148.11) .. controls (339.66,148.11) and (347.58,156.14) .. (347.58,166.05) .. controls (347.58,175.96) and (339.66,183.99) .. (329.88,183.99) .. controls (320.09,183.99) and (312.17,175.96) .. (312.17,166.05) -- cycle ;
\draw   (268.96,119.41) .. controls (268.96,109.5) and (276.89,101.47) .. (286.67,101.47) .. controls (296.45,101.47) and (304.38,109.5) .. (304.38,119.41) .. controls (304.38,129.32) and (296.45,137.35) .. (286.67,137.35) .. controls (276.89,137.35) and (268.96,129.32) .. (268.96,119.41) -- cycle ;
\draw [color={rgb, 255:red, 0; green, 0; blue, 255 }  ,draw opacity=1 ]   (304.38,119.41) .. controls (328.46,127.3) and (284.54,165.33) .. (312.17,166.05) ;
\draw [color={rgb, 255:red, 0; green, 0; blue, 255 }  ,draw opacity=1 ]   (286.67,137.35) .. controls (287.38,166.05) and (323.5,124.43) .. (329.88,148.11) ;
\draw [color={rgb, 255:red, 0; green, 0; blue, 255 }  ,draw opacity=1 ]   (234.96,113.67) .. controls (251.25,125.87) and (247.71,112.23) .. (270.37,115.1) ;
\draw [color={rgb, 255:red, 0; green, 0; blue, 255 }  ,draw opacity=1 ]   (286.67,101.47) .. controls (287.38,82.81) and (316.42,77.79) .. (292.33,64.15) ;
\draw [color={rgb, 255:red, 0; green, 0; blue, 255 }  ,draw opacity=1 ]   (343.33,178.97) -- (361.75,194.04) ;
\draw  [color={rgb, 255:red, 0; green, 0; blue, 255 }  ,draw opacity=1 ] (319.96,92.86) .. controls (334.13,85.68) and (351.83,85.68) .. (337.67,100.03) .. controls (323.5,114.38) and (350.42,93.57) .. (364.58,115.1) .. controls (378.75,136.63) and (351.13,153.13) .. (336.96,131.61) .. controls (322.79,110.08) and (305.79,100.03) .. (319.96,92.86) -- cycle ;
\draw [color={rgb, 255:red, 0; green, 0; blue, 255 }  ,draw opacity=1 ]   (232.44,154.51) .. controls (253.6,145.59) and (304.06,200.77) .. (296.73,215.38) ;
\draw [color={rgb, 255:red, 0; green, 0; blue, 255 }  ,draw opacity=1 ]   (17,109.15) -- (83,122.64) ;
\draw [color={rgb, 255:red, 0; green, 0; blue, 255 }  ,draw opacity=1 ]   (38,81.2) -- (70.87,139.5) ;
\draw [color={rgb, 255:red, 0; green, 0; blue, 255 }  ,draw opacity=1 ]   (73,147.69) .. controls (48,176) and (81.5,185.85) .. (89.5,157.9) ;
\draw   (485,137.5) .. controls (485,94.7) and (519.25,60) .. (561.5,60) .. controls (603.75,60) and (638,94.7) .. (638,137.5) .. controls (638,180.3) and (603.75,215) .. (561.5,215) .. controls (519.25,215) and (485,180.3) .. (485,137.5) -- cycle ;
\draw   (567.17,164.05) .. controls (567.17,154.14) and (575.09,146.11) .. (584.88,146.11) .. controls (594.66,146.11) and (602.58,154.14) .. (602.58,164.05) .. controls (602.58,173.96) and (594.66,181.99) .. (584.88,181.99) .. controls (575.09,181.99) and (567.17,173.96) .. (567.17,164.05) -- cycle ;
\draw   (523.96,117.41) .. controls (523.96,107.5) and (531.89,99.47) .. (541.67,99.47) .. controls (551.45,99.47) and (559.38,107.5) .. (559.38,117.41) .. controls (559.38,127.32) and (551.45,135.35) .. (541.67,135.35) .. controls (531.89,135.35) and (523.96,127.32) .. (523.96,117.41) -- cycle ;
\draw [color={rgb, 255:red, 0; green, 0; blue, 255 }  ,draw opacity=1 ]   (559.38,117.41) .. controls (583.46,125.3) and (539.54,163.33) .. (567.17,164.05) ;
\draw [color={rgb, 255:red, 0; green, 0; blue, 255 }  ,draw opacity=1 ]   (541.67,135.35) .. controls (542.38,164.05) and (578.5,122.43) .. (584.88,146.11) ;
\draw [color={rgb, 255:red, 0; green, 0; blue, 255 }  ,draw opacity=1 ]   (531,68) .. controls (530,84) and (501.29,114.54) .. (523.96,117.41) ;
\draw [color={rgb, 255:red, 0; green, 0; blue, 255 }  ,draw opacity=1 ]   (541.67,99.47) .. controls (542.38,80.81) and (518,94) .. (496,96) ;
\draw [color={rgb, 255:red, 0; green, 0; blue, 255 }  ,draw opacity=1 ]  (598.33,176.97) -- (616.75,192.04) ;
\draw   [color={rgb, 255:red, 0; green, 0; blue, 255 }  ,draw opacity=1 ]  (574.96,90.86) .. controls (589.13,83.68) and (606.83,83.68) .. (592.67,98.03) .. controls (578.5,112.38) and (605.42,91.57) .. (619.58,113.1) .. controls (633.75,134.63) and (606.13,151.13) .. (591.96,129.61) .. controls (577.79,108.08) and (560.79,98.03) .. (574.96,90.86) -- cycle ;
\draw   [color={rgb, 255:red, 0; green, 0; blue, 255 }  ,draw opacity=1 ]  (505,167) .. controls (505,156.51) and (513.51,148) .. (524,148) .. controls (534.49,148) and (543,156.51) .. (543,167) .. controls (543,177.49) and (534.49,186) .. (524,186) .. controls (513.51,186) and (505,177.49) .. (505,167) -- cycle ;
\draw[->]    (403,130) -- (463,130) ;

\draw (84,131.4) node [anchor=north west][inner sep=0.75pt] [font=\footnotesize]   {$1$};
\draw (281,112.4) node [anchor=north west][inner sep=0.75pt]   [font=\footnotesize] {$1$};
\draw (324,158.4) node [anchor=north west][inner sep=0.75pt] [font=\footnotesize]   {$2$};
\draw (536,110.4) node [anchor=north west][inner sep=0.75pt]  [font=\footnotesize]  {$1$};
\draw (579,156.4) node [anchor=north west][inner sep=0.75pt]  [font=\footnotesize]  {$2$};
\draw (190,130) node [anchor=north west][inner sep=0.75pt]  [font=\footnotesize]  {$\circ _{1}$};

\end{tikzpicture}\]
    \caption{The operadic composition of two wiring diagrams.}
    \label{fig:composition}
\end{figure}

\begin{definition}\label{def: circut algebra}
An oriented circuit algebra is an algebra over the operad of wiring diagrams $\mathsf{WD}$. 
\end{definition}

We briefly unravel the definition of an algebra over an operad (see also \cite[Definition 2.4]{DHR1}). An oriented circuit algebra is a collection of sets $\V=\{\V[A^{\tin};A^{\tout}]\}$, where the pairs $(A^{\tin}, A^{\tout})$ run over all pairs of label sets, together with a family of multiplication operations parameterised by oriented wiring diagrams. This means that, for each wiring diagram $\WD=(\mathcal{A},M,f)$, there is a corresponding function 
		$$F_\WD:\V[A^{\tin}_1;A^{\tout}_1] \times \hdots \times \V[A^{\tin}_r;A^{\tout}_{r}] \rightarrow \V[A^{\tout}_0;A^{\tin}_0].$$ 
Moreover, the assignment $\WD \mapsto F_\WD$ must be compatible with the operadic composition of wiring diagrams. 

\begin{definition}A \emph{homomorphism} of circuit algebras is a map of algebras over the operad of wiring diagrams. In other words, a homomorphism $\Phi:\mathsf{V}\rightarrow\mathsf{W}$ is a family of maps 
	$\{\Phi_{A^{\tin};A^\tout}:\mathsf{V}[A^\tin;A^\tout]\rightarrow \mathsf{W}[A^\tin;A^\tout]\}_{(A^{\tin}, A^{\tout})}$,
	which commutes with the action of wiring diagrams. For details, see \cite[Definition 2.5]{DHR1}.
\end{definition}

Wiring diagrams and circuit algebras can be naturally upgraded to a {\em coloured} context:

\begin{definition}\label{def:colouredCA}
    Given a finite set $\mathcal C$ of {\em colours}, an oriented {\em $\mathcal C$-coloured} wiring diagram is a wiring diagram $\WD=(\calA, M, f)$, along with a map
    $c: \pi_0(M) \to \mathcal C$. Here $\pi_0(M)$ is the set of connected components of $M$. Composition of $\mathcal C$-coloured wiring diagrams $$(\calA, M, f, c)\circ_i (\calB, N, g, d)$$ is only defined where the gluing of $M$ and $N$ respects colours. In other words, only connected components of the same colour may be glued.  An {\em oriented coloured circuit algebra} is an algebra over the coloured operad of oriented coloured wiring diagrams. 
\end{definition}

Elements of a circuit algebra can be viewed as decorations of wiring diagrams by elements of $\mathsf{V}$, where the operations are ``wiring elements together'': this is illustrated in the important example of the circuit algebra of $w$-foams described in Section~\ref{sec:foams} (e.g. Figure~\ref{fig:wfoam}). 

Circuit algebras can be given by a presentation $\mathsf{V} =\left<g_1,..,g_k\mid r_1,...,r_l\right>$ where the free circuit algebra on the generators $\{g_1,...,g_l\}$ is factored by the ideal generated by the relations $\{r_1,...,r_1\}$. The notions of {\em free} and {\em ideal} are the standard definitions for algebras over and operad (e.g.: \cite{algebraic_operads}). Examples of circuit algebras in this paper will also be equipped with {\em auxiliary operations}. These are additional operations which are not parameterised by wiring diagrams. 

\subsection{Circuit algebras as wheeled props}

Circuit algebras also possess an equivalent description as wheeled props \cite{DHR1}. A (coloured) \emph{prop}\footnote{Some authors prefer to capitalise the word PROP to emphasise that prop refers to ``PROduct and Permutation category."} is a strict symmetric tensor category in which the {\em monoid of objects} is freely generated (\cite[Chapter V]{maclane1965categorical}, \cite{hr1} or \cite{yj15}). This means that a one-coloured prop is a symmetric tensor category equipped with a distinguished object $x$, such that every object is a tensor power $x^{\otimes n}$, for some $n\geq 0$. Hence, morphisms in a prop are of the form $f:x^{\otimes n}\rightarrow x^{\otimes m}$. We picture these morphisms as graphs with $n$-inputs and $m$-outputs.  Composition of morphisms, also called {\em vertical composition}, is modelled diagrammatically by stacking the outputs of one graph to the inputs of another. The tensor product of the prop is realised by taking disjoint unions of graphs (stacking them next to each other), and is called {\em horizontal composition}.

\begin{example}
The collection of all braid groups $\mathsf{B}=\coprod \mathsf{B}_n$ forms a one-coloured prop (e.g. \cite[Example 1, page 55]{Joyal_Street}). Similarly, parenthesised braids $\PaB=\coprod_{n}\PaB(n)$ form a prop (\cite{BNGT}). 
\end{example}

A \emph{wheeled prop} is a prop in which every object has a dual. This gives rise to a family of linear ``trace'' or ``contraction'' maps \[\begin{tikzcd}\trace^{j}_{i}: x^{* \otimes m}\otimes x^{\otimes n}\arrow[r]& x^{*\otimes m-1}\otimes x^{\otimes n-1}.\end{tikzcd}\]  They arise, for example, in the Batalin-Vilkovisky quantization formalism, where formal germs of SP-manifolds are identified as representations of a certain wheeled prop \cite[Theorem 3.4.3]{mms}.  We note that in the literature, a strict symmetric tensor category with duals is also called a \emph{rigid} symmetric\footnote{Rigid symmetric tensor categories are also called compact closed categories.} tensor category. 

\begin{definition}\label{def:baised wheeled prop} Let $\calI$ denote a fixed countable alphabet. A \emph{wheeled prop} $\mathsf{E}:=\left(\mathsf{E}, \ast,\trace_{i}^{j}\right)$ consists of:
	\begin{enumerate}
		
		\item a collection of vector spaces $\mathsf{E}=\{\sE[I ;J]\}$, $I,J\subset \calI$;
		
		\item a \emph{horizontal composition}\[\begin{tikzcd}\ast:\sE[I;J] \otimes \sE[K;L] \arrow[r]& \sE[I\cup K;J\cup L], \end{tikzcd}\]where $I\cap K =\emptyset$ and $J\cap L =\emptyset$;
		
		\item a linear map $1_{\emptyset}:\Bbbk \to \sE[\emptyset;\emptyset]$ called the \emph{empty unit};
		
		\item a \emph{contraction} operation\[\label{note:xiij} \begin{tikzcd}\sE[I;J] \arrow[r, "\trace^{j}_{i}"] & \sE[I\setminus\{i\};J\setminus\{j\}],\end{tikzcd}\] for every pair $i\in I$ and $j\in J$;
		
		\item a linear map $$1_i: \Bbbk \to \sE[\{i\};\{i\}],$$ for every $i\in \calI$, called the {\em unit}.
		
	\end{enumerate} Moreover, the horizontal composition and contractions commute with each other and are associative, equivariant and unital (see \cite[Definition 6.1]{DHR1} for a full list of axioms). 
\end{definition} 

\begin{example}
The tensor category of tangles described by Turaev \cite{MR1027006} (see also the expository description in \cite[Section 11]{Joyal_Street}) is a wheeled prop. Similarly, the categories of virtual and welded tangles, as described in \cite[Section 3.4]{DHR1}. 
\end{example}

\begin{thm}\cite[Theorem 5.5]{DHR1}
There exists an equivalence of categories between the category of ($\mathfrak{C}$-coloured) circuit algebras and the category of ($\mathfrak{C}$-coloured) wheeled props. 
\end{thm}

\bibliographystyle{amsalpha}
\bibliography{kv}
\end{document}